\documentclass{article}
\usepackage[utf8]{inputenc}
\usepackage[T1]{fontenc}
\usepackage{graphicx} 

\usepackage{amsmath}
\usepackage{mathrsfs}
\usepackage{mathtools}
\usepackage{geometry}
\usepackage{tikz-cd}
\usepackage{relsize}
\usepackage[bbgreekl]{mathbbol}
\usepackage{amsfonts}
\DeclareSymbolFontAlphabet{\mathbb}{AMSb} 
\DeclareSymbolFontAlphabet{\mathbbl}{bbold}
\usepackage{bbm}
\usepackage{amssymb}
\usepackage{stmaryrd}

\usepackage{graphicx}
\usepackage{float}
\usepackage[style = alphabetic, backend = biber, maxnames = 99, minalphanames=3, maxalphanames=4]{biblatex}
\usepackage{rotating}
\makeatletter
\@ifundefined{c@abspage}{
}{
    \let\c@abspage\relax
}
\makeatother
\usepackage{enumitem}
\usepackage{amsthm, thmtools}
\usepackage{mdframed}
\usepackage[hidelinks]{hyperref}
\usepackage[capitalise]{cleveref}
\usepackage{xcolor}
\definecolor{linecolor}{gray}{0.70}

\newmdenv[
    linecolor=linecolor,
    leftline=true,
    rightline=false,
    topline=false,
    bottomline=false,
    linewidth=2pt,
    skipabove=\topsep,
    skipbelow=\topsep,
    innerleftmargin=10pt,
    innerrightmargin=10pt,
    innertopmargin=0pt,
    innerbottommargin=0pt,
    backgroundcolor=white
]{remarkbox}

\newtheorem{proposition}{Proposition}[section]
\newtheorem{theorem}[proposition]{Theorem}
\newtheorem{lemma}[proposition]{Lemma}

\newtheorem{corollary}[proposition]{Corollary}
\newtheorem*{theorem*}{Theorem}
\newtheorem*{proposition*}{Proposition}

\newtheorem{mainthm}{Theorem}

\newtheorem{maincor}[mainthm]{Corollary}
\crefname{maincor}{Corollary}{Corollaries}

\newcommand{\SpacingBeforeRemarkBox}{\vspace{1mm}}
\newcommand{\SpacingAfterRemarkBox}{\vspace{1mm}}
\newenvironment{remark}{\SpacingBeforeRemarkBox\refstepcounter{proposition}\begin{remarkbox}\noindent\textbf{Remark \theproposition.} }{\end{remarkbox}\SpacingAfterRemarkBox}

\newenvironment{example}{\SpacingBeforeRemarkBox\refstepcounter{proposition}\begin{remarkbox}\noindent\textbf{Example \theproposition.} }{\end{remarkbox} \SpacingAfterRemarkBox}
\newenvironment{notation}{\SpacingBeforeRemarkBox\refstepcounter{proposition}\begin{remarkbox}\noindent\textbf{Notation \theproposition.} }{\end{remarkbox}\SpacingAfterRemarkBox}

\newenvironment{definition}{\SpacingBeforeRemarkBox\begin{remarkbox}\noindent\textbf{Definition.} }{\end{remarkbox}\SpacingAfterRemarkBox}
\newenvironment{construction}{\SpacingBeforeRemarkBox\refstepcounter{proposition}\begin{remarkbox}\noindent\textbf{Construction \theproposition.} }{\end{remarkbox}\SpacingAfterRemarkBox}
\newenvironment{assumption}{\SpacingBeforeRemarkBox\refstepcounter{proposition}\begin{remarkbox}\noindent\textbf{Assumption \theproposition.} }{\end{remarkbox}\SpacingAfterRemarkBox}

\newcommand{\A}{\mathbb{A}}

\newcommand{\C}{\mathbb{C}}
\newcommand{\F}{\mathbb{F}}
\newcommand{\G}{\mathbb{G}}

\renewcommand{\L}{\mathbb{L}}
\newcommand{\N}{\mathbb{N}}
\renewcommand{\O}{\mathcal{O}}
\renewcommand{\P}{\mathbb{P}}
\newcommand{\Q}{\mathbb{Q}}
\newcommand{\R}{\mathrm{R}}

\renewcommand{\S}{\mathbb{S}}
\newcommand{\T}{\mathbb{T}}
\newcommand{\V}{\mathbb{V}}
\newcommand{\Z}{\mathbb{Z}}

\newcommand{\mc}[1]{\mathcal{#1}}
\newcommand{\mr}[1]{\mathrm{#1}}

\newcommand{\ms}[1]{\mathscr{#1}}
\renewcommand{\bar}[1]{\overline{#1}}
\renewcommand{\tilde}[1]{\widetilde{#1}}
\renewcommand{\hat}[1]{\widehat{#1}}

\newcommand{\llb}{\llbracket}
\newcommand{\rrb}{\rrbracket}

\newcommand{\<}{\left<}
\renewcommand{\>}{\right>}

\newcommand{\modc}[1]{\;\left(\mathrm{mod}\, #1\right)}

\newcommand{\inj}[1][]{\overset{#1}{\hookrightarrow}}
\newcommand{\sur}[1][]{\xrightarrow{#1}\mathrel{\mkern-14mu}\rightarrow}
\newcommand{\prism}{{\mathlarger{\mathbbl{\Delta}}}}
\newcommand{\heart}{\ensuremath\heartsuit}

\usepackage{etoolbox}

\newcommand{\defineobject}[2]{\expandafter#1\csname#2\endcsname{\mathrm{#2}}}
\forcsvlist{\defineobject{\newcommand}}{Ab, AbPreShv, AbShv, add, Alg, an, Ani, Aut, BK, Bun, CAff, CAlg, can, Cat, Ch, cobar, Coh, LCondAn, conj, Corr, CRing, crys, Crys, cts, cyc, CycSp, CycSyn, dbl, Det, Diff, Div, dR, End, ev, Ext, FF, Fil, FmlStk, Frac, Frob, FSch, Fun, Gal, gl, GL, gp, Gpd, gr, Grp, Gys, HC, HH, HKR, Hodge, Hom, HP, HT, id, Ind, inv, Isom, Iw, KL, ko, KO, Kos, ku, KU, LCA, lf, lin, Loc, Log, LProFSet, map, Map, MF, Mod, mot, Nilp, Nm, Nyg, op, ord, per, perf, Perf, perfd, Pic, Poly, Pro, proj, PreShv, qcoh, QCoh, qdR, qsyn, qSyn, QSyn, qrsp, rad, Rees, refl, Refl, Rep, rig, Sch, Sel, Set, Shv, Span, Spin, syn, Syn, SynSp, Tate, TC, THH, Thom, Top, tor, Tor, Tot, TP, tr, Tr, TR, triv, Vect, WCart}

\newcommand{\defineoperator}[2]{\expandafter#1\csname#2\endcsname{\operatorname{#2}}}
\forcsvlist{\defineoperator{\newcommand}}{cofib, coker, Cone, Exp, fib, im, mSpec, Proj, rank, Sp, Spa, Spd, Spec, Spev, Spf, supp, Sym, TSym, LSym}

\newcommand{\et}{\mr{\acute et}}

\newcommand{\proet}{\mathrm{pro\acute et}}

\DeclareMathOperator*{\colim}{colim}

\newcommand{\cartesian}[2]{\arrow["\lrcorner"{anchor=center, pos=0.25}, draw=none, from=#1, to=#2]}
\newcommand{\cocartesian}[2]{\arrow["\lrcorner"{anchor=center, pos=0.25}, draw=none, from=#2, to=#1]}

\title{Coleman Isomorphisms in Syntomic Cohomology and $\THH$}
\author{Kush Singhal}
\date{}

\begin{document}

\maketitle

\begin{abstract}
    This paper proves a generalisation of Coleman's isomorphism (between norm compatible cyclotomic units and a group of invertible power series) to various cohomology theories evaluated on proper regular $p$-adic formal schemes, for future applications to Iwasawa theory. This generalisation is an immediate consequence of a description, as a cyclotomic synthetic spectrum, of the limit over transfer maps as one goes up the cyclotomic tower of motivically filtered $\THH$. This crucially uses a calculation of $\THH(\Z_p[\zeta_{p^n}])$ due to Devalapurkar--Raksit as well as a calculation of the free loop transfer due to Schlichtkrull. Along the way, we construct transfer maps for various cohomology theories along finite locally free regular maps, using some elements of $\P^1$-stable motivic homotopy theory.
\end{abstract}
\setcounter{tocdepth}{2}
\tableofcontents

\section{Introduction}

Throughout, fix an odd prime $p$. 

\subsection{Statement of Main Results}

Consider the power series ring in 1-variable $\Z_p\llb q-1\rrb \simeq \varprojlim_n \Z_p[\Z/p^n]$. This is equipped with a ring endomorphism $\varphi$, the Frobenius lift, uniquely determined by $\varphi(q) = q^p$. The ring $\Z_p\llb q-1\rrb$ is also equipped with a $\Z_p^\times$-action via ring maps induced by multiplication on $\Z/p^n$. This article generalises and gives a more conceptual interpretation and generalisaiton of the following computation of Coleman \cite{coleman1979division}. 
\begin{theorem}[Coleman]\label{thm::coleman_iso}
    Let $N_\varphi : \Z_p\llb q-1\rrb \to \Z_p\llb q-1\rrb$ denote the unique map determined by \[(\varphi \circ N_\varphi) (f(q)) = \prod_{\zeta\in \mu_p} f(\zeta q).\] 
    Then, there is a $\Z_p^\times$-equivariant isomorphism of abelian groups 
    \[\varprojlim_n \Z_p[\zeta_{p^n}]^\times \cong \left\{f\in \Z_p\llb q-1\rrb^\times : N_\varphi(f) = f\right\}\]
    where the left hand side is the inverse limit over norm maps.
\end{theorem}
This result (and its generalisation to Galois cohomology of de Rham representations by Cherbonnier--Colmez) is one of the key local inputs in Iwasawa theory, in particular the theory of explicit reciprocity laws\footnote{For a comprehensive survey of explicit reciprocity laws and the role that Coleman's isomorphism and its generalisations play in this story, see the survey of Venjakob \cite{venjakob2025explicit} and the references therein.}. These explicit reciprocity laws are, for example, used in the construction of $p$-adic $L$-functions, which are the image of certain motivic classes under the Perrin-Riou $p$-adic regulator map. This regulator map is constructed using Coleman's isomorphism as well as some (rational) $p$-adic Hodge theory. Integral versions of this regulator map have been constructed in the literature, albeit with fairly severe restrictions on Hodge-Tate weights (see for instance \cite{lei2012coleman}). Integrality is important, since this gives control over the valuations of special values of the $p$-adic $L$-function of the motive and allows one to prove instances of the generalised Iwasawa main conjecture or the Bloch-Kato conjecture. A construction of the regulator map that is both integral and geometric\footnote{Many constructions of Euler systems require playing games with cohomology theories, and a geometric construction gives one better access to the underlying geometry of the motive.} would thus require understanding how various cohomology theories of integral $p$-adic Hodge theory, like the prismatic cohomology or syntomic cohomology of \cite{bms2, BS_prisms}, interpolate in the cyclotomic tower.

With a view towards studying (in future work) an integral version of Perrin-Riou's regulator map and its role in constructing $p$-adic $L$-functions using modern integral $p$-adic Hodge theory, we generalise Coleman's isomorphism to higher dimensional algebraic (formal) varieties, other cohomology theories like prismatic cohomology, as well as to higher motivic weights. We obtain such generalisations in one fell swoop by using homotopy-theoretic techniques as in \cite{bms2}. By \cite{bms2}, topological Hochschild homology $\THH$ and its motivic filtration $\Fil_\mot^{\ge *}\THH$ package together (Nygaard filtered) prismatic and syntomic cohomology. The following \cref{mainthm::desc_of_limit_of_THH_of_cyc} is the main result of this paper, with immediate consequences to (Nygaard-filtered) prismatic, syntomic, and (Hodge-filtered) de Rham cohomology listed in the following corollaries \ref{maincor::limit_of_prismatic_coh}-\ref{maincor::limit_of_de_rham_coh}. To keep track of motivic filtrations on $\THH$, we need to use the theory of `cyclotomic synthetic spectra' of \cite{antieau2024cyclotomic}; for unexplained notation see \cref{sctn::cyc_synth_sp}.

\begin{mainthm}[\cref{thm::limit_THH_for_cyclotomic_tower} and \cref{cor::when_we_bring_qdR_out_of_limit}]\label{mainthm::desc_of_limit_of_THH_of_cyc}
    Suppose $X$ is a proper regular $p$-adic formal scheme over $\Z_p[\zeta_p]$ and write $X_n := X\times_{\Spf \Z_p[\zeta_p]} \Spf \Z_p[\zeta_{p^n}]$. There is a functorial-in-$X$ equivalence of $\Fil_\mot^{\ge *}\THH(X)$-modules of cyclotomic synthetic spectra
    \begin{equation}\label{eqn::desc_of_lim_thh}
        \varprojlim_n \Fil_\mot^{\ge *}\THH(X_n)^\wedge_p \simeq \left(\Fil_\mot^{\ge *-1}\THH(X/\S_p\llb q-1\rrb)^\wedge_p \right) \; \hat\otimes_{\S_\ev\llb q-1\rrb^\Tate} \; F^{\ge *} \ms M_\Iw(1)[1]
    \end{equation}
    that is furthermore equivariant for the natural $(1+p\Z_p)^\times\subset \Z_p^\times$ via the Galois action on the left. Here:
    \begin{itemize}
        \item the limit on the left hand side is over transfer maps (see \cref{sctn::intro_trace_maps_in_coh_theories} below),
        \item $\S_\ev\llb q-1\rrb^\Tate$ is the cyclotomic spectrum whose underlying spectrum is the $E_\infty$-ring $$\S_\ev\llb q-1\rrb:= (\S_\ev[q])^\wedge_{(p,q-1)} \simeq \varprojlim_n \S_\ev[\Z/p^n]$$ with trivial circle action and with cyclotomic Frobenius given by the Tate diagonal \cite[Section III.1]{nikolaus_scholze}, and endowed with a $\Z_p^\times$-action coming from the usual multiplication action on $\Z/p^n$ as $n\to \infty$;
        \item $\Z_p[\zeta_p]$ is a $\S_p\llb q-1\rrb$-algebra via $q\mapsto \zeta_p$;
        \item $\THH(X/\S_p\llb q-1\rrb)^\wedge_p$ is equipped with a $\Z_p^\times$-action coming from $\S_p\llb q-1\rrb$;
        \item $F^{\ge *}\ms M_\Iw$ is the cyclotomic synthetic spectrum with underlying synthetic spectrum with $S^1$-action \[F^{\ge *} \ms M_\Iw := \varprojlim_\tau \S_\ev\llb q-1\rrb^\triv \]
        equipped with the trivial circle action where $$\tau : \S_\ev\llb q-1\rrb \to \S_\ev\llb q-1\rrb, \quad q^i \mapsto \begin{cases}
            0 & p\nmid i,\\
            q^{i/p} & p|i,
        \end{cases}$$ 
        and with the cyclotomic Frobenius given by the shift map 
        \[\varprojlim_\tau \S_\ev\llb q-1\rrb^\triv \to \left(\varprojlim_\tau \S_\ev\llb q-1\rrb^\triv\right)^{tC_{p,\ev}} \simeq \varprojlim_\tau \S_\ev\llb q-1\rrb^\triv, \quad (f_n) \mapsto (f_{n+1});\]
        \item $(1)$ denotes a Tate twist (for a more precise formulation see Notation \ref{notn::spectral_Tate_twist}); and
        \item the tensor product on the right is the $p$-complete tensor product of synthetic spectra.
    \end{itemize}

    In additional to the usual contravariant functoriality in $X$, this equivalence is also covariantly functorial in transfer maps for finite locally free regular maps, as well as Gysin maps for regular closed immersions. 
\end{mainthm}
We explain how this recovers Coleman's isomorphism in Remark \ref{rmk::recover_coleman_from_thm}.
\begin{remark}
    Readers familiar with Iwasawa theory may recognise the map \[ \Z_p\llb q-1\rrb \simeq \pi_0\S_p\llb q-1\rrb \xrightarrow{\pi_0\tau} \pi_0\S_p\llb q-1\rrb\simeq  \Z_p\llb q-1\rrb\]
    as being the normalised trace of Frobenius operator usually denoted by $\psi$ in the number theory literature. However, the symbol $\psi$ is also used to denote the Adams operations on $\KU_p$ which will be used frequently in \cref{sctn::trace_for_cyclotomic_tower}, and so we will use the symbol $\tau$ for the map defined above. We will abuse notation by also writing $\tau$ for the map $\Z_p\llb q-1\rrb\to \Z_p\llb q-1\rrb$ on $\pi_0$.
\end{remark}
\cref{mainthm::desc_of_limit_of_THH_of_cyc} has a lot of immediate consequences for other cohomology theories. To state these, we need to define the following $\Z_p\llb q-1\rrb$-module \[\mc M_\Iw := \pi_0 \ms M_\Iw \cong \varprojlim_\tau \Z_p\llb q-1\rrb.\]
It can be rewritten as \[\mc M_\Iw \cong \Hom_{\Z_p\llb q-1\rrb}(A_{\inf}(\Z_p^\cyc), \Z_p\llb q-1\rrb),\]
the module of (continuous) $\Z_p\llb q-1\rrb$-linear maps from $A_{\inf}(\Z_p^\cyc) \to \Z_p\llb q-1\rrb$, where $$A_{\inf}(\Z_p^\cyc) \cong \left(\colim_n \Z_p\llb q^{p^{-n}}-1\rrb\right)^\wedge_{p,q-1}$$ is Fontaine's period ring for the integral perfectoid ring $\Z_p^\cyc$.

\cref{mainthm::desc_of_limit_of_THH_of_cyc} has the following concrete application to the prismatic cohomology of \cite{bms2, BS_prisms, BL_absolute_prismatic_coh}. In the following, $\hat \prism_X$ is the completion with respect to the Nygaard filtration $\Fil_\Nyg^{\ge *}\hat \prism_X$ of the absolute prismatic cohomology of $X$, and $\hat \prism_{X/\Z_p\llb q-1\rrb}^{(1)}$ is the completion with respect to the Nygaard filtration $\Fil^{\ge r}_\Nyg \hat \prism_{X/\Z_p\llb q-1\rrb}^{(1)}$ of the Frobenius twisted prismatic cohomology of $X$ relative to the $q$-de Rham prism $\Z_p\llb q-1\rrb$. We will denote their $r$'th Breuil--Kisin twist by $\{r\}$. 

\begin{maincor}\label{maincor::limit_of_prismatic_coh}
    Suppose $X$ is a proper regular $p$-adic formal scheme over $\Z_p[\zeta_p]$, and write $X_n:= X\times_{\Spf\Z_p[\zeta_p]} \Spf \Z_p[\zeta_{p^n}]$. Then, there are $(1+p\Z_p)^\times$-equivariant equivalences\footnote{This statement should still hold true if one does not Nygaard complete both sides. 
    This should follow from a version of \cref{mainthm::desc_of_limit_of_THH_of_cyc} for the decompleted $\THH$ constructed in the upcoming work of Devalapurkar--Raksit--Hahn--Yuan \cite{DHRY_even_stacks}.}
    \begin{align*}
        \varprojlim_n \Fil_\Nyg^{\ge r} \hat \prism_{X_n}\{r\} &\simeq \Fil_\Nyg^{\ge r-1} \hat \prism_{X/\Z_p\llb q-1\rrb}^{(1)} \{r-1\} \hat\otimes_{\Z_p\llb q-1\rrb} \mc M_\Iw (1)[-1],\\
        \varprojlim_n \hat \prism_{X_n}\{r\} &\simeq \hat \prism_{X/\Z_p\llb q-1\rrb}^{(1)}\{r-1\}\hat\otimes_{\Z_p\llb q-1\rrb} \mc M_\Iw (1)[-1],
    \end{align*}
    where the limit on the left hand side is over prismatic trace maps (see \cref{sctn::intro_trace_maps_in_coh_theories} below), the right hand side of both these equivalences are $(p,q-1)$-adically completed, and $(1)$ is the Tate twist. 

    Moreover, in addition to the usual contravariant functoriality in $X$, these equivalences are covariantly functorial for Gysin maps associated to regular closed immersions and for transfers along finite locally free regular maps. 
\end{maincor}

Recall that both the absolute and relative prismatic cohomology are equipped with divided Frobenius maps 
\[\varphi_X\{r\} : (\Fil_\Nyg^{\ge r}\prism_X) \{r\} \to \prism_X \{r\} \quad \text{ and } \quad \varphi_X\{r\} : (\Fil^{\ge r}_\Nyg \prism_{X/\Z_p\llb q-1\rrb}^{(1)}) \{r\} \to \prism_{X/\Z_p\llb q-1\rrb} \{r\}.\]
The $r$'th integral syntomic cohomology $\Z_p(r)^\syn(X)$ of $X$ \cite{bms2, BL_absolute_prismatic_coh} is given by the fibre \[\Z_p(r)^\syn(X) := \fib\left((\Fil_\Nyg^{\ge r}\prism_X) \{r\} \xrightarrow{\can - \varphi\{r\}} \prism_X \{r\}\right) \simeq \fib\left((\Fil_\Nyg^{\ge r}\hat\prism_X) \{r\} \xrightarrow{\can - \varphi\{r\}} \hat \prism_X \{r\}\right)\]
where $\can : (\Fil_\Nyg^{\ge r}\prism_X) \{r\} \to \prism_X \{r\}$ is the canonical map from a filtration to the underlying object. 

\begin{maincor}\label{maincor::higher_coleman_iso}
    Suppose $X$ is a qcqs proper regular $p$-adic formal scheme over $\Z_p[\zeta_p]$. Write $X_n := X\times_{\Spf \Z_p[\zeta_p]} \Spf \Z_p[\zeta_{p^n}]$. Then, for any $r\in \Z$, there is a $(1+p\Z_p)^\times$-equivariant equivalence of objects in $\mc D(\Z_p)^\wedge_p$ 
    \begin{multline*}
        \varprojlim_n \Z_p(r)^\syn(X_n) \\\simeq 
        \fib\Bigg(\Fil^{r-1}_\Nyg \hat \prism_{X/\Z_p\llb q-1\rrb}^{(1)}\{r-1\} \hat \otimes_{\Z_p\llb q-1\rrb} \mc M_\Iw \xrightarrow{\can - \Phi_X\{r-1\} } \hat \prism_{X/\Z_p\llb q-1\rrb}^{(1)}\{r-1\} \hat \otimes_{\Z_p\llb q-1\rrb} \mc M_\Iw \Bigg)(1)[-1]
    \end{multline*}
    where the transition maps on the limit on the left hand side are given by the trace maps in syntomic cohomology (see \cref{sctn::intro_trace_maps_in_coh_theories} below), the tensor products on the right are $(p,q-1)$-adically completed, $(1)$ denotes the Tate twist, and where $\Phi_X\{r-1\}$ is the map given by the composition
\begin{align*}
    \Fil^{\ge r-1}_\Nyg \prism_{X/\Z_p\llb q-1\rrb}^{(1)}\{r-1\} \otimes_{\Z_p\llb q-1\rrb} \varprojlim_\tau \Z_p\llb q-1\rrb &\xrightarrow{\varphi_X\{r-1\} \otimes \mr{shift}} \prism_{X/\Z_p\llb q-1\rrb}\{r-1\} \otimes_{\Z_p\llb q-1\rrb} \varprojlim_\tau \Z_p\llb q-1\rrb \\&\longrightarrow \prism_{X/\Z_p\llb q-1\rrb}^{(1)}\{r-1\} \otimes_{\Z_p\llb q-1\rrb} \varprojlim_\tau \Z_p\llb q-1\rrb
\end{align*}
    with the second map induced by \[\id \otimes 1: \prism_{X/\Z_p\llb q-1\rrb} \to \prism_{X/\Z_p\llb q-1\rrb} \otimes_{\Z_p\llb q-1\rrb, \varphi} \Z_p\llb q-1\rrb =: \prism_{X/\Z_p\llb q-1\rrb}^{(1)}.\]

    This equivalence is moreover functorial in $X$ both contravariantly for the pull-back on cohomology, but also covariantly for Gysin maps associated to regular closed immersions and for transfers along finite locally free regular maps. 
\end{maincor}
\begin{remark}\label{rmk::recover_coleman_from_thm}
    Recall that $$\Z_p(1)^\syn(-) \simeq \R\Gamma_\et(-, \G_m)^\wedge_p [-1] .$$ Thus, Coleman's isomorphism is a computation of $\varprojlim_n \Z_p(1)^\syn(\Z_p[\zeta_{p^n}])$. One recovers Coleman's isomorphism from \cref{maincor::higher_coleman_iso} by taking $X=\Spf \Z_p[\zeta_p]$ and $r=1$, and using that \[\Fil^0_\Nyg \hat\prism_{\Z_p[\zeta_p]/\Z_p\llb q-1\rrb}^{(1)} \simeq \hat\prism_{\Z_p[\zeta_p]/\Z_p\llb q-1\rrb}^{(1)} \simeq \Z_p\llb q-1\rrb\] as well as the isomorphisms \[\fib\left(\varprojlim_\tau \Z_p\llb q-1\rrb \xrightarrow{\mr{shift}-\id} \varprojlim_\tau \Z_p\llb q-1\rrb\right) \simeq \fib(\Z_p\llb q-1\rrb \xrightarrow{\tau - \id} \Z_p\llb q-1\rrb) \xleftarrow[q \tfrac{\mr d}{\mr dq}\log]{\simeq} \left(\left(\Z_p\llb q-1\rrb^\times\right)^{N_\varphi = \id}\right)^\wedge_p.\]
\end{remark}

One can also prove a statement about $p$-complete Hodge completed (derived) de Rham cohomology $\hat\dR_X$, and its Hodge filtration $\Fil_H^{\ge r} \hat \dR_X$. This requires a little extra work, since one must go from $\THH$ to $\HH$. 

\begin{maincor}\label{maincor::limit_of_de_rham_coh}
    Suppose $X$ is a proper regular $p$-adic formal scheme over $\Z_p[\zeta_p]$, and write $X_n:= X\times_{\Spf\Z_p[\zeta_p]} \Spf \Z_p[\zeta_{p^n}]$. Then, there are $(1+p\Z_p)^\times$-equivariant equivalences
    \begin{align*}
        \varprojlim_n \Fil_H^{\ge r}\hat \dR_{X_n} &\simeq \Fil_H^{\ge r-1} \hat \dR_{X/\Z_p\llb q-1\rrb}^{(1)} \hat\otimes_{\Z_p\llb q-1\rrb} \mc M_\Iw (1)[-1],\\
        \varprojlim_n \hat \dR_{X_n} &\simeq \hat \dR_{X/\Z_p\llb q-1\rrb}^{(1)} \hat\otimes_{\Z_p\llb q-1\rrb} \mc M_\Iw (1)[-1],
    \end{align*}
    where the left hand side is the limit over trace maps in derived de Rham cohomology, the right hand side of all of these equivalences are $(p,q-1)$-adically completed, the $(-)^{(1)}$ denote Frobenius twists over $\Z_p\llb q-1\rrb$, and the remaining notation is as in \cref{maincor::limit_of_prismatic_coh}

    Moreover, in addition to the usual contravariant functoriality in $X$, these equivalences are covariantly functorial for Gysin maps associated to regular closed immersions and for transfers along finite locally free regular maps. 
\end{maincor}

\subsection{Transfer Maps}\label{sctn::intro_trace_maps_in_coh_theories}

A key technical input is the existence of trace maps in prismatic cohomology in the first place: indeed, the reason the trace maps (aka corestriction maps) \[\R\Gamma_\et(L, -) \to \R\Gamma_\et(K,-)\] exist in \'etale cohomology for finite field extensions $L/K$ is that $\Spec L \to \Spec K$ is a finite \'etale map. However, $\Spf \Z_p[\zeta_{p^n}] \to \Spf \Z_p [\zeta_{p^{n-1}}]$ is not \'etale, though it is finite locally free. 

Relatedly, on the homotopy theoretic side, we have glossed over a key technical detail: that the transfer maps, obtained by applying $\THH$ to the map \[\Big(\Perf(X) \xrightarrow{ f_*} \Perf(Y) \Big)\in \Cat_\infty^\perf\]
for $f$ a finite locally free map, preserve the motivic filtration. 

Both these issues are resolved by the following, which is one of the main technical results of this paper and may be of independent interest. In order to express the theorem in full generality, we use the theory of prismatic $F$-gauges, for which we refer the reader to \cite{Bhatt_FGaugeLec}. In the following, for a map $f:X\to Y$ of $p$-adic formal schemes we denote by $f^\syn : X^\Syn \to Y^\Syn$ the induced map on syntomifications. We will use $f^{\syn,*}$ and $f^\syn_*$ for the \emph{derived} pull-backs and pushforwards respectively. 

\begin{mainthm}\label{mainthm::prismatic_trace}
    Suppose $f:X\to Y$ is a finite locally free regular map of quasi-syntomic qcqs $p$-adic formal schemes. Then, for any perfect prismatic $F$-gauge $\mc E\in \mc \Perf(Y^\Syn)$, there is a `trace' map of prismatic $F$-gauges \[\tr_f^\syn(\mc E) : f_*^\syn (f^{\syn, *} \mc E) \to \mc E\]
    functorial in $\mc E$ as well as functorial in base-change in $f$ and composition in $f$ (the latter only up to non-canonical homotopy), satisfying the following properties:
    \begin{enumerate}
        \item \label{mainthm_part::normalisation} The composition \[\mc E \to  f_*^\syn (f^{\syn,*}\mc E) \xrightarrow{\tr_f^\syn(\mc E)} \mc E\]
        is multiplication by $\deg f$.

        \item \label{mainthm_part::projection_formula} Under the identification provided by the projection formula $f_*^\syn (f^{\syn, *} \mc E) \simeq f_*^\syn \O_{X^\Syn} \otimes \mc E$, one has $\tr_f^\syn(\mc E)\simeq \tr_f^\syn(\O_{X^\Syn}) \otimes \id_{\mc E}$. 
        
        \item \label{mainthm_part::etale_comparison} (\emph{\'Etale comparison}) For any $\mc E\in \Perf(Y^{\Syn})$, \'etale realisation of the trace map on $F$-gauges \[f_*^\syn f^{\syn,*} \mc E \to \mc E\]
        coincides with the \'etale trace map \[f_{\eta,*} f_\eta^* T_\et(\mc E) \to T_\et(\mc E)\]
        on $\Z_p$-local systems constructed in \cite{li_reinecke_zavyalov_trace}, where $T_\et$ denotes \'etale realisation.

        \item \label{mainthm_part::finite_etale_comparison} If $f$ is finite \'etale, then this coincides with the trace map for finite \'etale extensions discussed in \cite[Section 7.2]{guo_reinecke}.
        
        \item \label{mainthm_part::trace_on_perfectoids} When $X = \Spf S$ and $Y = \Spf R$ are both $p$-torsion free integral perfectoids with $S/R$ finite locally free, the trace map \[A_{\inf}(S) \to A_{\inf}(R)\]
        is induced by the usual trace map associated to the finite locally free extension $A_{\inf}(R)\to A_{\inf}(S)$.
    \end{enumerate}
\end{mainthm}
\begin{remark}
    In \cite{Bhatt_FGaugeLec}, \'etale realisation of an $F$-gauge is only defined for perfect $F$-gauges. However, $f_*^\syn \O_{X^\Syn}$ for $f$ a finite map that is not \'etale will not be perfect, and so \textit{a priori} \'etale realisation as defined in \cite{Bhatt_FGaugeLec} is not well-defined. We address this issue in Remark \ref{rmk::extending_etale_realisation}.
\end{remark}
\cref{mainthm::prismatic_trace} is conditional on the forthcoming work \cite{DHRY_even_stacks}. However, \cref{mainthm::desc_of_limit_of_THH_of_cyc} and its corollaries are not conditional on \cite{DHRY_even_stacks}, since \cref{mainthm::desc_of_limit_of_THH_of_cyc} only requires knowing that the transfer maps on $\THH$ preserve the motivic filtrations. The rest of \cref{mainthm::prismatic_trace} is simply a sanity check to ensure that the transfer maps induced on the associated graded of the motivic filtration satisfy expected properties.

The key point in the proof of \cref{mainthm::prismatic_trace} is that transfer maps exist for free for localising invariants, so all one needs to do is to check that the associated motivic filtrations are preserved by the transfer map; this is \cref{thm::transfer_map_for_finite_loc_free_extensions}. Following a suggestion of Clausen, we do this by factoring our finite locally free map $f:X\to Y$ through a regular closed immersion $X\inj \P^n_Y$ into projective space, and then using the Gysin maps of Tang \cite{tang2024syntomic, tang2026_gysin} as well as the projective bundle formula. Once we know that the motivic filtrations are preserved, the rest of \cref{mainthm::prismatic_trace} is then straightforward using standard techniques in the theory of prismatic cohomology.

\subsection{Idea Of Proof of {\cref{mainthm::desc_of_limit_of_THH_of_cyc}}}

In order to prove \cref{mainthm::desc_of_limit_of_THH_of_cyc}, we first give an explicit description of the filtered transfer map in motivically filtered $\THH$, in terms of a diagram of fibre sequences. One then takes limits in this diagram. One then needs to show that the limit of the middle term vanishes, which is a consequence of the fact that the $1+p^{n}\Z_p$ action on $\prism_{-/\Z_p\llb q_{n-1}-1\rrb}$ is trivial modulo $q-1$ (see \cite[Section 3.8]{BL_absolute_prismatic_coh} for the case $n=1$ and \cite[Theorem 1.0.6]{zeyuliu2025stacky} in general). This yields \cref{mainthm::desc_of_limit_of_THH_of_cyc} only at the level of synthetic spectra with synthetic circle action. As we elaborate below, the cyclotomic Frobenius requires some more work. 

In the following, we use the short hand $q_n := q^{p^{-n}}$, so that $\Z_p[\zeta_{p^n}]$ is a $\S_p\llb q_n-1\rrb$-algebra via $q_n\mapsto \zeta_{p^n}$. 
\begin{mainthm}[\cref{thm::computing_trace_maps_explicitly_for_any_R}]\label{mainthm::desc_of_transfers}
    Let $g$ be a topological generator of $(1+p\Z_p)$, and $R$ a $p$-quasi-syntomic $\Z_p[\zeta_p]$-algebra. Set $R_n := R\otimes_{\Z_p[\zeta_p]} \Z_p[\zeta_{p^n}]$. The transfer map $$\tr: \Fil_\mot^{\ge *} \THH(R_{n+1})\to \Fil_\mot^{\ge *} \THH(R_n)$$ sits in the following functorial-in-$R$ commutative diagram of $\Fil_\mot^{\ge *} \THH(R_n)$-modules in synthetic spectra with synthetic $S^1$-action:
    \[\begin{tikzcd}
        {\Fil_\mot^{\ge *}\THH(R_{n+1})} & {\Fil_\mot^{\ge *}\THH(R_{n+1}/\S_p\llb q_{n+1}-1\rrb)} && {\Fil_\mot^{\ge *-1}\THH(R_{n+1}/\S_p\llb q_{n+1}-1\rrb)(1)[2]} \\
        & {\Fil_\mot^{\ge *}\THH(R_{n}/\S_p\llb q_{n}-1\rrb)} && {\Fil_\mot^{\ge *-1}\THH(R_{n}/\S_p\llb q_{n}-1\rrb)(1)[2]} \\
        {\Fil_\mot^{\ge *}\THH(R_n)} & {\Fil_\mot^{\ge *}\THH(R_{n}/\S_p\llb q_{n}-1\rrb)} && {\Fil_\mot^{\ge *-1}\THH(R_{n}/\S_p\llb q_{n}-1\rrb)(1)[2]}
        \arrow[from=1-1, to=1-2]
        \arrow["\tr"', from=1-1, to=3-1]
        \arrow["{\psi^{g^{p^n}} - \id}", from=1-2, to=1-4]
        \arrow["{\rho_n}"', from=1-2, to=2-2]
        \arrow["{\rho_n(1)[2]}"', from=1-4, to=2-4]
        \arrow["{\psi^{g^{p^n}} - \id}", from=2-2, to=2-4]
        \arrow["{\sum_{i=0}^{p-1}\psi^{g^{ip^{n-1}}}}", from=2-2, to=3-2]
        \arrow["\id", equals, from=2-4, to=3-4]
        \arrow[from=3-1, to=3-2]
        \arrow["{\psi^{g^{p^{n-1}}}} - \id", from=3-2, to=3-4]
    \end{tikzcd}\]
    where top and bottom rows are fibre sequences, the left most map is the transfer map in filtered $\THH$, $\psi^{a}$ for $a\in \Z_p^\times$ are the Adams operations lifting the usual action of $\Z_p^\times$ on $\S_p\llb q_n-1\rrb$, and the map $\rho_n$ is given by\footnote{For a more precise definition of $\rho_n$ see Construction \ref{cons::construction_of_rho_n}.} 
    \begin{align*}
        \Fil_\mot^{\ge *}\THH(R_{n+1}/\S_p\llb q_{n+1}-1\rrb) &\simeq \Fil_\mot^{\ge *}\THH(R_{n}/\S_p\llb q_{n}-1\rrb) \otimes_{\S_p\llb q_n-1\rrb} \S_p\llb q_{n+1}-1\rrb \\
        &\simeq \bigoplus_{i=0}^{p-1} \Fil_\mot^{\ge *}\THH(R_{n}/\S_p\llb q_{n}-1\rrb) \cdot q_{n+1}^i\\
        &\xrightarrow{\id \otimes \proj_{i=0}}\Fil_\mot^{\ge *}\THH(R_{n}/\S_p\llb q_{n}-1\rrb).
    \end{align*}
\end{mainthm}
One can get formulas for the transfer maps in Nygaard filtered prismatic cohomology by taking associated graded of the motivic filtration on $\TC^-$. 
\begin{remark}\label{rmk::transfer_formula_not_commute_with_frob}
    We warn that this diagram is \emph{not} a commutative diagram in cyclotomic synthetic spectra. Indeed, the map labelled $\rho_n$ does not commute with cyclotomic Frobenius, as one checks upon applying $\pi_0(-)^{hS^1}$.
\end{remark}

\cref{mainthm::desc_of_transfers} is proven by a calculation using the explicit description of $\THH(\Z_p[\zeta_{p^n}])$ and $\THH(\Z_p[\zeta_p]/\S_p\llb q_1-1\rrb)$ as cyclotomic spectra due to Devalapurkar--Raksit (recalled in \cref{sctn::thh_of_cyclotomic_extns}) and the explicit calculation of the free loop transfer map $\THH(\S[\Z])\to \THH(\S[p\Z])$ due to Schlichtkrull (recalled in \cref{sctn::thh_of_grp_algebras}), followed by an analysis of motivic filtrations. Apart from \cref{mainthm::prismatic_trace}, this is the other place in the paper that the homotopy theoretic input is essential in proving Corollaries \ref{maincor::limit_of_prismatic_coh}--\ref{maincor::limit_of_de_rham_coh}, since we use in an essential way the structure of $\THH$ of spherical group algebras. Moreover, the reduction to the case of a point uses symmetric monoidality of $\THH$, whereas an analogous K\"unneth formula for absolute prismatic cohomology fails. 

\cref{mainthm::desc_of_transfers} suffices to give a description of $\varprojlim_n \Fil_\mot^{\ge *} \THH(X_n)$ as a synthetic spectrum with synthetic $S^1$-action. To get around the obstruction noted in Remark \ref{rmk::transfer_formula_not_commute_with_frob}, we make the observation that the Frobenius map $\varphi$ on $\Z_p\llb q_n-1\rrb$ factors as \[\Z_p\llb q_n-1\rrb \inj \Z_p\llb q_{n+1}-1\rrb \xrightarrow[q_{n+1} \mapsto q_n]{\simeq} \Z_p\llb q_n-1\rrb\]
where the first map is the canonical inclusion. Moreover, this second isomorphism is compatible with the projection map $\proj = \pi_0\rho_n^{hS^1}$:
\[\begin{tikzcd}
    \Z_p\llb q_{n+1}-1\rrb \ar[equals]{rr}{q_{n+1} \mapsto q_n}\ar{d}{\proj} && \Z_p\llb q_n-1\rrb \ar{d}{\proj} \\
    \Z_p\llb q_{n}-1\rrb\ar[equals]{rr}{q_{n} \mapsto q_{n-1}} && \Z_p\llb q_{n-1}-1\rrb.
\end{tikzcd}\]
One can calculate the cyclotomic Frobenius map in \cref{mainthm::desc_of_limit_of_THH_of_cyc} using a version of this algebraic observation for $\Fil^\bullet \THH(X_n / \S_p\llb q_{n-1}-1\rrb)$, along with a diagram chase. In fact, it is this observation that leads to the appearance of the shift map. Actually lifting this algebraic observation to spectra, and making the entire argument work, requires an incarnation of $\tau$ in terms of the $K(1)$-local ambidexterity of $BC_p$ (related to the idempotent cutting out the height 1 cyclotomic extension of \cite{CSY_chromatic_cyclotomic}), as well as an $S^1$-nilpotence calculation.

\subsection{Future Applications}
In \cite{CC_IwTh} (see also \cite{berger_exp_map, benois2000iwasawa, kkt, lei2012coleman}), Coleman's isomorphism and its generalisations to Galois cohomology are used to construct the Perrin-Riou $p$-adic regulator map, which interpolates between inverses of the Bloch--Kato exponential map and the dual exponential map. Recently, in \cite{flachkrausemorin_int_BK_exp}, a prismatic interpretation of the Bloch--Kato exponential map was studied. In future work, we will study the relation between the work carried out here and the work of \cite{flachkrausemorin_int_BK_exp}, in order to construct an integral version of Perrin-Riou's $p$-adic regulator. 

As previously mentioned, the Coleman isomorphism can be used to construct the Kubota--Leopoldt $p$-adic $L$-function (see \cite{coates_sujatha} for instance). However, it has already been observed in \cite{AndoHopkinsRezk} that the Kubota--Leopoldt $p$-adic $L$-function has a description as a certain endomorphism of $p$-complete topological real $K$-theory $\KO_p$ coming from the Atiyah--Bott--Shapiro spin orientation of $\KO$. In fact, the norm operator $N_\varphi$ in Coleman's isomorphism also shows up in Ando's criterion for an orientation to be $H_\infty$ \cite{ando1995isogenies}. One expects that a $K(1)$-local version of ($\TC$ applied to) \cref{mainthm::desc_of_limit_of_THH_of_cyc}, along with an analysis \textit{a la} \cite[Section 2]{sanath_arpon_image_of_j}, ties together these observations. This is also closely related to explicit class field theory for $\Q_p$ via the Lubin--Tate formal group $\hat \G_m$. From this perspective, one could even hope that Iwasawa theory for elliptic curves has some connection to orientation properties of the spectra of topological modular forms $\mr{tmf}$. We explore these ideas (at least for $\hat \G_m$) in current work in progress.

\subsection{Summary of Sections}\label{sctn::summary_of_sections}

\cref{sctn::preliminaries} gives some preliminaries on cyclotomic spectra, filtered localising invariants, and motivic filtrations that we will need. In \cref{sctn::transfer_maps_and_construction_of_trace}, we prove \cref{mainthm::prismatic_trace}: in \cref{sctn::transfer_map_for_finite_extension} we show that the transfer map on localising invariants respects the corresponding motivic filtrations, \cref{sctn::transfer_maps_on_F_gauges} proves the existence part of \cref{mainthm::prismatic_trace}, \cref{sctn::calculating_prismatic_trace_for_perfectoids} gives a calculation of the trace map for perfectoids, and finally \cref{sctn::etale_comparison_of_trace} establishes the comparison of our transfer maps with the \'etale traces of \cite{li_reinecke_zavyalov_trace}. 

The bulk of this paper is \cref{sctn::trace_for_cyclotomic_tower}; the plan for this section is as follows. In \cref{sctn::prelim_calculations}, we list a collection of technical lemmas that will be used in the following; the reader may safely skip this on a first reading and refer back to lemmas as necessary. In \cref{sctn::case_of_point}, we prove \cref{mainthm::desc_of_transfers} for the case $X = \Spf \Z_p$, using the results in \cref{sctn::thh_of_grp_algebras} and \cref{sctn::thh_of_cyclotomic_extns}. Using symmetric monoidality of $\THH$ and an analysis of its motivic filtration, we then obtain \cref{mainthm::desc_of_transfers} in general in \cref{sctn::computing_transfers_general_case}. With this description in hand, we then prove \cref{mainthm::desc_of_limit_of_THH_of_cyc} in \cref{sctn::lim_of_thh_in_cyc_tower}. Finally, in \cref{sctn::de_rham_coh} we prove \cref{maincor::limit_of_de_rham_coh}. 

\subsection*{Notation}
\addcontentsline{toc}{subsection}{Notation}

\begin{enumerate}
    \item $p$ is an odd prime number.
    
    \item Throughout, we will set $\Z_p^\cyc := \Z_p[\mu_{p^\infty}]^\wedge_p$ the integral subring of the $\Z_p^\times$-Galois perfectoid extension of $\Q_p$ containing all $p$-power roots of unity. We fix a compatible choice of primitive $p$-power roots of unity $\zeta_{p^n}$ (so $\zeta_{p^{n+1}}^p = \zeta_{p^n}$ and $\zeta_p\ne 1$) inside $\Z_p^\cyc$. We have the $(\Z/p^n\Z)^\times$-Galois extension $\Q_p(\zeta_{p^n})$ of $\Q_p$. We will denote the integral subrings of these fields by $\Z_p[\zeta_{p^n}]$. 

    \item For any $\infty$-category $\mc C$, write $\gr\mc C := \Fun(\Z^\delta, \mc C)$ for the category of graded objects of $\mc C$, and write $\Fil\mc C:= \Fun((\Z,\le ),\mc C)$ for the category of decreasingly filtered objects. For more on filtered objects see \cref{sctn::cyc_synth_sp}.

    \item For $G$ a topological group, we will write $\mc C^{BG} := \Fun(BG, \mc C)$ for the category of $\mc C$-valued local systems on $BG$, equivalently the category of objects of $\mc C$ with $G$-action, equivalently Borel $G$-objects in $\mc C$. Some more notation related to spectra with $G$-action is given in \cref{sctn::gp_actions_in_htpy_thy}.

    \item Write $\Sp$ for the symmetric monoidal category of spectra, with symmetric monoidal product written as $\otimes$ and unit the sphere spectrum $\S$. Write $\map(-, -)$ for the internal mapping spectrum in $\Sp$. One can, and should, think of $\Sp$ as the ``derived category $\mc D(\S)$ of $\S$-modules''. For $X\in \mc D(\Z)$, we will abuse notation by writing $X\in \Sp$ for the Eilenberg-MacLane spectrum associated to $X$.

    Throughout, $\CAlg$ will denote the category of $E_\infty$-ring spectra, and $\CAlg_R$ the category of $E_\infty$-$R$-algebras for $R$ an $E_\infty$-ring. We will also need to consider $\CAlg_\Z^\heart$, the 1-category of classical commutative $\Z$-algebras. 
    
    Write $(-)[1] : \Sp\to \Sp$ for the suspension functor in $\Sp$, which induces the canonical shift functor on the corresponding triangulated category. One has that \[\pi_n (X[1]) \simeq \pi_{n-1} X\]
    for all $X\in \Sp$. Cohomologically, we have $H^n(C[1]) = H^{n+1}(C)$ for $C$ a complex of abelian groups.

    Write $\tau_{\ge n}$ for the truncation functor on $\Sp$, so that \[\pi_i \tau_{\ge n}X := \begin{cases}
        0 & i<n,\\
        \pi_i X & i\ge n.
    \end{cases}\]

    \item Write $\Sp_p$ for the full-subcategory of $\Sp$ consisting of $p$-complete spectra. Unless otherwise specified, in \cref{sctn::trace_for_cyclotomic_tower} we will always be working in the $p$-complete category, and so we will suppress notation indicating $p$-completion. This is particularly relevant for all tensor products showing up in \cref{sctn::trace_for_cyclotomic_tower}: we use $\otimes$ instead of $\hat \otimes$ to denote the $p$-complete tensor product. Similarly, for the Tate construction for a group $G$ (see \cref{sctn::gp_actions_in_htpy_thy}), we write $(-)^{tG}$ to mean the $p$-completion of the usual Tate construction on $G$. 

    \item For any topological space $A$, we will write $\S_p[A]$ for the $p$-complete suspension spectrum $(\Sigma_+^\infty A)^\wedge_p$. As mentioned, we suppress the $p$-completion. 

    \item For any $E_\infty$-ring object $R$ in a symmetric monoidal stable $\infty$-category $\mc C$, write $\Mod_R(\mc C)$ for the category of $R$-module objects in $\mc C$. 

    \item We will write the $n$'th Frobenius twist of $\Z_p\llb q-1\rrb$ as $\Z_p\llb q_n-1\rrb$, where $q_n = q^{p^{-n}}$. Thus, the non-linear Frobenius map $\Z_p\llb q-1\rrb \to \Z_p\llb q-1\rrb$ can be viewed as a $\Z_p\llb q-1\rrb$-algebra map \[\Z_p\llb q-1\rrb \to \Z_p\llb q_1-1\rrb.\]
    Of course, Frobenius induces an isomorphism $\Z_p\llb q_{n+1}-1\rrb \xrightarrow{\simeq} \Z_p\llb q_n-1\rrb$ where $q_{n+1} \mapsto q_n$. 

    \item We reserve the notation $\varprojlim_n$ for specifically sequential limits, while denoting limits over any other indexing diagram with $\lim_i$. 
    
    \item Let $\KU_p$ denote the $p$-adic completion of topological complex $K$-theory. There is a continuous action of $\Z_p^\times$ via Adams operations on $\KU_p$, which will be denoted by $\psi^g : \KU_p \to \KU_p$ for $g\in \Z_p^\times$. We write $\ku_p := \tau_{\ge 0} \KU_p$. The Adams operations act on $\ku_p$, and we will also write $\psi^g$ for this. 
    
    \item For any module $M$ over $\ku_p$, write \[M(1)[2] := \tau_{\ge 2}\ku_p \otimes_{\ku_p} M.\] This notation makes sense since multiplication by the Bott class yields the equivalence $\tau_{\ge 2}\ku_p \simeq \ku_p[2]$ on underlying spectra, and the Adams operations act on the Bott class via multiplication.

    \item Recall that a map $R\to S$ of classical commutative $p$-complete rings is $p$-quasi-syntomic (resp. a $p$-quasi-syntomic cover) if $S/p$ is flat (resp. faithfully flat) over $R/p$ and $\L_{S/R}/p \in \mc D(S/p)$ has Tor amplitude in $[-1,0]$ (cf. \cite[Definition 4.10]{bms2}). A ring $R$ is $p$-quasi-syntomic if $\Z_p\to R$ is $p$-quasi-syntomic. As usual, we will abuse notation by only writing quasi-syntomic rather than $p$-quasi-syntomic.
    
    We write $\QSyn_{R}$ for the Grothendieck site of quasi-syntomic maps $R\to S$, where the covers are the quasi-syntomic covers. This site has a basis by the quasi-regular semi-perfectoids (see \cite[Section 4.4]{bms2}), which we will always abbreviate by QRSP.

    We also comment on the notion of regularity that we are assuming in \cref{mainthm::prismatic_trace}. We say a finite map $R\to S$ of classical commutative $p$-complete rings is regular if it is Koszul-regular in the sense of \cite[Tag068E]{stacks-project}, i.e. if it is a locally complete intersection. 

    \item For $f:X\to Y$ a map of $p$-adic formal schemes, we denote by $f^\syn : X^\Syn \to Y^\Syn$ the induced map on the syntomic stacks of \cite{Bhatt_FGaugeLec}. We will use $f^{\syn,*}$ and $f^\syn_*$ for the {derived} pull-backs and pushforwards respectively as we shall not need their underived versions. However, to prevent confusion, we still use the notation $\R\Gamma(-)$ to denote derived global sections. 
\end{enumerate}

\subsection*{Acknowledgements and Declaration on AI Use}
\addcontentsline{toc}{subsection}{Acknowledgements and Declaration on AI Use}
The author would like to thank their advisor Mark Kisin for invaluable advice and innumerable inputs. The author is grateful to Sanath Devalapurkar for helpful technical discussions about his work as well as his patience with many technical questions. Additionally, the author would also like to thank Keita Allen, Bhargav Bhatt, Jonathan Buchanan, Dustin Clausen, Oakley Edens, Jeremy Hahn, Dhilan Lahoti, Frank Lu, Akhil Mathew, Matthew Niemiro, Dylan Pentland, Lucas Piessevaux, Arpon Raksit, Maxime Ramzi, Florian Riedel, Tomer Schlank, Andy Senger, John Sim, Chris Skinner, Natalie Stewart, Fernando Trejos Suarez, and Longke Tang for \emph{many} helpful conversations in the course of this project. Finally, the author would like to thank Mark Kisin, Dylan Pentland, Florian Riedel, and Fernando Suarez for helpful comments on an early draft of this paper. 

ChatGPT 5.6 Sol was used during the writing of this paper to find typos, grammatical mistakes, and mathematical errors, as well as to occasionally trawl the literature for useful references. All of the proofs present in this paper are human-written and all references human-verified. The author takes full responsibility and ownership for the contents of this paper. 

\section{Preliminaries}\label{sctn::preliminaries}

\subsection{Cyclotomic Spectra and $\THH$}\label{sctn::prelim_cyc_spectr_and_thh}

In the following we give a very brief exposition of cyclotomic spectra and $\THH$, mostly in order to fix notation. We follow \cite{nikolaus_scholze}. 

\subsubsection{Group Actions in Homotopy Theory}\label{sctn::gp_actions_in_htpy_thy}

Suppose $G$ is a topological group, and $\mc C$ is a stable $\infty$-category. Consider a subgroup $H\subset G$, and let $\pi: BH\to BG$ and $\rho:BG\to B(G/H)$ be the canonical maps on classifying spaces induced by the inclusion $H\inj G$ and the projection $G\to G/H$. Pull-back $\pi^*: \mc C^{BG}\to \mc C^{BH}$ corresponds to the restriction functor, where an object with $G$-action is sent to the object with $H$-action obtained by forgetting the rest of the $G$-action. Similarly, $\rho^* : \mc C^{B(G/H)}\to \mc C^{BG}$ implements the functor sending an object with $G/H$-action to the object with $G$-action where $G$ acts via the quotient $G/H$. In particular, the pull-back functor $\rho^*$ when $H = G$ corresponds to endowing an object of $\mc C$ with trivial $G$-action, which we will denote by $(-)^\triv$. 

The functor $\rho^*$ has left and right adjoints, namely \[\rho_! : \mc C^{BG}\to \mc C^{B(G/H)} \quad \text{ and } \quad \rho_* : \mc C^{BG} \to \mc C^{B(G/H)}\]
obtained by taking the colimit and limit along $H$ respectively. These functors are often denoted by $\rho_! =: (-)_{hH}$ and $\rho_* =: (-)^{hH}$, and are referred to as the homotopy orbits and homotopy fixed points respectively. On homotopy groups, these functors give group homology and group cohomology respectively.

Now, if $G$ is a compact Lie group of dimension $d$, and $\mc C$ is stable, then there is a norm map \[\Nm_G : (-)_{hG}[d] \to (-)^{hG}\]
obtained informally by integrating along fibres. See \cite[Section I.1, I.4]{nikolaus_scholze} for a quick overview. 

\begin{definition}
    The cofibre of $\Nm_G$ is defined to be \emph{Tate cohomology $(-)^{tG}$}.
\end{definition}
If $G$ is finite and so $d = 0$, then the usual norm map $\Nm_G$ coincides with the map \[\sum_{g\in G} g : (-)_{hG}\to (-)^{hG}\]
from orbits to fixed points, and the resulting cofibre $(-)^{tG}$ coincides with the usual notion of Tate cohomology.

Since $(-)_{hG}$ and $(-)^{hG}$ are both exact functors of stable $\infty$-categories, so too is $(-)^{tG}$. However, not much more can be said about the commutation of $(-)^{tG}$ with limits or colimits in great generality. With connectivity conditions, we have the following. 

\begin{lemma}[{\cite[Lemma 2.11(b)]{antieau_nikolaus_2021cartier}}]\label{lem::orbits_and_inverse_lims_commute}
    Suppose that $F:I\to \Sp^{BS^1}$ is an $I$-diagram in spectra with $S^1$-action whose limit is $X:= \lim_i F(i)$, such that there exists $N\in \Z$ satisfying $F(i) \in \Sp_{\ge N}$ for all $i$. Suppose $I$ satisfies the following property: there exists $d\in \Z$ such that the functor $\lim_I : \Sp^I \to \Sp$ sends $\Sp^I_{\ge 0}$ to $\Sp_{\ge -d}$. Then, \[X_{hC_p} \simeq \lim_i (F(i)_{hC_p}) \quad \text{ and } \quad X^{tC_p} \simeq \lim_i (F(i)^{tC_p}).\]
    In particular, countable products and $\N$-indexed sequential limits in $\Sp^{BS^1}$ commute with $(-)_{hC_p}$ and $(-)^{tC_p}$.
\end{lemma}

\subsubsection{Cyclotomic Spectra}

\begin{definition}
    A \emph{($p$-typical) cyclotomic spectrum} is the following data:
    \begin{itemize}
          \item a spectrum $X$ with $S^1$-action, and
          \item a $S^1$-equivariant map \[\varphi_p : X\to X^{tC_p}.\]
    \end{itemize}
    The stable $\infty$-category of cyclotomic spectra is denoted $\CycSp$.
\end{definition}
\begin{construction}
    There is a functor \[(-)^\triv : \Sp \to \CycSp \]
    which sends a spectrum $X$ to the cyclotomic spectrum $X^\triv$ whose underlying spectrum is $X$ equipped with the trivial $S^1$-action, and whose cyclotomic Frobenius is the $S^1$-equivariant composition \[X\to X^{hC_p} \to X^{tC_p}.\]
    The first map is induced from the fact that $X$ has trivial circle action.
\end{construction}
Cyclotomic spectra have a symmetric monoidal structure as well, with unit $\S^\triv$. For a precise $\infty$-categorical construction of this structure, see \cite[Construction IV.2.1]{nikolaus_scholze}. 

\begin{construction}
    Suppose $X$ and $Y$ are cyclotomic spectra with cyclotomic Frobenii $\varphi_X : X\to X^{tC_p}$ and $\varphi_Y:Y\to Y^{tC_p}$. We define the tensor product of $X$ and $Y$ as a cyclotomic spectrum with underlying spectrum with $S^1$-action given simply by $X\otimes Y$ with the diagonal circle action, and with the cyclotomic Frobenius given by \[X\otimes Y \xrightarrow{\varphi_X \otimes \varphi_Y} X^{tC_p}\otimes Y^{tC_p} \to (X\otimes Y)^{tC_p}\]
    where the second arrow is from the lax symmetric monoidal structure on $(-)^{tC_p}$. We abuse notation and write $X\otimes Y$ viewed as a cyclotomic spectrum. 
\end{construction}
We will use the following fact implicitly.
\begin{lemma}
    The functor \[(-)^\triv : \Sp \to \CycSp\]
    is symmetric monoidal.
\end{lemma}

\begin{definition}
    For $X\in \CycSp$, write \[\TC(X) := \map_{\CycSp}(\S^\triv, X).\]
    This defines a functor \[\TC: \CycSp \to \Sp.\]
\end{definition}
Essentially by definition, one checks that $\TC$ is right adjoint to the functor $(-)^\triv : \Sp \to \CycSp$, and that for $X\in \CycSp$ \[\TC(X) \simeq \fib(X^{hS^1}\xrightarrow{\can - \varphi_X^{hS^1}} (X^{tC_p})^{hS^1}).\]

Here are some more examples of cyclotomic spectra. 
\begin{construction}\label{cons::shift_operator_on_cycsp}
    Suppose $X$ is a connective cyclotomic spectrum. We define a new cyclotomic spectrum $X^{(-1)}$ whose underlying $S^1$-spectrum is $\tau_{\ge 0}(X^{tC_p})$ (where $C_p\subset S^1$ is the subgroup of $p$'th roots of unity in $\C^\times$). There is a canonical $S^1$-equivariant map $\tilde \phi_X: X\to \tau_{\ge 0}(X^{tC_p})$ factoring the cyclotomic Frobenius $\phi_X: X\to X^{tC_p}$ of $X$ (since $X$ is connective). The cyclotomic Frobenius on $X^{(-1)}$ is then the composition \[\tau_{\ge 0} X^{tC_p} \to X^{tC_p} \xrightarrow{\tilde \phi_X^{tC_p}} (\tau_{\ge 0}X^{tC_p})^{tC_p}.\]
    One checks that the map $X\to \tau_{\ge 0} (X^{tC_p})$ upgrades to a map \[\phi_X^0 : X\to X^{(-1)}\]
    of cyclotomic spectra.

    The functor $(-)^{(-1)} : \CycSp_{\ge 0}\to \CycSp$ is lax symmetric monoidal since $\tau_{\ge 0}$ and $(-)^{tC_p}$ are too.
\end{construction}
If $X$ is an $S^1$-equivariant spectrum endowed with the trivial cyclotomic structure, then we write $X^{(-1)}$ for the cyclotomic spectrum $(X^\triv)^{(-1)} := \tau_{\ge 0} (X^\triv)^{tC_p}$.

\begin{construction}\label{cons::Tate_cyclotomic_structure}
    Suppose $A$ is an $E_\infty$-ring. Let $A^\Tate$ be the cyclotomic spectrum whose underlying spectrum with circle action is $A$ equipped with the trivial circle action, and whose cyclotomic Frobenius is given by the Tate-valued Frobenius $\varphi_A^\Tate : A\to A^\Tate$ defined in \cite[Section IV.1]{nikolaus_scholze}, given by \[A \xrightarrow{\Delta_p} (A\otimes \cdots \otimes A)^{tC_p} \xrightarrow{\mr{mult}} A^{tC_p}.\]
\end{construction}
One has $\S^\Tate \simeq \S^\triv$ (see \cite[Example IV.1.2(ii)]{nikolaus_scholze}). In general, the cyclotomic Frobenius on $A^\Tate$ can be computed via power operations on $A$.

We end with a technical lemma. In general, the forgetful functor $\CycSp \to \Sp^{BS^1}$ does not commute with arbitrary limits. However, in certain special cases, it does. 
\begin{lemma}\label{cor::limits_in_cycsp_computed_on_underlying}
    Suppose $F: I\to \CycSp_{\ge 0}$ is an $I$-diagram in connective cyclotomic spectra (i.e. the underlying spectrum is connective) where $I$ is a diagram as in \cref{lem::orbits_and_inverse_lims_commute}. Then, $\lim_I F$ exists in $\CycSp_{\ge 0}$, and the underlying spectrum with circle action of $\lim_I F$ is given by $\lim_i F(i) \in \Sp^{BS^1}$. 
\end{lemma}
\begin{proof}
    This is clear by \cref{lem::orbits_and_inverse_lims_commute}, since we can define the cyclotomic Frobenius $\varphi_p : \lim_I F \to (\lim_I F)^{tC_p}$ by the composition \[\lim_i F(i) \xrightarrow{\lim_i \varphi_{F(i)}} \lim_i F(i)^{tC_p} \simeq \left(\lim_i F(i)\right)^{tC_p}.\]
    That this satisfies the correct universal property is easy to see, using the fact that the forgetful functor $\Sp^{BS^1}\to \Sp$ commutes with all limits. 
\end{proof}

\subsubsection{$\THH$ and Related Functors}

\begin{definition}
    For an $E_\infty$-ring spectrum $A$, the \emph{topological Hochschild homology} $\THH(A)$ is a spectrum with $S^1$-action defined as the geometric realization of the cyclic bar construction
    \[\THH(A) \simeq A \otimes_{A \otimes A^{\op}} A.\]
\end{definition}
The cyclic bar construction may equivalently be expressed as the tensor of $A$ with the circle $S^1$ in the $\infty$-category of $E_\infty$-ring spectra:
\[\THH(A) \simeq A \otimes S^1.\]
This endows $\THH(A)$ with a natural action of the circle group $S^1$. For $A$ an $E_\infty$-ring, the definition of $\THH$ essentially yields the following universal property.
\begin{proposition}[{\cite[Proposition IV.2.2]{nikolaus_scholze}}]
    The (non-$S^1$-equivariant) map $A\to \THH(A)$ is initial among all maps $A\to R$ where $R$ is equipped with an $S^1$-action. 
\end{proposition}
The circle action on $\THH(A)$ extends to a canonical cyclotomic structure \[\varphi_A:\THH(A)\to \THH(A)^{tC_p}\] making $\THH(A)$ a ring object in cyclotomic spectra. The following is essentially \cite[Corollary IV.2.4]{nikolaus_scholze}.

\begin{lemma}
    There is a functorial map $\pi_A:\THH(A)\to A^\Tate$ of cyclotomic spectra that, on underlying spectra, is a retract of the canonical (non-$S^1$-equivariant) map $A\to \THH(A)$. 
\end{lemma}

\begin{proposition}
    $\THH : \CAlg \to \CycSp$ is symmetric monoidal, i.e. for $E_\infty$-rings $A,B,C$, we have \[\THH(B\otimes_A C) \simeq \THH(B)\otimes_{\THH(A)} \THH(C).\]
\end{proposition}
\begin{corollary}
    $\displaystyle \THH(A\otimes_\S B) \simeq \THH(A)\otimes_{\S^\triv} \THH(B).$
\end{corollary}

One can also define relative $\THH$, at least as spectra with $S^1$-action.
\begin{definition}
    For $R\to S$ a map of $E_\infty$-rings, define the $E_\infty$-$R$-algebra with $S^1$-action given by \[\THH(S/R) := S\otimes_{S\otimes_R S}S.\]
\end{definition}
In general, one has an equivalence of spectra with $S^1$-action \[\THH(S/R) \simeq \THH(S) \otimes_{\THH(R)} R^\triv.\]
If $R$ is a $\Z$-algebra, one often writes $\HH(S/R) := \THH(S/R)$. If $R = \Z$, then we simply write $\HH(S):= \THH(S/\Z)$. 

The cyclotomic structure is much more subtle; see \cite{blumberg2024relativecyclotomicstructuresequivariant} for more. However, we will be interested only in the case $R = \S_p\llb q-1\rrb$, where the latter is defined as follows. 
\begin{notation}
    Write $\S\llb x-1\rrb$ for the $(x-1)$-adic completion of $\S[x] := \S[\N]$, i.e. \[\S\llb x-1\rrb \simeq \varprojlim_n \cofib((x-1)^n : \S[x] \to \S[x]).\]
    Upon $p$-completion, we equivalently have \[\S_p\llb x-1\rrb \simeq \varprojlim_n \S_p[\Z/p^n].\]
    Thus, we can think of $\S_p\llb x-1\rrb$ as equivalently the completed group algebra $\S\llb \Z_p\rrb$. 
\end{notation}
\begin{definition}
    The cyclotomic structure on $\THH(S/\S\llb x-1\rrb)$ is given by the following tensor product in $\CycSp$ \[\THH(S/\S\llb x-1\rrb) := \THH(S) \otimes_{\THH(\S\llb x-1\rrb)} \S\llb x-1\rrb^\Tate.\]
\end{definition}
\begin{remark}
    Under $p$-completion, and for $S$ a $\Z_p[\zeta_p]$-algebra, there is an equivalence \[\THH(S/\S_p\llb x-1\rrb)^\wedge_p \simeq \THH(S/\S_p[x^{\pm 1}])^\wedge_p. \]
    The base reason is that in $\Sp_p$, $\S[B\Z_p] \simeq \S[B\Z] \simeq \S[S^1]$. 
\end{remark}

We also have other functors built out of $\THH$. 
\begin{definition}
    Suppose $R$ is an $E_\infty$-ring. Then, we have the functors $\CAlg \to \Sp$ given by 
    \[\TC^-(R) := \THH(R)^{hS^1} \quad \text{ and } \TP(R) := \THH(R)^{tS^1}.\]
\end{definition}
There are two maps $\TC^-(R)\to \TP(R)^\wedge_p$ when $R$ is connective. 
\begin{enumerate}
    \item For any spectrum with $S^1$-action, there is always a canonical map $(-)^{hS^1}\to (-)^{tS^1}$. This induces a map \[\can : \TC^-(R)\to \TP(R)^\wedge_p.\]
    \item We have the cyclotomic Frobenius $\varphi_{\THH(R)}: \THH(R) \to \THH(R)^{tC_p}$. Taking $(-)^{hS^1}$ yields the map \[\varphi^{hS^1}: \TC^-(R) \to (\THH(R)^{tC_p})^{hS^1} \simeq \TP(R)^\wedge_p\]
    where we use \cite[Lemma II.4.2]{nikolaus_scholze} for the last equivalence.
\end{enumerate}
One then has \[\TC(X) \simeq \fib(\TC^-(R) \xrightarrow{\can - \varphi} \TP(R)^\wedge_p).\]

\subsubsection{Decompleted Cyclotomic Spectra}\label{sctn::decompleted_cyclotomic_spectra}

We would like to get Nygaard decompleted $\TC^-$ and $\TP$. An \textit{ad hoc} approach is given in \cite{manam2024drinfeldformalgroup}. However, the following construction that will appear in forthcoming work \cite{DHRY_even_stacks}; an exposition without proofs is given in the thesis \cite{sanath_thesis}. Throughout, $C_p$ denotes the finite cyclic group of order $p$.

\begin{definition}
    A \emph{decompleted cyclotomic spectrum} $(X, \Phi X)$ is the data of 
    \begin{enumerate}
        \item a cyclotomic spectrum $X$ in the sense of \cite{nikolaus_scholze} (i.e. a spectrum $X$ with $S^1$-action along with an $S^1$-equivariant map $\varphi : X\to X^{tC_p}$),
        \item another spectrum $\Phi X$ with $S^1$-action, and
        \item maps $X\to \Phi X$ and $\Phi X\to X^{tC_p}$ that factor $\varphi$. 
    \end{enumerate}
    The cyclotomic spectrum $X$ will be referred to as the \emph{underlying cyclotomic spectrum} and the data of $\Phi X$ with the maps $X\to \Phi X$ and $\Phi X\to X^{tC_p}$ will be referred to as a \emph{decompleted structure} on $X$.

    The category of decompleted cyclotomic spectra will be denoted $\CycSp_\prism$. 
\end{definition}

Of course, there is a fully faithful functor \[\beta : \CycSp\to \CycSp_\prism\]
which endows a cyclotomic spectrum $X$ with the decompleted structure $\Phi X:= X^{tC_p}$ with the maps $\varphi : X\to \Phi X$ and $\id: \Phi X\to X^{tC_p}$. 

\begin{definition} 
    (cf. {\cite[Example 7.1.3]{sanath_thesis}})
    Let $R$ be a $\Z_p$-algebra. Denote by $\THH_\prism(R)^\wedge_p$ the decompleted cyclotomic $E_\infty$-ring whose underlying cyclotomic spectrum is $\THH(R)^\wedge_p$, and whose decompleted structure is given by declaring \[\Phi \THH_\prism(R)^\wedge_p := \left(\THH(R) \otimes_{\THH(\Z_p)} \THH(\Z_p)^{tC_p}\right)^\wedge_p.\]
    The map $\THH(R)^\wedge_p\to \Phi\THH_\prism(R)^\wedge_p$ is obtained by tensoring the unit map $\S_p\to \left(\THH(\Z_p)^{tC_p}\right)^\wedge_p$, while the map $\Phi\THH_\prism(R)^\wedge_p\to \left(\THH(R)^{tC_p}\right)^\wedge_p$ is induced by the commuting square \[\begin{tikzcd}
        \THH(\Z_p)^\wedge_p \ar{r}\ar{d} & \left(\THH(\Z_p)^{tC_p}\right)^\wedge_p \ar{d} \\
        \THH(R)^\wedge_p \ar{r} & \left(\THH(R)^{tC_p}\right)^\wedge_p
    \end{tikzcd}\]
    coming from functoriality of the cyclotomic Frobenius. 
\end{definition}
\begin{example}\label{eg::THH_prism(Z_p)}
    $\displaystyle \THH_\prism(\Z_p)^\wedge_p = \beta \THH(\Z_p)^\wedge_p$.
\end{example}

For $(X,\Phi X)$ a decompleted cyclotomic spectrum, we write $X^{C_p}$ for the ring spectrum with $S^1$-action defined by the following pull-back square
\[\begin{tikzcd}
    X^{C_p} \ar{r}\ar{d} & X^{h C_p} \ar{d}{\can} \\
    \Phi X \ar{r} & X^{ tC_p}.
	\arrow["\lrcorner"{anchor=center, pos=0.125}, draw=none, from=1-1, to=2-2]
\end{tikzcd}\]
In particular, the projection map yields a canonical map \[\can: (X^{C_p})^{hS^1} \to (\Phi X)^{hS^1}.\]
There is also a Frobenius map $\varphi:(X^{C_p})^{hS^1}\to (\Phi X)^{hS^1}$ defined as the composition \[(X^{C_p})^{hS^1} \to (X^{hC_p})^{hS^1} \simeq X^{hS^1} \to (\Phi X)^{hS^1}.\]
This sits in a Cartesian square 
\[\begin{tikzcd}
    {(X^{C_p})^{hS^1}} \ar{r}\ar{d}{\varphi} & {(X^{hC_p})^{hS^1} \simeq X^{hS^1}} \ar{d}{\varphi^{hS^1}} \\
    {(\Phi X)^{hS^1}} \ar{r} & {(X^{tC_p})^{hS^1}}.
    \cartesian{1-1}{2-2}
\end{tikzcd}\]

The decompleted cyclotomic spectrum $\THH_\prism(R)$ is related to prismatic cohomology by the following.
\begin{proposition}[{\cite{DHRY_even_stacks}}]
    Suppose $R$ is a QRSP. Then, 
    \begin{align*}
        \pi_{2n}\left(\left(\THH(R)^{C_p}\right)^{hS^1} \right) &\simeq (\Fil_\Nyg^n\prism_R)\{n\} \\
        \pi_{2n}\left(\left(\Phi\THH(R)\right)^{hS^1} \right) &\simeq \prism_R\{n\} 
    \end{align*}
    for any $n$. All the odd homotopy groups of the above spectra vanish.

    Moreover, the canonical map $\can:(\THH(R)^{C_p})^{hS^1}\to (\Phi \THH(R))^{hS^1}$ induces on homotopy groups the canonical map \[(\Fil_\Nyg^n \prism_R)\{n\} \inj \prism_R\{n\},\]
    and the Frobenius map $\varphi:(\THH(R)^{C_p})^{hS^1}\to (\Phi \THH(R))^{hS^1}$ induces on homotopy groups the divided Frobenius map \[\varphi_n :(\Fil_\Nyg^n \prism_R)\{n\} \to \prism_R\{n\}. \]
\end{proposition}
\begin{corollary}
    The even filtration (in the sense of \cite{HahnRaksitWilson_motivic_fil_on_TC}) on $\left(\THH(R)^{C_p}\right)^{hS^1}$ (resp. $\left(\Phi\THH(R)\right)^{hS^1}$)
    has associated graded given by $(\Fil_\Nyg^n\prism_R)\{n\}[2n]$ (resp. $\prism_R\{n\}[2n]$).
\end{corollary}

Thus, one can think of $$\TC^-_\prism(R):= \left(\THH(R)^{C_p}\right)^{hS^1}$$ as being Nygaard decompleted $\TC^-(R)$ and $$\TP_\prism(R):= \left(\Phi\THH(R)\right)^{hS^1}$$ as being Nygaard decompleted $\TP(R)$. Moreover, the Cartesian square defining $\THH(R)^{C_p}$ yields the following pull-back square 
\[\begin{tikzcd}
    \Fil^i_\Nyg\prism_R\{n\} \ar[hook]{r}\ar[hook]{d} & \Fil^i_\Nyg\hat\prism_R\{n\} \ar[hook]{d} \\
    \prism_R\{n\} \ar[hook]{r} & \hat\prism_R\{n\}
	\arrow["\lrcorner"{anchor=center, pos=0.125}, draw=none, from=1-1, to=2-2]
\end{tikzcd}\]
for all $n\ge 0$.

\begin{remark}
    Since $\THH$ satisfies quasi-syntomic descent, we can globalize the above constructions to define a sheaf in decompleted cyclotomic spectra $\THH_\prism$ on the quasi-syntomic site of $X$. In particular, we have the quasi-syntomic sheaves $\TC^-_\prism$ and $\TP_\prism$ whose even filtrations recover Nygaard filtered prismatic cohomology and usual prismatic cohomology of $X$ respectively. 
\end{remark}

Finally, we can define the motivic filtration on $\Phi\THH, \TC^-_\prism$, and $\TP_\prism$ as in \cite{bms2}.
\begin{definition}
    For $\mc F\in \{\Phi\THH_\prism, \TC^-_\prism, \TP_\prism\}$, we define the motivic filtration by \[\Fil^i \mc F(X):= \R\Gamma(\QSyn_X, \tau_{\ge 2i} \mc F(-)).\]
\end{definition}
Since QRSPs form a basis for $\QSyn$, we see that \[\Fil^i \mc F(X) \simeq \varprojlim_{R \text{ QRSP}, \;\Spf R\to X} \tau_{\ge 2i} \mc F(R). \]

\begin{remark}
    In \cite{HahnRaksitWilson_motivic_fil_on_TC}, the authors defined the even filtration $\Fil_\ev^\bullet$ on $\THH, \TP, \TC^-$, and $\TC$. Moreover, they showed that for any quasi-syntomic ring $R$, the even filtration coincides with the BMS filtration. It formally follows that the even filtration coincides with the above defined motivic filtration on affine $p$-adic formal schemes, for $\mc F\in \{\Phi\THH, \TC^-_\prism, \TP_\prism\}$. 

    However, the even filtration {does not} coincide with the BMS/motivic filtration on non-affine inputs. Since we will need to care about non-affine $p$-adic formal schemes, we thus use the BMS filtration rather than the even filtration.
\end{remark}

\subsection{A Zoo of Cohomology Theories}\label{sctn::zoo_of_coh_thries}

Many cohomology theories, and certainly all cohomology theories that we will consider, are in fact the associated graded of a multiplicative `motivic' filtration on a certain nice Zariski sheaf $\mc F : \Sch^\op \to \CAlg(\Sp)$ (resp. $\FSch^\op \to \CAlg(\Sp_p)$ for $\FSch$ the category of $p$-adic formal schemes)
valued in $E_\infty$-rings in spectra (resp. $p$-complete spectra). These sheaves are usually localising invariants.
\begin{definition}
    Let $\Cat_\infty^\perf$ denote the category of idempotent-complete small stable $\infty$-categories, and maps are exact functors (i.e. those that preserve finite limits and colimits). Note that $\Perf(X)\in \Cat_\infty^\perf$ for any ($p$-adic formal) scheme $X$.

    A sequence $$\mc A\xrightarrow{f} \mc B\xrightarrow{g} \mc C$$ in $\Cat_\infty^\perf$ is exact if the composite is zero, the functor $f$ is fully faithful, and the map $\mc B/\mc A\to \mc C$ is an equivalence.

    A \emph{localising invariant} is a functor $\mc F : \Cat_\infty^{\perf} \to \mc D$ valued in a stable $\infty$-category $\mc D$ that inverts Morita equivalences and sends any exact sequence of categories to a (co)fibre sequence in $\mc D$. 
\end{definition}
We will write $\mc F(X) := \mc F(\Perf(X))$ for any functor $\mc F:\Cat_\infty^\perf\to \mc D$.

\begin{lemma}
    Each of $\THH, \Phi\THH, \TP, \TP_\prism, \TC^-, \TC^-_\prism, \TC, \HH, \HP,\HC^-, L_{K(1)}\TC$ are all localising invariants. 
\end{lemma}
\begin{proof}
    For $\THH, \TP, \TC^-,$ and $\TC$, this is well known: for $\THH$ this is \cite[Proposition 10.2]{blumberg2013universal} (see also \cite{hesselholt2020topological}). In fact, $\THH$ refines to a localising invariant $\THH: \Cat_\infty^\perf \to \CycSp$, and it thus follows for $\TC,\TP,\TC^-$ as each of these are exact functors $\CycSp\to \Sp$. That $\Phi\THH$ is a localising invariant is \cite[Theorem 1.9]{mao2024equivariant_decompleted_TP}. The claim then follows for $$\TP_\prism = (\Phi \THH)^{hS^1},$$ and thus for $\TC^-_\prism := \TP_\prism \times_{\TP} \TC^-$. Throughout, we are using that finite colimits in a stable $\infty$-category are finite limits and thus commute with arbitrary limits. Since $K(1)$-localisation also commutes with finite (co)limits, we see that $L_{K(1)}K$ is also an localising invariant. Finally, since $\HH(-) \simeq \THH(-) \otimes_{\THH(\Z_p)}\Z_p$, the claim follows for $\HH(-)$ from the case of $\HH$. 
\end{proof}
By \cite[Theorem 1.12]{blumberg2013universal} we know that for any localising invariant $\mc F$, we have \[\Map(K,\mc F)\simeq \mc F(\S)\]
where $K$ is (non-connective) algebraic $K$-theory and $\S$ is the sphere spectrum. Suppose moreover that $\mc F$ is a lax symmetric monoidal functor into $\Sp$; then the unit map for the $E_\infty$-ring $\mc F(\S)$ yields a canonical (up to homotopy) choice of natural transformation of lax symmetric monoidal functors $T: K\to \mc F$. If $\mc F= \THH$ (resp. $\mc F=\TC$), then $T$ is the Dennis trace (resp. the cyclotomic trace) \cite[Section 10]{blumberg2013universal}.

We now list some motivic filtrations on various localising invariants, where throughout the input $X$ is a $p$-quasi-syntomic $p$-adic formal scheme and the output is implicitly $p$-completed. 
\begin{enumerate}
    \item For any discrete commutative ring $k$, there is a motivic filtration on $\HH(X/k)$, $\HP(X/k)$, and $\HC^-(X/k)$, whose associated graded yields, respectively, $p$-complete derived Hodge cohomology $L\Omega_{X/k}^*[2*]$, $p$-complete Hodge-completed derived de Rham cohomology $\hat \dR_{X/k}[2*]$, and the $p$-complete Hodge filtration $\hat \dR^{\ge *}_{X/k}[2*]$ thereon.
    For more on this filtration see \cite{bms2, antieau2019periodic, raksit2020hochschildhomologyderivedrham, HahnRaksitWilson_motivic_fil_on_TC}.
    
    \item There is a motivic filtration on $\TC(X)$, $\TP(X)$, $\TC^-(X)$, and $\THH(X)$, whose associated graded is respectively syntomic cohomology $\Z_p^\syn(*)(X)[2*]$, Breuil--Kisin twisted Nygaard completed prismatic cohomology $\hat \prism_X\{*\}[2*]$, the Nygaard filtration $(\Fil_\Nyg^{\ge *} \hat\prism_X) \{*\}[2*]$ thereon, and its associated graded $\gr_\Nyg^*\hat \prism_X \{*\}[2*]$. For more on this filtration see \cite{bms2}, \cite{BS_prisms}, \cite{HahnRaksitWilson_motivic_fil_on_TC}.
    
    \item As described in the previous section, there are motivic filtrations on the {decompleted} variants of the above localising invariants $\TP_\prism, \TC^-_\prism$, and $\THH_\prism$ \cite{DHRY_even_stacks}. The associated motivic graded identifies respectively with Breuil--Kisin twisted prismatic cohomology $\prism_X\{*\}[2*]$, the Nygaard filtration thereon $\Fil_\Nyg^{\ge *} \prism_X \{*\}[2*]$, and its associated graded $\gr_\Nyg^* \prism_X \{*\}[2*]$.
    
    \item Endow $\Z_p[\zeta_p]$ with a $\S_p\llb q-1\rrb$-algebra structure via the map $q\mapsto \zeta_p$. For $X$ a $p$-adic formal scheme over $\Z_p[\zeta_p]$, there is a motivic filtration on $\TC(X/\S_p\llb q-1\rrb)$, $\TP(X/\S_p\llb q-1\rrb)$, $\TC^-(X/\S_p\llb q-1\rrb)$, and $\THH(X/\S_p\llb q-1\rrb)$; see \cite[Appendix A]{wagner2025q} and \cite{HahnRaksitWilson_motivic_fil_on_TC}. This motivic filtration agrees with the Hahn--Raksit--Wilson even filtration on affines. The associated graded of this motivic filtration identifies with relative syntomic cohomology $\Z_p^\syn(*)(X/\Z_p\llb q-1\rrb)[2*]$, Frobenius twisted Nygaard completed relative prismatic cohomology $\hat \prism^{(1)}_{X_1/\Z_p\llb q-1\rrb}\{*\}[2*]$, the Nygaard filtration thereon $\Fil_\Nyg^{\ge *}\hat \prism^{(1)}_{X_1/\Z_p\llb q-1\rrb}\{*\}[2*]$, and its associated graded $\gr_\Nyg^{*}\hat \prism^{(1)}_{X_1/\Z_p\llb q-1\rrb}\{*\}[2*]$.

    \item The localising invariant $L_{K(1)}K(-)$, when viewed as a functor from $\Z[\tfrac 1p]$-algebras to ($K(1)$-local) spectra, has the Thomason filtration given by the \'etale sheafification of $\tau_{\ge 2*} L_{K(1)}K(-)$, whose associated graded is \[\R\Gamma_\et(-, \Z_p(*))[2*].\]
    This filtration is moreover exhaustive, but not necessarily complete. 
    For any ring $R$, one has $$L_{K(1)}\TC(R^\wedge_p)\simeq L_{K(1)}\TC(R) \simeq L_{K(1)}K(R) \simeq L_{K(1)}K(R[\tfrac 1p])$$ by \cite{bcm_rmks_on_K1}, and so one can view the functor $\Fil_\mot^\bullet L_{K(1)}\TC$ as being a functor from arbitrary commutative $\Z$-algebras to filtered spectra, with associated graded given by 
    \[\R\Gamma_\et((-)^\wedge_p[\tfrac 1p], \Z_p(*))[2*]\]
    the \'etale cohomology of the rigid generic fibre. In \cite{kim2025thomason}, it is shown that the Thomason filtration on $L_{K(1)}K(R[\tfrac 1p]) \simeq L_{K(1)}\TC(R)$ is compatible with the motivic filtration on $\TC$.
\end{enumerate}

All of the usual comparison maps between these cohomology theories are realised by maps at the level of the corresponding filtered localising invariant. Moreover, all of the above filtered localising invariants except for filtered $K(1)$-local $K$-theory satisfy quasi-syntomic descent \cite{bms2} and are all quasi-syntomic locally even. 

All of the above filtrations, at least on affines, are obtained in two ways:
\begin{enumerate}
    \item Via descent where, on a basis of a certain topology, the motivic filtration is the double speed Postnikov filtration (using that the value of the localising invariant at such a basis is always even). 
    \item Variants of the even filtration as in \cite{HahnRaksitWilson_motivic_fil_on_TC}. 
\end{enumerate}
Of course, both of these coincide in many cases of interest.

\subsection{Cyclotomic Synthetic Spectra}\label{sctn::cyc_synth_sp}
In order to keep track of motivic filtrations on $\THH$, we need to use synthetic spectra (introduced in \cite{pstrkagowski2023synthetic}) as well as cyclotomic synthetic spectra from \cite{antieau2024cyclotomic}. We will follow the latter reference. This also gives us a chance to recall some fact about filtrations. 

\begin{definition}
    For a stable category $\mc C$, write $$\Fil \mc C := \Fun(\Z^\op, \mc C)$$ for the $\infty$-category of filtered objects, where $\Z^\op$ is the poset of decreasing integers. 
\end{definition}
We will write an object of $\Fil \mc C$ as $F^{\ge *}M$. Evaluation at $i\in \Z$ yields a limit and colimit preserving functor \[F^{\ge k} : \Fil \mc C \to \mc C, \quad F^{\ge *}M \mapsto F^{\ge k} M.\]
For each $k\in \Z$, there is also a shifting functor  \[-\<k\> : \Fil\mc C\to \Fil\mc C, \quad F^{\ge *}(M\<k\>) := F^{*-k}M.\]
If $\mc C$ is symmetric monoidal with tensor product $\otimes_{\mc C}$, then so is $\Fil \mc C$ via Day convolution, so that \[F^{\ge k}(F^{\ge *}M\otimes_{\Fil \mc C} F^{\ge *}N) := \colim_{i+j\ge k} F^{\ge i}M \otimes_{\mc C} F^{\ge j}N.\]
If $\mc C$ admits all sequential colimits, we have the underlying object functor \[F^{\ge -\infty} : \Fil \mc C \to \mc C, \quad F^{\ge *}M \mapsto F^{\ge -\infty}M := \colim_{k\to -\infty} F^{\ge k}M.\]
We say that an object $F^{\ge *}M\in \Fil \mc C$ is complete if $\lim_i F^{\ge i}M\simeq 0$, and we will denote the full subcategory of complete filtered objects of $\mc C$ by $\hat \Fil\mc C$.

The basic example of filtered spectra we will need is the following construction, see \cite[Section 2]{HahnRaksitWilson_motivic_fil_on_TC}. 
\begin{lemma}
    There is a lax symmetric monoidal functor \[\CAlg(\Sp_p) \to \CAlg(\hat \Fil \Sp_p), \quad R \mapsto F_\ev^*R\] 
    which is right Kan extended from even rings and satisfies $F_\ev^{\ge *}R := \tau_{\ge 2*} R$ whenever $R$ is even. 
\end{lemma}
In particular, we set $\S_\ev := F_\ev^* \S_p \in \Fil \Sp_p$. We abuse notation by not including the dependence of $p$ in $\S_\ev$. 

\begin{definition}
    The category of ($p$-complete) synthetic spectra is the symmetric monoidal stable $\infty$-category $$\SynSp_p := \Mod_{\S_\ev}(\Fil \Sp_p).$$
\end{definition}
Note that $F_\ev^{\ge *}:\CAlg(\Sp_p) \to \CAlg(\hat \Fil \Sp_p)$ factors through a lax-symmetric monoidal functor \[F_\ev^{\ge *} : \CAlg(\Sp_p) \to \CAlg(\SynSp_p).\]

Set $\T_\ev := F_\ev^{\ge *} \S_p[S^1]$. The bicommutative bialgebra structure on $\S_p[S^1]$ naturally upgrades to a bicommutative bialgebra structure on $\T_\ev$ in the category $\SynSp_p$; see \cite{antieau2024cyclotomic} for more details. 

\begin{definition}
    A synthetic spectrum with synthetic circle action is a $\T_\ev$-module in $\SynSp_p$. Write $$\SynSp_{\T_\ev} := \Mod_{\T_\ev}(\SynSp_p),$$ where as usual we leave the $p$ implicit. 
\end{definition}
Since $\T_\ev$ is a bicommutative bialgebra, it in particular has an $E_\infty$-augmentation $\T_\ev \to \S_\ev$, and so restriction of scalars along this augmentation yields a functor \[(-)^\triv : \SynSp_p \to \SynSp_{\T_\ev} .\]
This functor has left and right adjoints \[(-)_{h\T_\ev} := \S_\ev \otimes_{\T_\ev} - : \SynSp_{\T_\ev} \to \SynSp_p \quad \text{ and } \quad (-)^{h\T_\ev} := \Hom_{\T_\ev}(\S_\ev, -) : \SynSp_{\T_\ev} \to \SynSp_p.\]
This is a filtered version of the homotopy orbits and homotopy fixed points under a circle action, whence the notation. There is a synthetic norm map $$\Nm_{\T_\ev}: (-)_{h\T_\ev} (-1)[1] \to (-)^{h\T_\ev}$$
(see \cite[Construction 2.58]{antieau2024cyclotomic}) and we define \[(-)^{t\T_\ev} := \cofib(\Nm_{\T_\ev}).\]

The $p^k$-power map $S^1 \to S^1$ induces a map $\rho(k) : \T_\ev \to \T_\ev$. There are left and right adjoints of the restriction-of-scalars along $\rho(k)$ functor $\rho(k)_*: \SynSp_{\T_\ev} \to \SynSp_{\T_\ev}$ given by
\begin{align*}
    (-)_{hC_{p^k,\ev}} := \rho(n)_*\T_\ev \otimes_{\T_\ev} - &: \SynSp_{\T_\ev} \to \SynSp_{\T_\ev} , \text{ and}\\ 
     (-)^{hC_{p^k,\ev}} := \Hom_{\T_\ev}(\rho(k)_*\T_\ev, -) &: \SynSp_{\T_\ev} \to \SynSp_{\T_\ev}.
\end{align*}
Again, there is a norm functor $$\Nm_{C_{p^k,\ev}}: (-)_{hC_{p^k,\ev}} (-1)[1] \to (-)^{hC_{p^k,\ev}}$$
and we define \[(-)^{tC_{p^k,\ev}} : \cofib(\Nm_{C_{p^k,\ev}}).\]

We finally define cyclotomic synthetic spectra. 
\begin{definition}
    A ($p$-complete) cyclotomic synthetic spectrum is a $p$-complete synthetic spectrum $F^{\ge *}M$ with synthetic circle action equipped with a map $\varphi_M : F^{\ge *}M \to (F^{\ge *}M)^{tC_{p,\ev}}$. Let $\CycSyn_p$ denote the category of $p$-complete cyclotomic synthetic spectra. 

    This is a symmetric monoidal stable $\infty$-category with monoidal unit $\S_\ev^\triv$. 
\end{definition}

The following example is \cite[Proposition 3.25]{antieau2024cyclotomic}.
\begin{example}\label{ex::even_cyc_sp_becomes_synth}
    Suppose $X\in \CycSp$ is an $E_\infty$-ring in cyclotomic spectra such that the underlying spectrum of $X$ is even. Then $F_\ev^{\ge *}X$ naturally admits the structure of an $E_\infty$-ring object in $\CycSyn$. This essentially follows because $F_\ev^{\ge *}X \simeq \tau_{\ge 2*}X$ and $(F_\ev^{\ge *}X)^{tC_{p,\ev}} \simeq \tau_{\ge 2*}X^{tC_p}$, and so the cyclotomic Frobenius $\varphi_X:X\to X^{tC_p}$ naturally upgrades to a filtered map. See \cite[Proposition 3.25]{antieau2024cyclotomic} for an upgrade of this construction to an equivalence of categories. 
\end{example}

Given $F^{\ge *}M\in \CycSyn_p$, we define \[\TC(F^{\ge *}M) := \map_{\CycSyn}(\S_\ev^\triv, F^{\ge *}M) \simeq \fib\left(\can - \varphi_M^{h\T_\ev} : (F^{\ge *}M)^{h\T_\ev} \to\left((F^{\ge *}M)^{tC_{p,\ev}}\right)^{h\T_\ev} \right).\]
With all of this notation out of the way, we finally come to the result we will need.
\begin{theorem}[{\cite{antieau2024cyclotomic}}]
    There is a functor \[F_\ev^{\ge *} \THH(-) : \QSyn_p \to \CAlg(\CycSyn_p)\]
    such that for any $R\in \QSyn$, there are natural equivalences of $E_\infty$-rings in $p$-complete filtered spectra:
    \begin{align*}
        F_\ev^{\ge *}\THH(R) &\simeq \Fil_\mot^{\ge *}\THH(R), &  \left(F_\ev^{\ge *}\THH(R) \right)^{h\T_\ev} &\simeq \Fil_\mot^{\ge *}\TC^-(R), \\
        \left(F_\ev^{\ge *}\THH(R) \right)^{t\T_\ev} &\simeq \Fil_\mot^{\ge *}\TP(R), & \TC\left(F_\ev^{\ge *}\THH(R) \right) &\simeq \Fil_\mot^{\ge *}\TC(R).
    \end{align*}
    Moreover, the underlying cyclotomic spectrum of $F_\ev^{\ge *}\THH$ identifies canonically with $\THH$.
\end{theorem}

\subsection{Computations in $\THH$}
In this subsection we record some computations of $\THH$ of certain $E_\infty$-rings that will be crucial in the proof of \cref{mainthm::desc_of_limit_of_THH_of_cyc}.

\subsubsection{$\THH$ of Group Algebras}\label{sctn::thh_of_grp_algebras}

For $G$ a topological group, write $\S[G]:= \Sigma_+^\infty G$ for the corresponding $E_1$-ring. If $G$ is abelian, which we will assume from now on, then $\S[G]$ is an $E_\infty$-ring. 
In this case, the cyclotomic spectrum $\S[G]^\Tate$ is particularly nice.
\begin{lemma}\label{lem::generalized_segal_conjecture_and_Tate_frobenius}
    For $G$ either a discrete abelian group or a topological abelian group that is a finite CW complex, we have $\S_p[G]^{tC_p} \simeq \S_p[G]^\triv$ of $S^1$-equivariant spectra, and the Tate-valued Frobenius \[\varphi_{\S_p[G]}^\Tate : \S_p[G]\to \S_p[G]^{tC_p} \simeq \S_p[G]\]
    is induced by the multiplication-by-$p$ map $G\to G$. 
\end{lemma}
\begin{proof}
    That $\S_p[G]^{tC_p} \simeq \S_p[G]$ follows in the discrete case from \cite[Lemma 6.7]{yuan2023integral} and in the finite CW complex case from \cite[Theorem 1.2 and Example 4.4]{burklund2024note}; these follow from the now proven Segal conjecture $\S_p^{tC_p} \simeq \S_p$. The second part of the statement then follows from the definition of the Tate-valued Frobenius.
\end{proof}

We summarise everything we need to know about $\THH(\S[G])$ for $G$ an abelian topological group. All of these results are classical, but for a modern reference see \cite{nikolaus_scholze}. 

\begin{theorem}\label{thm::THH_of_group_rings}
    Let $G$ be an abelian topological group.
    \begin{enumerate}
        \item There is a natural $S^1$-equivariant equivalence \[\THH(\S[G]) \simeq \S[LBG]\]
        where $LBG := \Map_{\ms S}(S^1, BG)$ is the free loop space of the classifying space $BG$, with the natural $S^1$-action. 

        \item The cyclotomic Frobenius is given by \[\THH(\S[G]) \simeq \S[LBG] \to \S[LBG]^{hC_p} \to \S[LBG]^{tC_p} \simeq \THH(\S[G])^{tC_p}\]
        where the first map is induced by the map \[LBG \to LBG^{hC_p} \simeq \Map(S^1, BG)^{hC_p} \simeq \Map(S^1/C_p, BG) \]
        given by taking $\Map(-,BG)$ of the $p$-power covering map $[p]:S^1\to S^1$ (this latter map factors through $S^1/C_p$ and induces the equivalence $S^1/C_p\cong S^1$). 

        \item The assembly map $\S[BG]^\triv \to \THH(\S[G])$, under the equivalence $\THH(\S[G]) \simeq \S[LBG]$, is induced by the $S^1$-equivariant map $BG\to LBG$ sending a point of $BG$ to the constant map valued at that point. 

        \item We have an equivalence of cyclotomic spectra \[\S^\triv \otimes_{\S[BG]^\triv} \THH(\S[G]) \simeq \S[G]^\Tate.\]

        \item For $H$ a subgroup of $G$, the canonical map $\THH(\S[H]) \to \THH(\S[G])$ induced by the ring map $\S[H]\to \S[G]$ is induced by the map $LBH\to LBG$ given by post-composing with the projection $BH\to BG$. 
    \end{enumerate}
\end{theorem}

\begin{remark}\label{rmk::thh_of_gruop_alg_for_infinite_cyc_grp}
    For $G = \Z$, we can write the $S^1$-action on $S^1\times \Z \simeq LB\Z$ as $t\cdot (s,n) = (t^ns, n)$ for $t\in S^1$ and $(s,n)\in S^1\times \Z$. 

    Under the $S^1$-equivariant equivalence $\THH(\S[\Z]) \simeq \S[S^1 \times \Z]$, the cyclotomic Frobenius lift $\THH(\S[\Z])\to \THH(\S[\Z])^{hC_p}$ is induced by the map \[S^1\times \Z \simeq \Map(S^1, B\Z) \to \Map(S^1, B\Z)^{hC_p} \simeq \Map(S^1/C_p, B\Z) \simeq S^1 \times \Z\]
    given by $(s, n) \mapsto (s, pn)$. Indeed, the map $z\mapsto sz^n$ upon pre-composition with the $p$-power map goes to $z\mapsto sz^{pn}$. This fixes a typo in \cite[Proposition 11.3]{bms2}. 
\end{remark}

We now make a few remarks about the transfer map in $\THH$ for group algebras of abelian groups, which is the key input in the proof of \cref{mainthm::desc_of_transfers}. Let $G$ be an abelian topological group and $H$ an open subgroup of finite index. Then $\S[G]\in \Perf(\S[H])$ so that we have the restriction of scalars map \[\Perf(\S[G])\to \Perf(\S[H]).\]
This map is moreover $\Perf(\S[H])$ linear. In particular, we get the $\THH(\S[H])$-linear transfer map \[\tr_{G/H}:\THH(\S[G]) \to \THH(\S[H]).\]
This map is explicitly described in \cite{schlichtkrull1998transfer, schlichtkrull2006transfer}. Rather than stating the result in full generality, we specialise to the case of infinite cyclic groups. 

\begin{proposition}[Schlichtkrull]\label{prop::explicit_thh_transfer_for_cyclic_groups}
    Under the $S^1$-equivariant equivalences of Remark \ref{rmk::thh_of_gruop_alg_for_infinite_cyc_grp}, the transfer map \[\tr_{\Z/n\Z} : \THH(\S[\Z])\to \THH(\S[n\Z])\] is given by \[\S[\Z\times S^1] \to \S[n\Z \times S^1], \quad (a, s) \mapsto \begin{cases}
        0 & n\nmid a, \\
        \sum_{j=0}^{n-1} (a, s_j) & n|a,
    \end{cases}\]
    where $\{s_0, ..., s_{n-1}\}$ is the preimage of $s\in S^1$ under the $n$-power map $S^1\to S^1$. 
\end{proposition}
\begin{remark}
    Here, the $S^1$-action on the right hand side of the decomposition $\THH(\S[n\Z]) \simeq \S[n\Z \times S^1]$ is given informally by \[t \cdot (a, s) = (a, t^{a/n} s)\]
    where $a\in n\Z$ and $s,t\in S^1$. 
\end{remark}
One checks rather easily that the above map as written is $S^1$-equivariant and commutes with Frobenius.

\subsubsection{$\THH$ of Cyclotomic Extensions}\label{sctn::thh_of_cyclotomic_extns}
In this subsection, we give an explicit description of the cyclotomic spectrum $\THH(\Z_p[\zeta_{p^n}])$, due to Devalapurkar--Raksit \cite{sanath_arpon_image_of_j}. These results are one of the key inputs in the proof of \cref{mainthm::desc_of_limit_of_THH_of_cyc}.
In order to give this explicit description, we need to recall some notation from \cite[Remark 6.1.5]{sanath_thesis}.
\begin{construction}
    For $n\ge 0$, define  \[\ku_{p,n} := \tau_{\ge 0} (\ku_p ^{hC_{p^{n}}})\]
    where $\ku_p$ has the trivial $C_{p^n}$-action. Here, $\ku_{p,0} := \ku_p$. There is a residual action of $S^1 / C_{p^n} \cong S^1$ on $\ku_{p,n}$ (it is the trivial circle action on $\ku_{p,0}=\ku_p$). This $S^1$-action on $\ku_{p,n}$ refines to a cyclotomic structure on $\ku_{p,n}$ as follows. 

    Since $\ku_{p}$ has the trivial $S^1$-action, we have an $S^1$-equivariant $E_\infty$-ring map $\ku_p \to \ku_p^{hC_p}$. Taking $hC_{p^n}$ fixed points followed by connective covers yields a map \[\ku_{p,n} = \tau_{\ge 0} \ku_p^{hC_{p^n}} \to \tau_{\ge 0} (\ku_p^{hC_p})^{hC_{p^n}} \simeq \tau_{\ge 0} (\ku_p^{hC_{p^n}})^{hC_p}.\] Since $\tau_{\ge 0} :\Sp\to \Sp_{\ge 0}$ is a right adjoint and so commutes with limits (computed in the appropriate category), we have \[\tau_{\ge 0} (\ku_p^{hC_{p^n}})^{hC_p} \simeq \tau_{\ge 0} \ku_{p,n}^{hC_p}.\] We then have the composition \[\ku_{p,n} \to \tau_{\ge 0} \ku_{p,n}^{hC_p} \to \ku_{p,n}^{hC_p} \to \ku_{p,n}^{tC_p}\]
    which defines the cyclotomic Frobenius on $\ku_{p,n}$. 
\end{construction}
Notice that the Adams operations induce an action of $\Z_p^\times$ on $\ku_{p, n}$, which is moreover compatible with the given cyclotomic structure. It is also clear that each $\ku_{p,n}$ is a $\ku_{p,n-1}$-algebra in cyclotomic spectra: indeed, we simply take $\tau_{\ge 0}(-)^{hC_{p^{n-1}}}$ of the $S^1$-equivariant map of $E_\infty$-rings $\ku_p \to \ku_p^{hC_p}$. This $\ku_{p,n-1}$-algebra structure on $\ku_{p,n}$ is also $\Z_p^\times$-equivariant. 

We now define the cyclotomic spectrum $ j_{n}$ as follows. 
\begin{construction}
    Define the cyclotomic spectrum 
    \[ j_{n} := \tau_{\ge 0} \left( \ku_{p,n}^{h(1+p^{n+1}\Z_p)} \right).\]
    This has an $S^1$-action induced by the $S^1$-action on $\ku_{p,n}$. The cyclotomic Frobenius of $ j_{n}$ is constructed as follows: recall that the cyclotomic Frobenius on $\ku_{p,n}$ factors through $\ku_{p,n}^{hC_p}$. Taking $(1+p^{n+1}\Z_p)$ fixed points and then taking connective covers yields \[ j_{n} \to \tau_{\ge 0} (\ku_{p,n}^{hC_p})^{h(1+p^{n+1}\Z_p)} \simeq \tau_{\ge 0} \left (\tau_{\ge 0}\ku_{p,n}^{h(1+p^{n+1}\Z_p)}\right)^{hC_p} \simeq \tau_{\ge 0}  j_{n}^{hC_p} \to  j_{n}^{hC_p} \to  j_{n}^{tC_p}.\]
\end{construction}
Since homotopy fixed points and connective covers are lax-symmetric monoidal, the canonical map $ j_n\to \ku_{p,n}$ is a map of cyclotomic $E_\infty$-rings. We also get a map of cyclotomic $E_\infty$-rings $ j_{n-1}\to  j_n$ given by taking connective covers of the composition \[\ku_{p,n-1}^{h(1+p^n\Z_p)} \to \ku_{p,n-1}^{h(1+p^{n+1}\Z_p)} \to \ku_{p,n}^{h(1+p^{n+1}\Z_p)},\]
where the first map is induced by restricting to fixed points of a smaller subgroup, and the second map is induced by the $\ku_{p,n-1}$-algebra structure map on $\ku_{p,n}$. Note that the $\Z_p^\times$-action on $\ku_{p,n}$ induces a residual $\Z_p^\times/(1+p^n\Z_p) \cong (\Z/p^{n+1}\Z)^\times$ action on $ j_n$.

The main result we will need is the following result, stated in \cite[Remark 6.1.5, Theorem 6.4.1]{sanath_thesis}. 
\begin{theorem}[{Devalapurkar--Raksit}]\label{thm::identification_of_THH_cyclotomic_extn_2}
    For $n\ge 1$, there is an equivalence of $\Gal(\Q_p(\zeta_{p^{n}}) /\Q_p) \cong (\Z/p^n\Z)^\times$-equivariant cyclotomic $E_\infty$-ring spectra \[\THH(\Z_p[\zeta_{p^{n}}]) \simeq  j_{n-1}^{(-1)}.\]
    Moreover, the $E_\infty$-cyclotomic ring map $\THH(\Z_p[\zeta_{p^n}]) \to \THH(\Z_p[\zeta_{p^{n+1}}])$ induced by functoriality of $\THH$ is the $(-)^{(-1)}$-functor of Construction \ref{cons::shift_operator_on_cycsp} applied to the $E_\infty$-ring map $ j_{n-1}\to  j_n$. 
\end{theorem}

One also has the following statement on relative $\THH$, the case $n=1$ of which is already \cite[Theorem 6.4.1]{sanath_thesis}.
\begin{corollary}[Devalapurkar]\label{cor::rel_thh_of_cyc_extns}
    For any $n\ge 1$, view $\Z_p[\zeta_{p^n}]$ as a $\S_p\llb q_n-1\rrb$-algebra via the map $q_n\mapsto \zeta_{p^n}$. Then, there is an equivalence of $\Z_p^\times$-equivariant cyclotomic $E_\infty$-ring spectra \[\THH(\Z_p[\zeta_{p^n}]/\S_p\llb q_n-1\rrb) \simeq \ku_{p,n-1}^{(-1)},\]
    compatible with the equivalences $\THH(\Z_p[\zeta_{p^n}])\simeq j_{n-1}^{(-1)}$. 
\end{corollary}

\begin{remark}\label{rmk::on_conventions_for_frob_twists}
    The reader may have noticed that we have switched convention from the introduction, as we now consider $\Z_p[\zeta_{p^n}]$ as a $\S_p\llb q_n-1\rrb$-algebra via $q_n \mapsto \zeta_{p^n}$, unlike in the introduction (especially the statement of Theorem \ref{mainthm::desc_of_limit_of_THH_of_cyc} and its corollaries) where we use the convention that $\Z_p[\zeta_p]$ is a $\S_p\llb q-1\rrb$-algebra via $q\mapsto \zeta_p$. This is related to Frobenius twists. 
    
    Indeed, the Frobenius map $\varphi:\Z_p\llb q-1\rrb \to \Z_p\llb q-1\rrb$ as also the inclusion $\Z_p\llb q-1\rrb \to \Z_p\llb q_1-1\rrb$. The Frobenius twist $\Z_p\llb q-1\rrb^{(1)}= \varphi^*\Z_p\llb q-1\rrb$, after restricting scalars along Frobenius (i.e. applying $\varphi_*$), identifies with $\Z_p\llb q_1-1\rrb$ as a $\Z_p\llb q-1\rrb$-algebra. 
    
    One has \[\pi_0\ku_{p,0}^{(-1),tS^1} \simeq \prism_{\Z_p[\zeta_p]/\Z_p\llb q-1\rrb}^{(1)}.\]
    In particular, it is a $\Z_p\llb q-1\rrb^{(1)}$-algebra.
    Since we will need to often apply $\varphi_*\varphi^*$, we just make the convention that $\Z_p\llb q-1\rrb^{(1)} =: \Z_p\llb q_1-1\rrb$. In particular, from here on out ( especially in \cref{sctn::trace_for_cyclotomic_tower}) we will use the convention that $\Z_p[\zeta_{p^n}]$ is a $\S_p\llb q_n-1\rrb$-algebra via $q_n \mapsto \zeta_{p^n}$
\end{remark}

One can also calculate $\THH(\Z_p^\cyc)$ as a cyclotomic spectrum. This is in fact much easier than the preceding theorem and follows by the exact same argument as in \cite[Proposition 11.7, Corollary 11.8]{bms2}, using that $\THH(\Z_p[\zeta_p]/\S_p\llb q-1\rrb) \simeq \ku_p^{(-1)}$.
\begin{corollary}[Devalapurkar--Raksit]\label{lem::thh_of_Zp_cyc}
    There is a $\Z_p^\times$-equivariant equivalence of cyclotomic $E_\infty$-ring spectra \[\THH(\Z_p^\cyc) \simeq \ku_p^{(-1)} \otimes_{\S_p\llb q_1-1\rrb^\Tate } \S_p\llb q^{1/p^\infty}-1\rrb^\Tate \]
    compatible with the equivalences $\THH(\Z_p[\zeta_{p^n}]) \simeq  j_{n-1}^{(-1)}$. Here, \[\S_p\llb q^{1/p^\infty}-1\rrb := \left(\colim_n \S_p\llb q_n-1\rrb\right)^\wedge_{(p,q-1)}\] 
\end{corollary}
We end with some remarks on motivic filtrations. 

\begin{remark}
    The motivic filtration on $\THH(\Z_p[\zeta_{p^n}]/\S_p\llb q_{n}-1\rrb) \simeq \ku_{p,n}^{(-1)}$ is given by the even filtration, which identifies with the double-speed Postnikov filtration.

    For absolute $\THH$, the motivic filtration is given by \[\Fil_\mot^{\ge *} \THH(\Z_p[\zeta_{p^n}]) \simeq \tau_{2*-1}  j_{n-1}^{(-1)}\]
    by \cite[Corollary 1.4]{morin2024topological}.
\end{remark}

\section{Transfer Maps for Prismatic Cohomology}\label{sctn::transfer_maps_and_construction_of_trace}
In \cref{sctn::transfer_map_for_finite_extension}, we prove that the transfer maps on the localising invariants considered in \cref{sctn::zoo_of_coh_thries} respect the motivic filtration when $f:X\to Y$ is a finite locally free regular map of $p$-adic formal schemes. As a consequence, we obtain the existence result in \cref{mainthm::prismatic_trace}. The rest of this section is then devoted to proving \cref{mainthm::prismatic_trace}.

\subsection{Transfer Maps for Finite Extensions}\label{sctn::transfer_map_for_finite_extension}
We put ourselves in the following situation. 
\begin{assumption}\label{ass::nice_filtered_localising_invariants}
    Let $\mc F$ be a localising invariant $\Cat_\infty^\perf\to \Sp$ equipped with a functorial exhaustive filtration $$\Fil^\bullet \mc F : \CAlg_\Z^\heart \to \CAlg(\Fil \Sp)$$ satisfying the following conditions:
    \begin{enumerate}
        \item\label{ass_part::gysin_descent} $\Fil^\bullet\mc F$ satisfies descent for the Nisnevich topology; in particular, by Zariski descent, we can view $\Fil^\bullet \mc F:\mr{Sch}^\op \to \CAlg(\Fil\Sp)$.
        \item\label{ass_part::gysin_chern_class} There is a `Chern class' map $$c: \Pic \to \Fil^1\mc F$$ equipped with the data of the following commutative diagram
        \[\begin{tikzcd}
            \Pic \ar{rr}{\mc L\mapsto \O - \mc L^\vee} \ar{d}{c} && K \ar{d}\\
            \Fil^1\mc F\ar{rr}&& \mc F
        \end{tikzcd}\]
        where the right vertical arrow is the canonical map $K\to \mc F$ coming from universality of algebraic $K$-theory as a localizing invariant \cite{blumberg2013universal}. 
        \item\label{ass_part::gysin_proj_bdle_formula}  For all $R\in \CAlg_\Z^\heart$ and all $r\ge 0$, the $\Fil^\bullet \mc F(R)$-linear map \[\sum_{i=0}^r c(\O_{\P^n_R}(1))^i : \bigoplus_{i=0}^r \Fil^{\bullet - i}\mc F(R) \to \Fil^\bullet \mc F(\P^r_R) \]
        is an equivalence; in other words, $\Fil^\bullet \mc F$ satisfies a filtered projective bundle formula. 
    \end{enumerate}
\end{assumption}
\begin{remark}\label{rmk::no_need_for_p-completions_in_motivic_spectra}
    Over here, we have not required the input rings $R$ to be $p$-completed. However, in all of the examples of $\mc F$ that we will care about, one has an equivalence $\Fil^\bullet \mc F(R) \simeq \Fil^\bullet \mc F(R^\wedge_p)$.
\end{remark}

Notice that all of the filtered localising invariants listed in \cref{sctn::zoo_of_coh_thries} satisfy Assumption \ref{ass::nice_filtered_localising_invariants}:
\begin{enumerate}
    \item All of the filtrations are obtained by the sheafification in a topology finer than the Nisnevich topology (\'etale in the case of $L_{K(1)}\TC(-) \simeq L_{K(1)}K(-[\tfrac 1p])$, quasi-syntomic for the rest) of the double-speed Postnikov filtration, so that Assumption \ref{ass::nice_filtered_localising_invariants} (\ref{ass_part::gysin_descent}) holds. 
    
    \item Assumption \ref{ass::nice_filtered_localising_invariants} (\ref{ass_part::gysin_chern_class}) holds for $\Fil^\bullet_\mot \TC$ essentially by construction (see also \cite[Proposition 6.7]{atiyah_duality_motivic_spectra}). The remaining examples of $\Fil^\bullet \mc F$ in \cref{sctn::zoo_of_coh_thries} admit maps from $\Fil_\mot^\bullet \TC$, and one defines their Chern class map to be the composition \[\Pic \to \Fil^1_\mot\TC\to \Fil^1_\mot \mc F.\]
    In particular, they trivially satisfy Assumption \ref{ass::nice_filtered_localising_invariants} (\ref{ass_part::gysin_chern_class}). 
    
    \item For Assumption \ref{ass::nice_filtered_localising_invariants} (\ref{ass_part::gysin_proj_bdle_formula}), if $\mc F\ne L_{K(1)}\TC$ then the filtration is complete and on associated graded the projective bundle formula is given in \cite[Section 9.1]{BL_absolute_prismatic_coh}. 
    
    It remains to check for $L_{K(1)}\TC$. On associated graded, we have \[\R\Gamma_\et(X_\eta, \Z_p(r)) \simeq \fib\left(\varphi - \id : \prism_X\{r\}[\tfrac{1}{\mc I_\prism}]^\wedge_p \to \prism_X\{r\}[\tfrac{1}{\mc I_\prism}]^\wedge_p\right)\]
    and since prismatic cohomology satisfies the projective bundle formula, \'etale cohomology must too. At the level of underlying objects, this follows from $K(1)$-localisation of the projective bundle formula for $\TC$. Since the associated graded functor together with the underlying object functor are jointly conservative, it follows that $\Fil_\mot^\bullet L_{K(1)}\TC$ satisfies it too. 
\end{enumerate}

We will need the following result. 
\begin{theorem}[{\cite[Construction 3.1, Section 4]{tang2026_gysin}}]\label{thm::gysin_map_in_decompleted_TP_preserves_filn}
    Let $\Fil^\bullet\mc F$ satisfy Assumption \ref{ass::nice_filtered_localising_invariants}. 
    Let $X$ be a bounded quasi-syntomic $p$-adic formal scheme and $i : Z\inj X$ a regular closed immersion of codimension $r$. 
    
    Then, there is a canonical map 
    \begin{equation}\label{eqn::filt_gysin_map}
        \Fil^\bullet \mc F(Z) \to \Fil^{\bullet+r} \mc F(X)
    \end{equation}
    of filtered spectra, such that on underlying spectra (i.e. taking the colimit as $\bullet \to \infty$) the induced map $\mc F(Z)\to \mc F(X)$ canonically identifies with \[\mc F(Z) \simeq \mc F(\Perf(Z)) \xrightarrow{\mc F(i_*)} \mc F(\Perf(X)) \simeq \mc F(X).\]

    Moreover, the filtered map (\ref{eqn::filt_gysin_map}) is linear over $\Fil^\bullet \mc F(X)$.
\end{theorem}
\begin{remark}
    At the time of writing, the version of \cite{tang2026_gysin} available publicly does not have a Section 4. The author is relying on an updated version (to be released) coming from private communication with Tang.
\end{remark}

The following is the main result of this section.
\begin{theorem}\label{thm::transfer_map_for_finite_loc_free_extensions}
    Suppose $f:X\to Y$ is a regular finite locally free map of bounded quasi-syntomic $p$-adic formal schemes. Suppose $\Fil^\bullet \mc F$ satisfies Assumption \ref{ass::nice_filtered_localising_invariants}. 
    
    Then, there is a $\Fil^\bullet\mc F(Y)$-linear filtered map \[\tr_f^{\mc F}: \Fil^\bullet \mc F(X) \to \Fil^\bullet\mc F(Y)\]
    satisfying the following properties:
    \begin{enumerate}
        \item on passing to underlying objects of the filtered object, it identifies canonically with the transfer map $\mc F(X)\to \mc F(Y)$;
        \item it is functorial in base-change in $f$;
        \item\label{trace_thm_part::functoriality_in_FilF} it is functorial in maps $\Fil^\bullet \mc F\to \Fil^\bullet \mc F$ that also preserve the Chern class map;
        \item if $g:Y\to Z$ is another regular finite locally free map of bounded quasi-syntomic $p$-adic formal schemes, then there is a non-canonical homotopy \[\tr_g^{\mc F} \circ \tr_f^{\mc F} \simeq \tr_{g\circ f}^{\mc F};\] 
        \item\label{trace_thm_part::factorisation} if $X\inj[\iota] \A^n_Y\subset \P^n_Y \sur[\pi] Y$ is a factorisation of $f$ through a regular closed immersion $\iota$, then there is a homotopy 

\[\begin{tikzcd}
	{\Fil_\mot^{\bullet + n} \mc F(\P^n_Y)} & {\bigoplus_{i=0}^n \Fil_\mot^{\bullet +n-i} \mc F(Y) \cdot c(\O(1))^i} \\
	{\Fil_\mot^\bullet \mc F(X)} & {\Fil_\mot^\bullet \mc F(Y)}
	\arrow[""{name=0, anchor=center, inner sep=0}, "\simeq"{description}, draw=none, from=1-1, to=1-2]
	\arrow["{\text{coeff of } c(\O(1))^n}", from=1-2, to=2-2]
	\arrow["{\mr{gys}_\iota}", from=2-1, to=1-1]
	\arrow[""{name=1, anchor=center, inner sep=0}, "{\tr_f^{\mc F}}"', from=2-1, to=2-2]
	\arrow[shorten <= 15pt, shorten >= 15pt, Rightarrow, from=0, to=1]
\end{tikzcd}\]
        where the map marked $\mr{gys}_\iota$ is the filtered map of \cref{thm::gysin_map_in_decompleted_TP_preserves_filn} for the regular closed immersion $\iota$. This homotopy is functorial in $\mc F$ as in (\ref{trace_thm_part::functoriality_in_FilF}) above, as well as functorial in base-change in $Y$. 
    \end{enumerate}
\end{theorem}
The following lemma explains why statement (\ref{trace_thm_part::factorisation}) above makes sense. 

\begin{lemma}\label{lem::Gysin_image_maps_to_correct_filn_piece}
    Let $\iota : \Spf S\inj \A^n_R\subset \P^n_R$ be a regular closed immersion of codimension $n$, $\pi:\P^n_R \to \Spf R$ the projection map, and $f = \pi\circ \iota :\Spf S\to \Spf R$. Suppose $f$ is finite flat. Suppose $\Fil^\bullet \mc F$ is as in Assumption \ref{ass::nice_filtered_localising_invariants}. Then, the following hold.
    \begin{enumerate}
        \item the map $\mc F(\pi_*):\mc F(\P^n_R)\to \mc F(R)$ sends $c(\O(1))^i$ to $1$ for all $0\le i\le n$;
        \item the composition \[\mc F(S) \xrightarrow{\mc F(\iota_*)} \mc F(\P^n_R) \simeq \bigoplus_{i=0}^n \mc F(R)\cdot c(\O(1))^i \xrightarrow{\text{coefficient of } c(\O(1))^i} \mc F(R)\]
        is canonically zero unless $i = n$, in which case it is canonically homotopic to $\mc F(f_*): \mc F(S) \to \mc F(R)$. 
    \end{enumerate}    
\end{lemma}
\begin{proof}
    By Assumption \ref{ass::nice_filtered_localising_invariants}(\ref{ass_part::gysin_chern_class}), and since the lemma is about the underlying unfiltered localising invariant, it suffices to restrict to $\mc F = K$. For the first statement, since the projective bundle formula for $K(\P^n_R)$ is $K(R)$-linear, it suffices to check that $\pi_*c^i$ is 1 for the class $c := 1-[\O(-1)] \in K_0(\P^n_R)$. By \cite[Tag01XT]{stacks-project}, one computes that $$\pi_* \O(-i) = \R\Gamma(\P^n_R, \O(-i)) = 0$$
    for $1\le i\le n$. The binomial formula then yields that $\pi_*c^i = 1$. 
    
    We now prove the second statement. Since $\iota_*$ and $\R\Gamma(\P^n_R, -)$ are $\Perf(R)$-linear functors (i.e. the projection formula holds), the maps $K(S)\to K(R)$ are all $K(R)$-linear. It thus suffices to prove the claim at the level of $K_0$. Let $M\in \Perf(S)$. By the filtered projective bundle formula as well as the proof of \cite[Theorem 1.5]{blumberg2012localization}, we can express $[\iota_*M]\in K_0(\P^n_R)$ as a linear combination 
    \[[\iota_*M] = \sum_{i=0}^n c_i u^i\]
    of powers of $u = 1- c(\O(1)) = [\O(-1)]\in K_0(R)$, where $c_i\in K_0(R)$ are defined by $c_0 = [\R\Gamma(\P^n_R, \iota_*M)]$ and \[c_k = [\R\Gamma(\P^n_R, (\iota_*M)(k))] - \sum_{i=0}^{k-1}\binom{n+k-i}{n} c_i.\]
    By the projection formula, we have $$(\iota_*M)(k) \simeq \iota_*\big(M\otimes \iota^*\O_{\P^n_R}(k)\big).$$
    Since $\Spf S$ does not intersect the hyperplane at $\infty$ by assumption, we have $\iota^*\O_{\P^n_R}(k) \cong S$, the trivial line bundle, and so \[\R\Gamma(\P^n_R, (\iota_*M)(k)) \simeq \R\Gamma(\P^n_R, \iota_*M) \simeq f_*M.\] 
    Hence, we have $c_k = [f_* M] a_k\in K_0(R)$ where $a_k\in \Z$ satisfy the recurrence relation $a_0 = 1$ and \[a_k = 1 - \sum_{i=0}^{k-1} \binom{n+k-i}{n}a_i.\]
    One checks that $a_k = (-1)^k\binom{n}{k}$, which can be proven by induction and an application of \cref{lem::combinatorial_identity_computing_chern_class_for_O(-n-1)} below.

    Thus, we see that \[[\iota_*M] = [f_*M]\sum_{k=0}^n(-1)^k\binom{n}{k} u^k \in K_0(\P^n_R).\]
    Changing base, we see that \[[\iota_* M] = [f_* M] \beta^n.\]
    This proves the second statement.
\end{proof}
\begin{lemma}\label{lem::combinatorial_identity_computing_chern_class_for_O(-n-1)}
    For all integers $0\le k< n$, one has the identity
    \[(-1)^{k+1} \binom{n}{k+1} = 1 - \sum_{i=0}^{k} (-1)^i \binom{n+k+1-i}{n}\binom{n}{i}.\]
\end{lemma}
\begin{proof}
    The generating function of the right hand side is 
    \begin{align*}
        G(z) &:= \sum_{k\ge 0} \left(1 - \sum_{i=0}^{k} (-1)^i \binom{n+k+1-i}{n}\binom{n}{i}\right)z^k = \frac{1}{1-z} - \sum_{k\ge i\ge 0}(-1)^i \binom{n+k+1-i}{n}\binom{n}{i} z^k\\
        &= \frac{1}{1-z} - \sum_{i,j\ge 0} (-1)^i \binom{n+1+j}{n}\binom{n}{i}z^{i+j} = \frac{1}{1-z} - \left(\sum_{i\ge 0} (-1)^i\binom{n}{i}z^i\right)\left(\sum_{j\ge 0} \binom{n+1+j}{n} z^j\right)\\
        &= \frac{1}{1-z} - (1-z)^n\left(\frac{(1-z)^{-n-1}-1}{z}\right) = \frac{(1-z)^n-1}{z}
    \end{align*}
    and the coefficient of $z^k$ is $(-1)^{k+1}\binom{n}{k+1}$, as expected. 
\end{proof}

The following is an immediate consequence of \cref{lem::Gysin_image_maps_to_correct_filn_piece} and the filtered projective bundle formula.
\begin{corollary}\label{cor::Gysin_image_maps_to_correct_filn_piece}
    Keep notation as in \cref{lem::Gysin_image_maps_to_correct_filn_piece}. There is a canonical choice of null-homotopy, functorial in $\Fil^\bullet \mc F$ (as in \cref{thm::transfer_map_for_finite_loc_free_extensions}(\ref{trace_thm_part::functoriality_in_FilF}) above) and depending only on the factorisation $f=\pi\circ i$, for the composition \[\Fil^\bullet \mc F(S) \to \mc F(S) \xrightarrow{\mc F(i_*)} \mc F(\P^n_R) \xrightarrow{\mc F(\pi_*)} \mc F(R)\to \mc F(R)/\Fil^{\bullet}\mc F(R).\]
\end{corollary}

We are now in a position to prove \cref{thm::transfer_map_for_finite_loc_free_extensions}.

\begin{proof}[Proof of \cref{thm::transfer_map_for_finite_loc_free_extensions}]
    By Zariski descent, we may suppose that $Y = \Spf R$ and thus $X = \Spf S$ are affine. \cref{cor::Gysin_image_maps_to_correct_filn_piece} gives the entire theorem for free, as long as we can do the following:
    \begin{enumerate}
        \item choose a factorisation $X\inj[\iota] \A^n_Y\subset \P^n_Y \sur Y$ that is functorial in base-change in $f$, where $\iota$ is a regular closed immersion, and
        \item show that any two such factorisations $X\inj[\iota] \P^n_Y\sur Y$ and $X\inj[\iota'] \P^{n'}_Y\sur Y$ have a common refinement $X\inj[\tilde\iota] \P^m_Y \sur Y$, in the sense that the following diagram commutes
        \[\begin{tikzcd}
            & \A^n_Y\ar[hook]{d}{\sigma} \ar[hook]{r} & \P^n_Y \ar[hook]{d} \ar[two heads]{rd} & \\
            X \ar[hook]{r}{\tilde \iota} \ar[hook,swap]{rd}{ \iota'} \ar[hook]{ru}{\iota} & \A^m_Y \ar[hook]{r} & \P^m_Y \ar[two heads]{r} & Y\\
            & \A^n_Y\ar[hook]{u}{\sigma'} \ar[hook]{r} & \P^{n'}_Y \ar[swap, hook]{u} \ar[two heads]{ru} & 
        \end{tikzcd}\]
        for some choice of regular closed immersions $\sigma:\A^n_Y \inj \A^m_Y$ and $\sigma':\A^{n'}_Y\inj \A^m_Y$.
    \end{enumerate}

    For the first, since $f_*\O_X$ is a vector bundle on $Y$ by assumption, we take $\iota$ to be the canonical map \[X \inj \V_Y(f_*\O_X);\]
    this is obtained by taking $\Spf$ of the map $\Sym_R(S) \to S$ induced by the multiplication on $S$. That this $\iota$ is a regular closed immersion follows from \cite[Tag069G]{stacks-project}. This is also clearly functorial in pull-backs along arbitrary maps $Y'\to Y$. 

    For the second, we can simply take $m = n+n'$ and $\tilde \iota:X\inj \A_Y^m$ to be the factorisation \[X \inj X\times_Y X \inj[\iota\times \iota'] \A^n_Y \times_Y \A^{n'}_Y \simeq \A^{n+n'}_Y,\]
    with $\sigma$ and $\sigma'$ the inclusion of the first $n$ and last $n'$ coordinates respectively. 
\end{proof}

\begin{remark}
    We expect there to be a canonical choice of null-homotopy of the composite \[\Fil^\bullet \mc F(X) \to \mc F(X) \xrightarrow{\mc F(f_*)} \mc F(Y) \to \frac{\mc F(Y)}{\Fil^\bullet \mc F(Y)}.\]
    For instance, if $\ms I$ denotes the 1-category of regular closed immersions $\iota : X\inj \A^n_Y$ factorising $f$ (morphisms given by maps $\A^n_Y\to \A^m_Y$ compatible with $f$), then \cref{cor::Gysin_image_maps_to_correct_filn_piece} yields a functor from $\ms I$ to the space of null-homotopies of the above composition. If one can show that this functor admits a colimit, then one immediately gets a stronger result than \cref{thm::transfer_map_for_finite_loc_free_extensions}, where the choice of map $\tr_f^{\mc F}$ is canonically functorial in composition in $f$ too. Unfortunately $\ms I$ is not filtered, so it is non-trivial to check whether such a colimit exists without digging deeper into the space of null-homotopies of the above composition. 
\end{remark}

\subsection{Transfer Maps and Prismatic $F$-Gauges}\label{sctn::transfer_maps_on_F_gauges}
We now show \cref{thm::transfer_map_for_finite_loc_free_extensions} yields a construction of transfer maps for $F$-gauges, thus proving the existence part of \cref{mainthm::prismatic_trace}.

\begin{lemma}\label{lem::abstract_trace_maps_from_transfers}
    Suppose $\mc C$ is a class of maps between bounded quasi-syntomic $p$-adic formal schemes closed under base-change and composition. Suppose for every $f:X\to Y \in \mc C$ that $f_*$ induces transfer maps \[\Fil^\bullet\TC_\prism^-(X)\to \Fil^{\bullet+k} \TC_\prism^-(Y) \quad \text{ and } \quad \Fil^\bullet\TP_\prism(X)\to \Fil^{\bullet+k} \TP_\prism(Y)\]
    for some fixed $k\in \Z$ independent of $f$, functorial in change-of-base $Y$ and composition in $f$ (the latter only up to possibly non-canonical homotopy), and compatible with the two maps $\can, \varphi : \TC_\prism^-\to \TP_\prism$. Then, there is a map of prismatic $F$-gauges 
    \[f_*^\syn \O_{X^\syn} \to \O_{Y^\syn}\{k\}[2k]\]
    again functorial in base-change in $Y$ and composition in $f$ (the latter only up to possibly non-canonical homotopy).
\end{lemma}
\begin{proof}
    By quasi-syntomic descent, we may assume that $Y = \Spf R$ is a QRSP. In this case, recall that the category of prismatic $F$-gauges over $Y$ a QRSP is equivalent to the category of $(p,I)$-complete $\Z$-indexed filtered modules $\Fil^\bullet M$ over the filtered ring $\Fil^\bullet_\Nyg \prism_R$ equipped with a $\Fil^\bullet_\Nyg\prism_R$-linear map \[\varphi : \Fil^\bullet M \to I^\bullet M := I^\bullet\prism_R \otimes_{\prism_R} M\]
    such that the map \[\Fil^\bullet M \hat\otimes_{\Fil^\bullet_\Nyg \prism_R, \varphi} I^\bullet \prism_R \to I^\bullet M\]
    is an isomorphism of filtered $I^\bullet \prism_R$-modules \cite[Remark 3.22]{mondal2024dieudonn}; here, the tensor product on the left is $(p,I)$-completed. In particular, the prismatic $F$-gauge associated to $\O_{Y^\Syn}\{r\}[2r]$ is simply $(\Fil^{\bullet+r}_\Nyg \prism_R)\{r\}[2r]$ itself, while the prismatic $F$-gauge associated to $f^\syn_* \O_{X^\Syn}$ is $\Fil^\bullet_\Nyg \R\Gamma_\prism(X)$. 

    Taking the $i$'th associated graded of the filtered transfer maps in the statement of the lemma, we get a map \[\Fil^i_\Nyg \R\Gamma_\prism(X, \O\{i\})[2i] \to (\Fil^{i+k}_\Nyg \prism_R)\{i+k\}[2i+2k]. \]
    Since $\R\Gamma_\prism(X)$ is an algebra over $\prism_R$ and $\R\Gamma_\prism(X, \O\{i\})$ is a module over $\R\Gamma_\prism(X)$, it follows that we can pull the Breuil--Kisin twist out of the cohomology. We thus get a map \[\Fil^i_\Nyg \R\Gamma_\prism(X) \to (\Fil^{i+k}_\Nyg \prism_R)\{k\}[2k]\]
    for all $i$. This map is compatible with the canonical and Frobenius maps by assumption. Hence, the map \[\Fil^\bullet_\Nyg \R\Gamma_\prism(X) \to (\Fil^{\bullet+k}_\Nyg \prism_R)\{k\}[2k]\]
    is actually a map of prismatic $F$-gauges on $Y^\Syn$. This gives us the required map $f^\syn_* \O_{X^\Syn} \to \O_{Y^\Syn}\{k\}[2k]$.
\end{proof}

\begin{lemma}
    For $f:X\to Y$ any quasi-syntomic map of qcqs quasi-syntomic $p$-adic formal schemes, and any $\mc E\in \mc D_\qcoh(X^\Syn)$ and $\mc F\in \Perf(Y^\Syn)$, the projection formula holds:\[(f^\syn_*\mc E) \otimes \mc F \simeq f^\syn_*(\mc E\otimes f^{\syn,*}\mc F).\]
\end{lemma}
\begin{proof}
    Note that there is always a canonical map \[(f^\syn_*\mc E) \otimes \mc F \to f^\syn_*(\mc E\otimes f^{\syn,*}\mc F)\]
    adjoint to the map $$f^{\syn,*}((f^\syn_*\mc E) \otimes \mc F) \to (f^{\syn,*} f^\syn_*\mc E) \otimes f^{\syn,*}\mc F \to \mc E\otimes f^{\syn,*}\mc F.$$
    By a thick subcategory argument, we reduce to the case $\mc F = \O_{Y^\Syn}$ in which case the projection formula is a tautology. 
\end{proof}

We finally prove the existence of the prismatic trace map. 

\begin{proof}[Proof of Existence of Traces and Statement (\ref{mainthm_part::projection_formula}) in \cref{mainthm::prismatic_trace}]
   By \cref{thm::transfer_map_for_finite_loc_free_extensions} and \cref{lem::abstract_trace_maps_from_transfers} for $f:X\to Y$ finite flat and regular, we get a trace map of prismatic $F$-gauges \[\tr^\syn_f: f_*^\syn \O_{X^\Syn} \to \O_{Y^\Syn}.\] Using the projection formula \[f_*^\syn \O_{X^\Syn} \otimes \mc E \simeq f^\syn_* f^{\syn,*}\mc E\]
    for any $\mc E\in \Perf(Y^\Syn)$, we get a trace map \[\tr_f^\syn({\mc E}):  f^\syn_* f^{\syn,*}\mc E \to \mc E\]
    for any prismatic $F$-Gauge $\mc E$ on $Y$. This also establishes \cref{mainthm::prismatic_trace}(\ref{mainthm_part::projection_formula}). 
\end{proof}

\begin{remark}
    Combining \cref{thm::gysin_map_in_decompleted_TP_preserves_filn} and \cite[Theorem 5.38]{tang2024syntomic}, we see that the Gysin map \[\Gys^\syn_{\O\{r\}} : i_*\O_{Z^\Syn} \to \O_{X^\Syn}\{r\}[2r]\]
    coming from \cref{lem::abstract_trace_maps_from_transfers} applied to \cref{thm::gysin_map_in_decompleted_TP_preserves_filn} 
    upon taking coherent cohomology over $X^\Syn$ coincides with the Gysin map \[\Z_p(i)(Z) \to \Z_p(i+r)(X)[2r]\]
    on syntomic cohomology given in \cite[Definition 5.32]{tang2024syntomic}.
\end{remark}

\begin{lemma}
    Suppose $f:X\to Y$ is a finite locally free regular map of quasi-syntomic qcqs $p$-adic formal schemes. Then, for any prismatic $F$-gauge $\mc E\in \mc D_{\qcoh}(Y^\Syn)$, the composition \[\mc E\to f_*^\syn f^{\syn,*}\mc E\xrightarrow{\tr_f^\syn(\mc E)} \mc E\]
    is multiplication by $\deg f$.

    In other words, \cref{mainthm::prismatic_trace}(\ref{mainthm_part::normalisation}) is true.
\end{lemma}
\begin{proof}
    By Zariski descent, we may suppose $X = \Spf S$ and $Y=\Spf R$ are affine and that $S$ is a free $R$-module of rank $d$. By construction of the prismatic trace, it suffices to check that the composition \[\mc F(R)\to \mc F(S)\xrightarrow{\text{transfer}} \mc F(R)\]
    is multiplication by $d$, for $\mc F\in \{\TC^-_\prism, \TP_\prism\}$. Since this composition is $\mc F(R)$-linear and so $K(R)$-linear, it suffices to check that the above composition is multiplication by $d$ in the `universal' case $\mc F = K$. By linearity again, it suffices to check that $1\in \pi_0K(R)$ is sent to $d\in \pi_0K(R)$. However, this is clear, since the image of 1 in $K_0(R)$ under the above composition is by definition $[S] = [R^d] = d\in K_0(R)$. 
\end{proof}

\subsection{Trace Maps on Prismatic Cohomology for Perfectoids}\label{sctn::calculating_prismatic_trace_for_perfectoids}

We analyse the prismatic trace map associated to a (module-)finite locally free extension $R\to S$ where both $R$ and $S$ are perfectoid.
Then, $\prism_R = A_{\inf}(R)$ and $\prism_S = A_{\inf}(S)$. 

The trace map on prismatic cohomology is thus a map \[\tr_{S/R}^\prism: A_{\inf}(S)\to A_{\inf}(R).\]
On the other hand, by derived Nakayama, we know that the extension $A_{\inf}(R)\to A_{\inf}(S)$ is finite and of the same degree as $S/R$. By the same argument, it is also clear that $A_{\inf}(S)$ is locally free as an $A_{\inf}(R)$-module as the same is true for $S$ over $R$. Thus, there is also the usual trace map \[\tr_{A_{\inf}}:A_{\inf}(S)\to A_{\inf}(R)\]
associated to a finite locally free extension. In this subsection we check that these two maps coincide.

\begin{proposition}\label{prop::compatibility_for_trace_with_perfectoids}
    For any finite locally free extension $S/R$ of $p$-torsion free integral perfectoids, the two maps $\tr^\prism_{S/R}$ and $\tr_{A_{\inf}}$ described above coincide.

    In other words, \cref{mainthm::prismatic_trace}(\ref{mainthm_part::trace_on_perfectoids}) is true. 
\end{proposition}
\begin{proof}
    By Zariski descent we assume $S$ is finite free as an $R$-module. By the results of \cite[Section 6.1]{bms2}, Remark \ref{rmk::bootstrap_from_spectral_trace_to_synthetic_trace}, and multiplicativity of the homotopy-fixed point spectral sequence, the transfer map $\TC^-(S)\to \TC^-(R)$ induces the map of spectral sequences 
    
\[\begin{tikzcd}
	{S[u]\llb t\rrb} & {A_{\inf}(S)[u,t]/(ut-d)} \\
	{R[u]\llb t\rrb} & {A_{\inf}(R)[u,t]/(ut - d)}
	\arrow[Rightarrow, from=1-1, to=1-2]
	\arrow["\tr"', from=1-1, to=2-1]
	\arrow["{\tr^\prism_{S/R}}", from=1-2, to=2-2]
	\arrow[Rightarrow, from=2-1, to=2-2]
\end{tikzcd}\]
    where $d$ is a generator of $\ker(A_{\inf}(R)\to R)$ (and $\ker(A_{\inf}(S)\to S)$), and the vertical maps send $u$ and $t$ to themselves and act on $S$ (resp. $A_{\inf}(S)$) as the transfer maps $\pi_0\tr^\THH$ (resp. $\pi_0\tr^{\TC^-}$). We claim that the map \[\pi_0\tr^\THH: S = \THH_0(S)\to \THH_0(R) = R\]
    coincides with the usual trace map $S\to R$. Indeed, recall that $\THH_0 = \HH_0$ in general, and so this claim is equivalent to the claim that $$\pi_0\tr^\HH:S = \HH_0(S) \to \HH_0(R)=R$$
    is the usual trace map. This latter claim follows by a straightforward calculation unravelling the explicit weak equivalence $\HH(-)\simeq \HH(\Perf(-))$ described in \cite[Section 2]{mccarthy1994_HH_of_category}. 

    By functoriality of the transfer with respect to a map of localising invariants, the map $\tr^\prism_{S/R}$ is also Frobenius equivariant. The usual trace map $\tr_{A_{\inf}}:A_{\inf}(S) \to A_{\inf}(R)$ is also Frobenius equivariant, since one can apply Frobenius to the matrix for multiplication by any element $x\in A_{\inf}(S)$. In particular, the difference \[h:= \tr_{S/R}^\prism - \tr_{A_{\inf}} : A_{\inf}(S)\to A_{\inf}(R)\]
    is a Frobenius equivariant $A_{\inf}(R)$-linear map that is zero modulo $d$. We claim that any such map $h$ has to be zero.  
    Indeed, since $h$ is zero modulo $d$, we have $h(A_{\inf}(S))\subseteq d\cdot A_{\inf}(R)$. Frobenius equivariance then says that $h(A_{\inf}(S))\subseteq \varphi^n(d)\cdot A_{\inf}(R)$ for all $n\in \Z$. We are done by the following \cref{lem::frob_separation_for_perfds}.
\end{proof}
The following lemma is probably well-known but the author could not find a reference in the literature.
\begin{lemma}\label{lem::frob_separation_for_perfds}
    Suppose $R$ is any $p$-torsion free integral perfectoid ring and $d$ a generator of the kernel of Fontaine's map $\theta:A_{\inf}(R)\to R$. Then, \[\bigcap_{n\ge 0} \varphi^n(d)A_{\inf}(R) = 0.\]
\end{lemma}
\begin{proof}
    Let $R^\flat$ be the tilt of $R$, so that $A_{\inf}(R) = W(R^\flat)$ and $R^\flat$ is $\bar d$-adically complete where $\bar d$ is the reduction of $d$ modulo $p$. Note that $\bar d$ is also non-zero divisor in $R^\flat$ by our $p$-torsion-free assumption on $R$. Suppose $x\in \bigcap_{n\ge 0} \varphi^n(d)A_{\inf}(R)$. Reducing mod $p$, one has \[ \bigcap_{n\ge 0} \bar d^{p^n} R^\flat = 0\]
    by $\bar d$-adic completeness of $R^\flat$. In particular, $x = px_1$ for some $x_1\in A_{\inf}(R)$. Next, writing $x = \varphi^n(d)b_n$ implies that $\bar d^{p^n} \cdot b_n = 0$ modulo $p$, and since $\bar d$ is a non-zero divisor we get $b_n\in pA_{\inf}(R)$. As $A_{\inf}(R)$ is $p$-torsion free, we thus have $x_1 = \varphi^n(d)\frac{b_n}{p}$ for all $n$. Thus $\bigcap_{n\ge 0} \varphi^n(d)A_{\inf}(R)$ is $p$-divisible. Since $A_{\inf}(R)$ is $p$-complete, it follows that $\bigcap_{n\ge 0} \varphi^n(d)A_{\inf}(R) = 0$. 
\end{proof}

\subsection{\'Etale Comparison}\label{sctn::etale_comparison_of_trace}

In this section, we prove the compatibility of the prismatic trace with \'etale realisation. 
For the rest of this section, we assume that $f:X\to Y$ is a regular finite locally free map of quasi-syntomic $\Z_p$-flat $p$-adic formal schemes such that the generic fibres $X_\eta$ and $Y_\eta$ are locally Noetherian analytic adic spaces. Then, $f_\eta: X_\eta\to Y_\eta$ is a finite flat map of analytic adic spaces. In this setting, \cite{li_reinecke_zavyalov_trace} construct a trace map $$\tr_f^\et(\mc F): f_{\eta *} f_\eta^* \mc F\to \mc F$$ for any abelian sheaf $\mc F$; we recall their result.

\begin{theorem}[{\cite[Theorem 2.5.6]{li_reinecke_zavyalov_trace}}]
    For any finite flat map $f_\eta:X_\eta\to Y_\eta$ of locally Noetherian analytic adic spaces, there is a unique map $$\tr_f^\et(\mc F): f_{\eta *} f_\eta^* \mc F\to \mc F$$  
    functorial in $\mc F\in \Ab(Y_{\eta,\et})$, compatible with compositions and pull-backs in $f_\eta$, and normalized so that the composition $\mc F\to f_{\eta,*}f^*_\eta\mc F\to \mc F$ is locally given by multiplication by $\deg f$. 
\end{theorem}
\begin{remark}
    \cite{li_reinecke_zavyalov_trace} only construct these \'etale trace maps for finite flat maps of locally Noetherian analytic adic spaces. Using \cref{thm::transfer_map_for_finite_loc_free_extensions}, we may relax the Noetherianity hypothesis at the cost of adding a regularity condition and working with $p$-complete coefficients. 
\end{remark}

 Here, we note that $f_{\eta,*}^\heart$, the underived push-forward, is exact for finite morphisms $f_\eta$ by \cite[Proposition 2.6.3]{huberbook}, so that $f_{\eta,*} \simeq f_{\eta,*}^\heart$.
 
We can extend the construction of trace maps for finite flat maps given in \cite{li_reinecke_zavyalov_trace} to the derived category $\mc D^b(Y_{\eta,\et}, \Ab)$ of \'etale sheaves of bounded complexes of abelian groups. In particular, for any complex of lisse sheaves $\mc L \in \mc D^b_{\mr{lisse}}(Y_\eta, \Z_p)$, we can look at the \'etale sheaves $\mc L/p^n$ for all $n$, and we have a trace map \[\tr_f^\et(\mc L/p^n) : f_{\eta*}f^*_\eta (\mc L/p^n) \to \mc L/p^n\]
compatible for the various transition maps. Taking limits, and passing to the pro-\'etale site for the formation of pushforwards/pullbacks, we get a $\Z_p$-linear trace map \[\tr_f^\et(\mc L): f_{\eta,*} f^*_\eta (\mc L) \to \mc L.\]
In particular, for $\mc L = {\underline \Z_p}_{Y_\eta}$, we have \[\tr_f^\et := \tr_f^\et({\underline \Z_p}_{Y_\eta}) : f_{\eta,*}{\underline \Z_p}_{X_\eta} \to {\underline \Z_p}_{Y_\eta}.\]
By the projection formula, if $\mc L$ is a bounded complex of lisse sheaves, we can identify $\tr_f^\et(\mc L)$ with the map \[f_{\eta,*} f_\eta^*\mc L \simeq f_{\eta,*}({\underline \Z_p}_{X_\eta}) \otimes \mc L \xrightarrow{\tr_f^\et \otimes \id_{\mc L}} \mc L.\]

With this trace map on $p$-adic lisse sheaves for finite morphisms of rigid spaces in hand, we have the following.

\begin{theorem}\label{thm::etale_comparison_of_prismatic_trace}
    Suppose $f:X\to Y$ finite locally free $p$-completely syntomic map of bounded quasi-syntomic $p$-adic formal schemes flat over $\Z_p$. Suppose also that $X_\eta,Y_\eta$ are locally Noetherian. Then, for any prismatic $F$-gauge $\mc E\in \Perf(Y^\Syn)$, the \'etale realisation of the trace map on $F$-gauges \[f_*^\syn f^{\syn, *} \mc E\to \mc E\]
    coincides with the \'etale trace map \[\tr_f^\et(T_\et\mc E) : f_{\eta,*} f_\eta^*\left(T_\et(\mc E)\right) \to T_\et(\mc E)\]
    of \cite{li_reinecke_zavyalov_trace} described above.

    In other words, \cref{mainthm::prismatic_trace}(\ref{mainthm_part::etale_comparison}) holds.
\end{theorem}
\begin{remark}\label{rmk::extending_etale_realisation}
    Suppose $Y$ is some fixed bounded qcqs quasi-syntomic $p$-adic formal scheme flat over $\Z_p$. \'Etale realisation as defined in \cite{Bhatt_FGaugeLec} is a symmetric monoidal exact functor \[T_\et : \Perf(Y^\Syn) \to \mc D_{\mr{lisse}}^b(Y_\eta, \Z_p).\]
    However, for a map $f:X\to Y$ as in \cref{thm::etale_comparison_of_prismatic_trace}, the $F$-gauge $f_*^\syn f^{\syn, *} \mc E$ need not lie in $\Perf(Y^\Syn)$. We thus need to extend the domain of $T_\et$. 
    
    We thus define an exact functor \[T_\et : \mc D_{\mr{geom}}(Y^\Syn) \to \mc D(Y_{\eta, \proet}, \Z_p)\]
    where 
    \begin{itemize}
        \item the source category is the thick subcategory of $\mc D_\qcoh(Y^\Syn)$ generated by $\Perf(Y^\Syn)$ and sheaves of the form $f_*^\syn f^{\syn,*} \mc E$ as $f$ runs over all possible quasi-syntomic maps $f:X\to Y$ from $X$ a bounded qcqs quasi-syntomic $p$-adic formal scheme flat over $\Z_p$ and $\mc E\in \Perf(Y^\Syn)$, and 
        \item the target category is the category of $\Z_p$-sheaves on the pro-\'etale site of $Y_\eta$.
    \end{itemize}

    One defines, for any $\mc E\in \mc D_{\mr{geom}}(Y^\Syn)$, the underlying pro-\'etale sheaf of $T_\et(\mc E)$ on $Y_\eta$ to be given on a perfectoid test object $(R[\tfrac 1p], R)$ (where $\Spf R$ maps to $Y$) by \[\left(\mc E(A_{\inf}(R))\left[\frac{1}{d}\right]^\wedge_p\right)^{\varphi = \id}\]
    where $d$ generates the kernel of $\theta: A_{\inf}(R)\to R$ and $\mc E(A_{\inf}(R))$ denotes the pull-back of $\mc E$ along the syntomification of the map $\Spf R \to Y$. This is compatible with the usual \'etale realisation $\Perf(Y^\Syn)\to \mc D^b_{\mr{lisse}}(Y_\eta, \Z_p)$ by construction.

    Essentially by definition and \cite[Theorem 6.1]{guo_reinecke}, if $\mc E = f_*^\syn \mc E'$ for $\mc E'\in \Perf(X^\Syn)$ (where $f:X\to Y$ is a quasi-syntomic map), then \[T_\et(\mc E) \simeq f_{\eta, *} T_\et(\mc E')\]
    as pro-\'etale sheaves on $Y_\eta$.
\end{remark}

\begin{remark}
    By following the proof of \cite[Lemma 7.15]{guo_reinecke}, and using \cref{prop::compatibility_for_trace_with_perfectoids} and the proof of the \'etale comparison, one checks rather easily that \cref{mainthm::prismatic_trace}(\ref{mainthm_part::finite_etale_comparison}) is true. 
\end{remark}

We now go on to prove \cref{thm::etale_comparison_of_prismatic_trace}, for which we record some lemmas. 

First, recall that for an adic space $X$, a geometric point $\bar x$ of $X$ is a map $\bar x: \Spa(C,C^+)\to X$ where $C$ is a complete algebraically closed nonarchimedean field over $\Q_p$ and $C^+$ is an open bounded valuation subring of $C$. For $\mc F$ a sheaf of abelian groups on the \'etale site of an adic space $X$ and for $\bar x :\Spa(C,C^+)\to X$ a geometric point of $X$, the stalk $\mc F_{\bar x}$ of $\mc F$ at $\bar x$ is defined to be \[\mc F_{\bar x}:= \R\Gamma(\Spa(C,C^+), \bar x^*\mc F)\in \Ab.\]
The following lemma is \cite[Lemma 2.5.1]{li_reinecke_zavyalov_trace}.
\begin{lemma}\label{lem::check_equality_on_geometric_points_on_etale_site}
    For $f:X\to Y$ as above, $\bar y$ a geometric point of $Y_\eta$, and $\mc F$ a sheaf of abelian groups on $Y_\et$. Then there is a natural isomorphism \[\bigoplus_{\bar x\in f_\eta^{-1}(\bar y)} \mc F_{\bar y} \xrightarrow{\simeq} (f_{\eta,*}f_\eta^*\mc F)_{\bar y}.\]
\end{lemma}
\begin{corollary}
    For $f:X\to Y$ as above and $\bar y$ a geometric point of $Y_\eta$, there is a natural isomorphism \[\bigoplus_{\bar x\in f_\eta^{-1}(\bar y)} {\underline \Z_p}_{\bar y} \xrightarrow{\simeq} (f_{\eta,*}{\underline \Z_p}_{X_\eta})_{\bar y}\]
    of sheaves on the pro-\'etale site of $\bar y$. 
\end{corollary}
\begin{proof}
    Take $\mc F = {\underline{\Z/p^n\Z}}_{X_\eta}$ in the lemma and then take the limit over $n$. 
\end{proof}

\begin{lemma}\label{lem::check_equality_on_geometric_points}
    Suppose $f,g:\mc L\to \mc L'$ are two maps of pro-\'etale discrete (i.e. concentrated in cohomological degree 0) $\Z_p$-sheaves on an analytic adic space $X$. Then, $f = g$ if and only if for all geometric points $\bar x$ of $X$, we have $f_{\bar x} = g_{\bar x}$ as maps of stalks $\mc L_{\bar x}\to \mc L'_{\bar x}$. 
\end{lemma}
\begin{proof}
    It suffices to prove this for the map $\mc L/p^n \to \mc L'/p^n$ of sheaves for all $n$; in particular, we may reduce to working on the \'etale site rather than the pro-\'etale site. In this case, \cite[Proposition 2.5.5]{huberbook} implies that the collection of functors $\mc F\mapsto \mc F_{\bar x}$ as $\bar x$ ranges over all geometric points of $X$ is jointly conservative. As this functor is valued in an abelian category (we restricted ourselves to discrete sheaves), this implies that the stalk functors are jointly fully faithful, as required. 
\end{proof}

The following lemma was used implicitly in the proof of \cite[Theorem 2.5.6]{li_reinecke_zavyalov_trace}, and we will need it ourselves.
\begin{lemma}\label{lem::reduction_of_finite_connected_adic_spaces_over_C}
    Let $\mc X$ be a connected adic space with a finite map $f:\mc X\to \Spa(C,C^+)$, where $C^+$ is an open integrally closed subring of an algebraically closed extension $C$ of $\Q_p$. Then $\mc X_{\mr{red}} \simeq \Spa(C,C^+)$.
\end{lemma}
\begin{proof}
    Since $\Spa(C,C^+)$ is already reduced, we see that $f$ factors through a finite map $\mc X_{\mr{red}} \to \Spa(C,C^+)$. \cite[Lemma 2.1.3(i)]{li_reinecke_zavyalov_trace} then implies that $\mc X_{\mr{red}} = \Spa(K,K^+)$ for $K$ a finite extension of $C$ and $K^+$ the integral closure of $C^+$ in $K$. Since $C$ is algebraically closed, we thus have $K = C$ and $K^+=C^+$ as required. 
\end{proof}

We finally prove \cref{thm::etale_comparison_of_prismatic_trace}. The idea of the proof is based on the proof of uniqueness of the finite flat trace map in \cite[Theorem 2.5.6]{li_reinecke_zavyalov_trace}.

\begin{proof}[Proof of \cref{thm::etale_comparison_of_prismatic_trace}]
    By the projection formula it suffices to prove this for $\mc E = \O_{Y^\Syn}$, whose \'etale realisation is the local system $\underline \Z_p$ on $Y_\eta$. We have two maps \[f_{\eta,*}  {\underline \Z_p}_{X_\eta} \to {\underline \Z_p}_{Y_\eta},\]
    one of which comes from the \'etale trace and one which comes from the \'etale realization of the prismatic trace.
    By \cref{lem::check_equality_on_geometric_points}, we reduce to checking the case of a geometric point $Y = \Spf C^+$ (so $Y_\eta = \Spa(C,C^+)$) for $C$ a complete algebraically closed non-archimedean field over $\Q_p$ and $C^+$ an open bounded valuation subring (see also the proof of uniqueness of \cite[Theorem 2.5.6]{li_reinecke_zavyalov_trace}). Recall from \cref{sctn::zoo_of_coh_thries} that $$\R\Gamma_\et(\Spa(R[\tfrac 1p],R), \Z_p) \simeq \gr^0_\mot L_{K(1)}\TC(R)$$ for any $p$-complete ring $R$, and that this is compatible with the prismatic \'etale realisation. It thus suffices to check that the \'etale trace \[\R\Gamma_\proet(X_\eta, \Z_p) \to \Z_p\]
    coincides with the abstract transfer map \[\tr_\et'(X_\eta) := \gr^0_\mot  \tr^{L_{K(1)}\TC}_f : \R\Gamma_\proet(X_\eta, \Z_p)\simeq \gr^0_\mot L_{K(1)}\TC(X)\to \gr^0_\mot L_{K(1)}\TC(Y)\simeq \Z_p.\]
    Since the motivic filtration on $L_{K(1)}\TC(X)\simeq L_{K(1)}K(X_\eta)$ is actually an invariant of $X_\eta$ (since it identifies with the Thomason filtration), this abstract transfer map 
    is actually functorial in the analytic adic space $X_\eta$. In particular, it follows that for $X_\eta = \bigsqcup \mc X_i$ a connected component decomposition of the analytic adic space $X_\eta$ (which need not come from a decomposition of the formal scheme $X$) we have $\tr_\et'(X_\eta) = \sum_i \tr'_\et(\mc X_i)$ as maps \[\R\Gamma_\et(X_\eta, \Z_p) \simeq \bigoplus_i \R\Gamma_\et(\mc X_i, \Z_p) \to \Z_p. \]
    Of course, this is also true of the \'etale trace maps of \cite[Section 2.5]{li_reinecke_zavyalov_trace}. 

    We have thus reduced to the case that $\mc X = \Spa(S, S^+)$ is a connected finite flat analytic adic space over $Y_\eta = \Spa(C,C^+)$, at the cost of losing that the map is the generic fibre of a finite flat map of formal schemes. \cref{lem::reduction_of_finite_connected_adic_spaces_over_C} says that $\mc X_{\mr{red}} = \Spa(C,C^+)$. The topological invariance of the \'etale site then implies that \[\R\Gamma_\et(\mc X,\Z_p) \simeq \R\Gamma_\et(\mc X_{\mr{red}},\Z_p) \simeq \R\Gamma_\et(\Spa(C,C^+),\Z_p) \xleftarrow\simeq \Z_p,\]
    where the last arrow is the canonical map $\Z_p \to \R\Gamma_\et(-,\Z_p)$. In particular, the \'etale trace map of \cite[Theorem 2.5.6]{li_reinecke_zavyalov_trace} identifies with the map \[\Z_p \simeq \R\Gamma_\proet(\mc X, \Z_p) \to \Z_p\]
    given by multiplication by $d:= \mr{rank}_{C}\O(\mc X)$. This is of course also true of the trace map $\tr'_\et(\mc X)$ coming from applying $\gr^0_\mot$ to the transfer map $L_{K(1)}K(\mc X) \to L_{K(1)}K(Y_\eta)$.
\end{proof}

\section{Prismatic Trace for the Cyclotomic Tower}\label{sctn::trace_for_cyclotomic_tower}

In this section we finally prove all of our main theorems except the already established \cref{mainthm::prismatic_trace}. 

Throughout this section, we will use the notation $q_n := q^{p^{-n}}$ as well as the notation introduced in \cref{sctn::thh_of_cyclotomic_extns} without comment. Before we can go on to actually prove \cref{mainthm::prismatic_trace}, we make a few remarks since \cref{thm::transfer_map_for_finite_loc_free_extensions} is \textit{a priori} not strong enough for our purposes.

Fix $f:X\to Y$ a finite locally free regular map. In \cref{thm::transfer_map_for_finite_loc_free_extensions}, we have shown that the transfer map $\THH(X)\to \THH(Y)$ upgrades canonically to a map \[\tr_f^\THH: \Fil_\mot^\bullet \THH(X)\to \Fil_\mot^\bullet \THH(Y)\]
of filtered spectra. However, we will crucially need that $\Fil_\mot^\bullet \THH$ takes values in cyclotomic synthetic spectra, and \textit{a priori} \cref{thm::transfer_map_for_finite_loc_free_extensions} does not construct $\tr_f^\THH$ as a map of cyclotomic synthetic spectra. In the following remark, let us explain how the construction of \cref{thm::transfer_map_for_finite_loc_free_extensions} automatically upgrades to a map of cyclotomic synthetic spectra. 
    
\begin{remark}\label{rmk::bootstrap_from_spectral_trace_to_synthetic_trace}
    Recall that a synthetic spectrum is a $\S_\ev$-module in $\Fil \Sp$, and that a synthetic spectrum with synthetic circle action is a $\T_\ev$-module in synthetic spectra. In particular, $\Fil_\mot^\bullet \THH(Y)$ being an $E_\infty$-ring object in synthetic spectra with synthetic circle action, it is in fact an $E_\infty$-algebra over $\T_\ev$ in $\Fil \Sp$. By $\Fil_\mot^\bullet \THH(Y)$-linearity of the filtered transfer map $\tr_f^\THH$, the map $\tr_f^\THH$ is a $\T_\ev$-module map in filtered spectra. Thus $\tr_f^\THH$ naturally upgrades to a map of synthetic spectra with synthetic circle action. The same reasoning applies to $\HH$, so that the filtered transfer map \[\tr_f^\HH: \Fil_\mot^\bullet \HH(X)\to \Fil_\mot^\bullet \HH(Y),\]
    \textit{a priori} only a map of filtered spectra, is naturally a map of synthetic spectra with synthetic circle action.

    For the cyclotomic Frobenius, note that for any qcqs quasi-syntomic $p$-adic formal scheme $X$, one has \[\left(\Fil_\mot^\bullet \THH(X)\right)^{tC_{p,\ev}} \simeq \R\Gamma\left(\QSyn_X, \tau_{\ge 2*} \left(\THH(R)^{tC_p}\right)\right)\]
    as a consequence of quasi-syntomic descent to QRSPs and \cite[Lemma 3.18]{antieau2024cyclotomic}. In particular, we can define $$\Fil_\mot^\bullet \THH(X)^{tC_p} := \left(\Fil_\mot^\bullet \THH(X)\right)^{tC_{p,\ev}},$$ 
    and this promotes $\Fil_\mot^\bullet \THH^{tC_p}$ to a sheaf valued in synthetic spectra with synthetic circle action.
    One checks rather easily that $\Fil_\mot^\bullet \THH^{tC_p}$ now satisfies Assumption \ref{ass::nice_filtered_localising_invariants}. In particular, \cref{thm::transfer_map_for_finite_loc_free_extensions} and the previous argument yields a transfer map \[\tr_f^{\THH^{tC_p}}: \Fil_\mot^\bullet \THH(X)^{tC_p}\to \Fil_\mot^\bullet \THH(Y)^{tC_p}\]
    of synthetic spectra with synthetic circle action natural in $f$. Moreover, since the synthetic cyclotomic Frobenius is just the evenly filtered cyclotomic Frobenius $\THH\to \THH^{tC_p}$ (at least on affines), the functoriality-in-$\mc F$ of \cref{thm::transfer_map_for_finite_loc_free_extensions} implies the existence of a canonical commutative diagram 
    \[\begin{tikzcd}
        \Fil_\mot^\bullet \THH(X) \ar{r}{\varphi_X} \ar{d}{\tr_f^\THH} & \Fil_\mot^\bullet \THH(X)^{tC_p} \ar{d}{\tr_f^{\THH^{tC_p}}} \ar[description, draw=none]{r}{\simeq} & \left(\Fil_\mot^\bullet \THH(X)\right)^{tC_{p,\ev}}\ar{d}{(\tr_f^\THH)^{tC_{p,\ev}}}\\
        \Fil_\mot^\bullet \THH(Y) \ar{r}{\varphi_X}  & \Fil_\mot^\bullet \THH(Y)^{tC_p} \ar[description, draw=none]{r}{\simeq} & \left(\Fil_\mot^\bullet \THH(Y)\right)^{tC_{p,\ev}}.
    \end{tikzcd}\]
    Here, the identification $\tr_f^{\THH^{tC_p}} = (\tr_f^{\THH})^{tC_{p,\ev}}$ holds by the assertion on functoriality in the filtered sheaf of \cref{thm::transfer_map_for_finite_loc_free_extensions}. Hence, $\tr_f^\THH$ upgrades to a map of cyclotomic synthetic spectra. 
\end{remark}

Recall next that \cref{thm::transfer_map_for_finite_loc_free_extensions} does not give a canonical homotopy $\tr_{g\circ f}^{\mc F} \simeq \tr_g^{\mc F} \circ \tr_f^{\mc F}$ where $X\xrightarrow{f} Y \xrightarrow{g} Z$ are two composable finite locally free regular maps. We make a remark to get around this technical deficiency.

\begin{remark}
    For the rest of this paper we will be dealing with the transfer maps in a tower \[\cdots \to X_3\to X_2\to X_1.\] 
    We will not have need to deal with the filtered transfer map along $X_{n+2}\to X_n$ directly. Rather, we will only remember the filtered transfer maps $\tr_n^{\mc F} : \Fil^\bullet \mc F(X_{n+1})\to \Fil^\bullet \mc F(X_n)$ coming from \cref{thm::transfer_map_for_finite_loc_free_extensions}, and for $\Fil^\bullet \mc F(X_{n+k})\to \Fil^\bullet \mc F(X_n)$ we will replace $\tr_{X_{n+k}\to X_n}^{\mc F}$ with the composition $\tr_{n+k-1}^{\mc F} \circ \cdots \circ \tr_{n}^{\mc F}$. This is fine since we only care about computing the limit. Since the above $k$-fold composition is (non-canonically) homotopic to the true transfer map, the limit will not change up to homotopy. 

    Notice however that \cref{thm::transfer_map_for_finite_loc_free_extensions} does guarantee that the filtered transfer $\tr^{\mc F}$ is functorial in pull-backs, so we only actually need to make this choice once and for all for the tower \[\cdots \to \Spf \Z_p[\zeta_{p^{n+1}}] \to \Spf \Z_p[\zeta_{p^n}]\to \cdots .\]
\end{remark}

\subsection{Preliminary Calculations}\label{sctn::prelim_calculations}

In the following, we will need the following standard computation of the homotopy groups of (various $S^1$-fixed points of) $\ku_{p,n}$. Since this result is standard, we only sketch a proof. 

\begin{lemma}\label{lem::computing_htpy_gps_of_fixed_pts_of_kupn}
    For $n\ge 0$, we have 
    \begin{align*}
        \pi_* \ku_{p,n} &\cong \left(\Z_p\llb q_n - 1\rrb/(q-1)\right)[\beta],\\
        \pi_* \ku_{p,n}^{hS^1} &\cong \Z_p \llb q_n - 1\rrb[\beta, t]/(\beta t - (q-1)),\\
        \pi_*\ku_{p,n}^{tS^1} &\cong \Z_p \llb q_n - 1\rrb[t^{\pm 1}]\\
        \pi_* \ku_{p,n}^{tC_p} &\cong \Z_p[\zeta_{p^{n+1}}][t^{\pm 1}] \\
        \pi_* (\ku_{p,n}^{(-1)})^{hS^1} &\cong \Z_p\llb q_{n+1} - 1\rrb[u, t]/(u t - [p]_{q_1}).
    \end{align*}
    where $|\beta| = |u| = 2, |q| = 0$, and $|t|=-2$. Moreover, the following are true:
    \begin{enumerate}
        \item The Adams operations corresponding to $g\in (1+p^{n+1}\Z_p)$ acts on $\ku_{p,n}^{hS^1}$ via \[\beta \mapsto g\beta, \quad q \mapsto q^g, \quad \text{and} \quad t \mapsto \frac{[g]_q}{g}t,\] 
        where $[g]_q := \frac{q^g-1}{q-1}$.
        \item The ring map $\ku_{p,n}^{hS^1}\to (\ku_{p,n}^{(-1)})^{hS^1}$ is given on homotopy by the graded algebra map \[\Z_p\llb q_n-1\rrb[\beta, t]/(\beta t - (q-1)) \to \Z_p\llb q_{n+1} -1\rrb[u,t]/(ut - [p]_{q_1}) \]
        sending $q_n \mapsto q_{n+1}^p$, $t\mapsto t$, and $\beta \mapsto (q_1-1)u$. 
        \item The Frobenius map $\varphi^{hS^1} : (\ku_{p,n}^{(-1)})^{hS^1} \to (\ku_{p,n}^{(-1)})^{tS^1}$ is given on homotopy by the graded algebra map \[\Z_p\llb q_{n+1}-1\rrb[u, t]/(u t - [p]_{q_1}) \to \Z_p\llb q_{n+1} -1\rrb[t^{\pm 1}]\]
        sending $q_{n+1} \mapsto q_{n+1}^p$, $t\mapsto [p]_{q} t$, and $u\mapsto t^{-1}$. In particular, the image of the Bott class is `fixed' by Frobenius; more precisely, $\varphi^{hS^1}(\beta) = \can(\beta)$. 
    \end{enumerate}
\end{lemma}
\begin{proof}[Proof Sketch]
    A standard result says that $\ku_p^{hS^1} \simeq \Z_p\llb q-1\rrb[\beta, t]$, and that \[\ku_p^{hC_{p^n}} \simeq \Z_p\llb q-1\rrb[\beta, t]/([p^n]_qt). \]
    This yields the first isomorphism. The second isomorphism follows from evenness (which gives the description as a graded abelian group by degeneration of the homotopy fixed point spectral sequence) and the fact that $\ku_p^{hS^1} \to \ku_{p,n}^{hS^1}$ are ring maps which are injective on homotopy groups in order to identify the multiplicative structure. The third and fourth equivalences are immediate from the second by standard facts about complex orientations. The fifth isomorphism as well as statement (2) follow in the same way as the second isomorphism, though one needs to use the ring map $\ku_{p,n}^{hS^1} \to (\ku_{p,n}^{(-1)})^{hS^1}$ (an injection on homotopy groups) to identify the multiplicative structure. The claim on Adams operations is a standard fact about Adams operations on $\ku_p^{hS^1}$, while the second statement is clear from unravelling identifications.
    
    The final claim comes about by taking $S^1$-fixed points of the cyclotomic Frobenius map $\varphi^{hS^1}:\ku_{p,n}^{(-1)} \to \ku_{p,n}^{(-1),tC_p}$. By construction of the cyclotomic Frobenius, we see that $\varphi^{hS^1}$ factors as \[\ku_{p,n}^{(-1),hS^1} \to (\ku_{p,n}^{tC_p})^{hS^1} \simeq \ku_{p,n}^{tS^1} \to \ku_{p,n}^{(-1),tS^1}.\]
    The claim follows by computing the image of $q_n$ and $t$ under the map \[\frac{\Z_p\llb q_{n+1}-1\rrb[u,t]}{ut - [p]_{q_1}} \cong \pi_*\ku_{p,n}^{(-1),hS^1}  \to \pi_*(\ku_{p,n}^{tC_p})^{hS^1} \simeq \pi_*\ku_{p,n}^{tS^1} \cong \Z_p\llb q'-1\rrb[(t')^{\pm 1}],\]
    using that $q_{n+1}$ maps to $q'$ and that the Euler class $t$ on the source is sent to the $p$-fold iteration of the formal group law on $t'$ under the identification $(\ku_{p,n}^{tC_p})^{hS^1} \simeq \ku_{p,n}^{tS^1}$.
\end{proof}

\begin{corollary}\label{lem::fibre_seq_computing_jpn_from_kupn}
    There is a fibre sequence \[ j_{n} \to \ku_{p,n} \xrightarrow{\psi^{1+p^{n+1}}_\circ} \tau_{\ge 2} \ku_{p,n}\]
    where $\psi^{1+p^{n+1}}_\circ$ is the unique map factorizing $$\psi^{1+p^{n+1}} - 1 : \ku_{p,n} \to \ku_{p,n}$$
    through the connective cover map $\tau_{\ge 2}\ku_{p,n}\to \ku_{p,n}$. 
\end{corollary}
\begin{proof}
    For $\psi^{1+p^{n+1}}_\circ$ to exist, notice that the Adams operation $\psi^{1+p^{n+1}}$ acts trivially on $\pi_0(\ku_{p,n}) = \Z_p\llb q_n - 1\rrb/(q-1)$, since \[\psi^{1+p^{n+1}}(q_n) = q^{(1+p^{n+1})/p^n} = q_n\cdot q^p \equiv  q_n \modc{q-1}.\]
    As $1+p^{n+1}\Z_p$ is topologically generated by $1+p^{n+1}$, we see that we have the fibre sequence \[\ku_{p,n}^{h(1+p^{n+1})\Z_p} \to \ku_{p,n} \xrightarrow{\psi^{1+p^{n+1}} - 1} \ku_{p,n}.\]
    Taking connective covers and using the fibre sequence \[\tau_{\ge 2} \ku_{p,n} \to \ku_{p,n} \to \pi_0\ku_{p,n}\]
    then yields the lemma. 
\end{proof}

For the remaining corollaries, we introduce the following notation. 
\begin{notation}\label{notn::spectral_Tate_twist}
    For any $\Z_p^\times$-equivariant module $M$ over $\ku_p$, write \[M(1)[2] := \tau_{\ge 2}\ku_p \otimes_{\ku_p} M.\]
\end{notation}
This notation is reasonable since, ignoring $\Z_p^\times$-actions, there is an equivalence $\tau_{\ge 2}\ku_p \simeq \ku_p[2]$ as a consequence of Bott periodicity. Note also that \[\tau_{\ge 2}\ku_{p,n} \simeq \ku_{p,n}(1)[2].\]

\begin{corollary}\label{cor::htpy_groups_of_tau_>=2ku_as_ideal}
    There is an equivalence $$\tau_{\ge 2}((\tau_{\ge 2}\ku_{p,n})^{tC_p}) \simeq \ku_{p,n}^{(-1)}(1)[2].$$
    The homotopy groups of this $\ku_{p,n}^{(-1)}$-module is the (graded) ideal \[(\zeta_p-1)u \cdot \Z_p[\zeta_{p^{n+1}}][u]\]
    inside $\pi_*\ku_{p,n}^{(-1)} \simeq \Z_p\llb \zeta_{p^{n+1}}[u]$, where the inclusion of the ideal is induced by \[\tau_{\ge 2}((\tau_{\ge 2}\ku_{p,n})^{tC_p}) \to \tau_{\ge 2}(\ku_{p,n}^{tC_p}) \to \ku_{p,n}^{(-1)}.\]
    
    Similarly, the homotopy groups of the $\ku_{p,n}^{(-1)}$-module $$(\tau_{\ge 2}((\tau_{\ge 2}\ku_{p,n})^{tC_p})^{hS^1} \simeq \ku_{p,n}^{(-1),hS^1}(1)[2]$$ is the (graded) ideal \[(q_1-1)u \cdot \Z_p\llb q_{n+1} - 1\rrb[u,t]/(ut - [p]_{q_1}).\]
\end{corollary}
\begin{proof}
    The first equivalence follows in the same way as \cite[Lemma 6.2.5]{sanath_thesis}, using that \[\ku_{p,n}^{(-1)} \simeq \ku_{p,n} \otimes_{\ku_p} \ku_p^{(-1)}.\] The rest follow from \cref{lem::computing_htpy_gps_of_fixed_pts_of_kupn} since multiplication by the Bott class induces the equivalence $\tau_{\ge 2}\ku_{p,n} \simeq \ku_{p,n}[2]$. 
\end{proof}

One proves the following corollary similarly to \cref{lem::fibre_seq_computing_jpn_from_kupn}.
\begin{corollary}\label{cor::-1-twist_of_fibre_seq_computing_jpn}
    For any $g\in 1+p^{n+1}\Z_p$ a topological generator, there is a fibre sequence \[ j_{n}^{(-1)} \to \ku_{p,n}^{(-1)} \xrightarrow{\psi^{g}_{\circ\circ}} \ku_{p,n}^{(-1)}(1)[2].\]
    where $\psi^{g}_{\circ\circ}$ is induced by $\psi^{g}_\circ$.

    In particular, using the equivalences of \cref{sctn::thh_of_cyclotomic_extns}, we have a fibre sequence of cyclotomic spectra 
    \[\THH(\Z_p[\zeta_{p^n}]) \to \THH(\Z_p[\zeta_{p^n}]/\S_p\llb q_n-1\rrb) \xrightarrow{\psi^g_{\circ\circ}} \THH(\Z_p[\zeta_{p^n}]/\S_p\llb q_n-1\rrb)(1)[2]\]
    where now $g$ is a topological generator of $1+p^n\Z_p$. 
\end{corollary}
\begin{remark}
    There is a generalisation of this to arbitrary quasi-syntomic $\Z_p[\zeta_{p^n}]$-algebras, given in \cref{prop::prismatic_coh_from_q_de_rham_coh}. 
\end{remark}

The following is a generalisation of \cite[Lemma 6.4.10]{sanath_thesis} and may be proven similarly. 
\begin{lemma}\label{lem::some_pushouts_of_cyc_spectra}
    There are pushout squares of cyclotomic $E_\infty$-ring spectra 
    \[\begin{tikzcd}
        \S_p[Bp^{-n}\Z]^\triv \ar{r}\ar{d} & \S_p[BC_{p^n}]^\triv\ar{r}\ar{d}  &  j_{n-1}^{(-1)} \ar{d}\\
        \S_p^\triv \ar{r} & \S_p[B^2\Z]^\triv \ar{r} & \ku_{p,n-1}^{(-1)} 
        \cocartesian{1-1}{2-2}
        \cocartesian{1-2}{2-3}
    \end{tikzcd}\]
    where:
    \begin{itemize}
        \item the left square is induced by applying $\S_p[-]$ to the cofibre sequence of spectra \[p^{-n}\Z [1] \to C_{p^n}[1] \to \Z[2],\]
        \item the map $\S_p[BC_{p^n}]^\triv \to  j_{n-1}^{(-1)}$ is the composition \[\S_p[BC_{p^n}]^\triv \to \THH(\S_p[C_{p^n}]) \to \THH(\Z_p[\zeta_{p^n}]) \simeq  j_{n-1}^{(-1)}\]
        where the first map is the assembly map, and
        \item the map $\S_p[B^2\Z]^\triv \to \ku_{p,n-1}^{(-1)}$ is adjoint to the map $\Z \to \pi_2(\ku_{p,n-1}^{(-1),hS^1}) \cong \Z_p\llb q_n-1\rrb \cdot u$ sending $1\in \Z$ to $\beta = (q_1-1)u$. 
    \end{itemize}
\end{lemma}

If one is just interested in describing $\Fil_\mot^\bullet \THH(X_n)$ just as a spectrum with $S^1$-action, the above lemmas suffice. The rest of this subsection is now devoted to far more technical lemmas that we will need just to pin down Frobenius.

For a concrete interpretation of the following result in terms of prismatic cohomology, see Remark \ref{rmk::comparison_bw_zeyuliu_and_our_qdR_comp}.  

\begin{proposition}\label{lem::S1-nilp-of-tilde-jpn-square}
    Fix $n\ge 1$ and $g$ a topological generator of $1+p^{n+1}\Z_p$. For any bounded below $S^1$-equivariant $ j_0^{(-1)}$-module $M$, the following commuting square is a fibre square of spectra with $S^1$-action.
    \[\begin{tikzcd}
        \left(M\otimes_{ j_0} \ku_{p,n-1}^{(-1)}\right)^{tC_p} \ar{r}{\psi_{\circ\circ}^g} \ar{d} & \left(M\otimes_{ j_0} \ku_{p,n-1}^{(-1)}\right)^{tC_p}(1)[2] \ar{d} \\
        \left(M\otimes_{ j_0} \ku_{p,n}^{(-1)}\right)^{tC_p} \ar{r}{\psi_{\circ\circ}^g} & \left(M\otimes_{ j_0} \ku_{p,n}^{(-1)}\right)^{tC_p}(1)[2].
    \end{tikzcd}\]
\end{proposition}
\begin{proof}
    Let $F$ be the total cofibre of the following commutative square of $ j_0^{(-1)}$-modules in cyclotomic spectra
    \begin{equation}\label{eqn::S1-nilp-fibre-square}
        \begin{tikzcd}
            \ku_{p,n-1}^{(-1)} \ar{r}{\psi_{\circ\circ}^g} \ar{d} & \ku_{p,n-1}^{(-1)}(1)[2] \ar{d} \\
            \ku_{p,n}^{(-1)} \ar{r}{\psi_{\circ\circ}^g} & \ku_{p,n}^{(-1)}(1)[2].
        \end{tikzcd}
    \end{equation}
    We claim that $F/(p,\beta)$ is $S^1$-nilpotent in the sense of \cite[Definition 3.1.2]{sanath_arpon_image_of_j}, i.e. it is contained in the thick-tensor ideal of $\Mod_{\S_p^\triv}(\Sp^{BS^1})$ generated by $\S_p[S^1]$. 
    
    Suppose we know the $S^1$-nilpotence of $F/(p,\beta)$. Then $X:= M\otimes_{ j_0^{(-1)}} F$ is also such that $X/(p,\beta)$ is $S^1$-nilpotent, which by \cref{lem::tate_Cp_kills_induced_S1_spectra} below implies that \[\left(X/(p,\beta)\right)^{tC_p} \simeq X^{tC_p}/(p,\beta) \simeq 0.\]
    This implies that $(X^{tC_p})^\wedge_\beta \simeq 0$. Using the same argument as in \cite[Remark 3.1.6]{sanath_arpon_image_of_j}, noting that $X$ is bounded below (since $M$ is), the natural map $X^{tC_p} \to (X^{tC_p})^\wedge_\beta$ is an equivalence, so that $X^{tC_p}\simeq 0$. Since $X^{tC_p}$ is the total cofibre of the square occuring in the statement of \cref{lem::S1-nilp-of-tilde-jpn-square}, we win.

    It remains to establish the $S^1$-nilpotence of $F/(p,\beta)$. By \cite[Lemma 3.4.2]{sanath_arpon_image_of_j} we see that $F/(p,\beta)$ is an $S^1$-equivariant $\F_p$-module. In particular, the criterion of \cite[Proposition 3.1.5]{sanath_arpon_image_of_j} applies, i.e. it suffices to check that the Euler class $t$ acts nilpotently on $(F/(p,\beta))^{hS^1} \simeq F^{hS^1}/(p,\beta)$. One computes the following:
    \begin{align*}
        \pi_*\left(\ku_{p,n-1}^{(-1),hS^1}/(p,\beta)\right) &= \left(\bigoplus_{i\ge 1} \frac{\F_p\llb q_n-1\rrb}{(q_1-1)}\cdot u^i \right) \oplus \left(\bigoplus_{i\ge 0} \frac{\F_p\llb q_n-1\rrb}{q-1} \cdot t^i\right),\\
        \pi_*\left(\ku_{p,n-1}^{(-1),hS^1}(1)[2]/(p,\beta)\right) &= \left(\bigoplus_{i\ge 2} \frac{(q_1-1)\F_p\llb q_n-1\rrb}{(q_1-1)^2}\cdot u^i \right) \oplus \beta\left(\frac{\F_p\llb q_n-1\rrb}{(q-1)}\right) \oplus \left(\bigoplus_{i\ge 0} \frac{(q-1)\F_p\llb q_n-1\rrb}{(q-1)^2} \cdot t^i\right),
    \end{align*}
    and similarly for $n$ replaced by $n+1$. Moreover, the vertical maps in the mod $(p,\beta)$-reduction of (\ref{eqn::S1-nilp-fibre-square}) are simply the obvious inclusion maps induced by $\F_p\llb q_n-1\rrb \inj \F_p\llb q_{n+1}-1\rrb$. 
    
    Recall the identity $ut = [p]_{q_1}$ inside $\ku_{p,n}^{(-1),hS^1}$. Since $[p]_{q_1}=0$ modulo $(p,q_1-1)$, multiplication by $t$ is zero on $\pi_*(-)$, $*\ge 4$, for every term in the mod $(p,\beta)$-reduction of (\ref{eqn::S1-nilp-fibre-square}). In particular, multiplication by $t$ is zero for $\pi_* F^{hS^1}/(p,\beta)$ for $*\ge 4$. To then see that $t$ acts nilpotently on all of $\pi_*F^{hS^1}/(p,\beta)$ (with uniform exponent), it suffices to show that $\pi_* F^{hS^1}/(p,\beta) = 0$ for $*\le 1$.
    
    To see this vanishing claim, since the fibres of the horizontal maps in (\ref{eqn::S1-nilp-fibre-square}) are up to homotopy independent of the choice of $g$ (they are given by $ j_{n-1}^{(-1)}$ and $ j_n^{(-1)}$ respectively), we may choose $g = 1+p^{n+1}$. Using \cref{lem::computing_htpy_gps_of_fixed_pts_of_kupn}, the bottom horizontal map in (\ref{eqn::S1-nilp-fibre-square}) acts as 
    \[\psi_{\circ\circ}^{1+p^{n+1}}(q_{n+1}^i t^j) =   q_{n+1}^i \left(\frac{q^i[1+p^n]_q^j - 1}{q-1}\right)\cdot (q-1) t^{j}  \]
    where the term $q_{n+1}^i \left(\frac{q^i[1+p^n]_q^j - 1}{q-1}\right)$ is viewed as living in $\F_p\llb q_{n+1}-1\rrb/(q-1)$. If $p|i$, then this also gives a formula for the top horizontal map in (\ref{eqn::S1-nilp-fibre-square}). Now, we calculate \[\frac{[1+p^n]_q-1}{q-1}=\frac{[1+p^n]_q-(1+p^n)}{q-1} = \sum_{k=0}^{p^n} \frac{q^k-1}{q-1} = \sum_{k=0}^{p^n} k = p^n \cdot \frac{p^n+1}{2} = 0 \in \frac{\F_p\llb q_{n+1}-1\rrb}{(q-1)},\]
    where we use the fact that $\frac{q^k-1}{q-1} = [k]_q \equiv k \modc{q-1}$. 
    This implies  \[q_{n+1}^i \left(\frac{q^i[1+p^n]_q^j - 1}{q-1}\right) = q_{n+1}^i [i]_q = iq_{n+1}^i\]
    in $\F_p\llb q_{n+1}-1\rrb/(q-1)$, i.e. we have shown that \[\psi^{1+p^{n+1}}_{\circ\circ}(q_{n+1}^i t^j) = iq_{n+1}^i  \cdot (q-1) t^j \in q_{n+1}^i \frac{(q-1)\F_p\llb q_n-1\rrb}{(q-1)^2} t^j .\]
    Now, taking vertical cofibres in the mod $(p,\beta)$ reduction of (\ref{eqn::S1-nilp-fibre-square}) yields the following exact sequence for all $j\ge 0$
    \[0\to \pi_{-2j+1}(F^{hS^1}/(p,\beta)) \to \bigoplus_{i=1}^{p-1} q_{n+1}^i \frac{\F_p\llb q_n-1\rrb}{q-1} t^j \xrightarrow{\psi^{1+p^{n+1}}_{\circ\circ}} \bigoplus_{i=1}^{p-1} q_{n+1}^i \frac{(q-1)\F_p\llb q_n-1\rrb}{(q-1)^2} t^j \to \pi_{-2j}(F^{hS^1}/(p,\beta)) \to 0\]
    where we have used the basis decomposition \[\F_p\llb q_{n+1}-1\rrb \simeq \bigoplus_{i=0}^{p-1} q_{n+1}^i \F_p\llb q_n-1\rrb.\]
    It thus suffices to show that $\psi_{\circ\circ}^{1+p^{n+1}}$ induces an isomorphism of $\F_p$-modules \[q_{n+1}^i \frac{\F_p\llb q_n-1\rrb}{q-1} t^j \xrightarrow[\psi_{\circ\circ}^{1+p^{n+1}}]{\simeq} q_{n+1}^i \frac{(q-1)\F_p\llb q_n-1\rrb}{(q-1)^2} t^j \]
    for $1\le i\le p-1$. The preceding calculation shows that this is simply the map $q_{n+1}^it^j \cdot q_n^h \mapsto q_{n+1}^i t^j \cdot (i+ph)q_n^h$, which is clearly invertible since $i+ph$ is a unit as long as $1\le i\le p-1$.
\end{proof}
\begin{remark}\label{rmk::comparison_bw_zeyuliu_and_our_qdR_comp}
    The preceding proposition has the following concrete consequence in prismatic cohomology. 
    
    As a consequence of the results of \cref{sctn::thh_of_cyclotomic_extns} and \cref{cor::-1-twist_of_fibre_seq_computing_jpn}, we have the following explicit description of prismatic cohomology of $\Z_p[\zeta_{p^n}]$:
    \[\prism_{\Z_p[\zeta_{p^n}]}\{r\} \simeq \left[\Z_p\llb q_n-1\rrb\{r\} \xrightarrow{\frac{g-1}{q-1}} \Z_p\llb q_n-1\rrb \{r\} \right].\]
    Here, $g$ is a topological generator of $1+p^n\Z_p$. However, \cite[Theorem 1.0.2]{zeyuliu2025stacky} says that
    \[\prism_{\Z_p[\zeta_{p^n}]}\{r\} \simeq \left[\Z_p\llb q_{n-1} -1\rrb\{r\} \xrightarrow{\frac{g-1}{q-1}} \Z_p\llb q_{n-1}-1\rrb \{r\} \right].\]
    The proposition above, in the special case $M= j_0^{(-1)}$, is simply saying that the two complexes appearing on the right are quasi-isomorphic, as one expects. One can actually check that these are quasi-isomorphic by a direct computation as well (though this direct calculation will essentially boil down to the calculation carried out in the proof of \cref{lem::S1-nilp-of-tilde-jpn-square}).

    For arbitrary $M$, the same remark holds, except one replaces the Breuil--Kisin twists $\O_{\Z_p[\zeta_{p^n}]^\prism}\{r\}$ with $\pi_* \O_{X_n^\prism}\{r\}$, where $X$ is a $p$-adic formal scheme over $\Z_p[\zeta_p]$, $X_n := X \times_{\Spf \Z_p[\zeta_p]} \Spf \Z_p[\zeta_{p^n}]$, and $\pi:X_n^\Syn \to \Z_p[\zeta_{p^n}]^\Syn$ is the structure map of syntomifications. 
\end{remark}
The following lemma was used in the proof of the above proposition.
\begin{lemma}\label{lem::tate_Cp_kills_induced_S1_spectra}
    Suppose $X$ is an $S^1$-equivariant spectrum that is $S^1$-nilpotent in the sense of \cite[Definition 3.1.2]{sanath_arpon_image_of_j}, i.e. it is in the thick tensor ideal of $\Mod_{\S_p^\triv}(\Sp^{BS^1})$ generated by $\S_p[S^1]$. Then $X^{tC_p} \simeq 0$.
\end{lemma}
\begin{proof}
    \footnote{The author would like to thank Maxime Ramzi for some help in correctly formalising this proof.} By the usual thick subcategory argument, it suffices to show that $(-)^{tC_p}$ kills $X\otimes \S_p[S^1]$ where $X$ is any spectrum with $S^1$-action. Since $(-)^{tC_p}$ kills the thick tensor ideal generated by spectra with induced $C_p$-action \cite[Lemma I.3.8]{nikolaus_scholze}, it suffices to show that $\S_p[S^1]$ is in the thick subcategory generated by $\S_p[C_p]$. To see this, consider the usual $C_p$-equivariant cell structure on $S^1$ given by one $C_p$-orbit of 0-simplices and one $C_p$-orbit of 1-simplices. This expresses $S^1$ as the coequalizer in spaces with $C_p$-action of the two maps $\id:C_p\to C_p$ and $\cdot \sigma:C_p\to C_p$ (where $\sigma$ is a generator of $C_p$). Applying $\S_p[-]$ then yields the $C_p$-equivariant cofibre sequence \[\S_p[C_p] \xrightarrow{\id - \sigma} \S_p[C_p]\to \S_p[S^1]\]
    as required. 
\end{proof}

We next study the Frobenius on the cyclotomic spectrum $\ku_{p,n}^{(-1)}$. 
\begin{lemma}\label{lem::factorisation_of_frob_on_kupn}
    There is a canonical $S^1$-equivariant factorization of the cyclotomic Frobenius \[\varphi : \ku_{p,n}^{(-1)} \to \ku_{p,n}^{(-1),tC_p}\]
    through the canonical map $\ku_{p,n}^{(-1)} \to \ku_{p,n+1}^{(-1)}$.
\end{lemma}
\begin{proof}
    Unwinding definitions, the cyclotomic Frobenius of $\ku_{p,n}^{(-1)}$ is the composition \[\ku_{p,n}^{(-1)} \to \ku_{p,n}^{tC_p} \to \left(\tau_{\ge 0}\left((\ku_p^{hC_{p^n}})^{hC_p}\right)\right)^{tC_p} \simeq \left(\tau_{\ge 0}\ku_{p,n}^{hC_p}\right)^{tC_p} \to \left(\ku_{p,n}^{(-1)}\right)^{tC_p},\]
    where the second arrow is induced by the map $\ku_p^{hC_{p^n}} \to \ku_p^{hC_{p^{n+1}}}$ coming from triviality of the action on $\ku_p$. However, we have $\tau_{\ge 0}(\ku_p^{hC_{p^n}})^{hC_p} \simeq \tau_{\ge 0} \ku_p^{hC_{p^{n+1}}} \simeq \ku_{p,n+1}$. The composition of the first two arrows factors through the connective cover map $\ku_{p,n+1}^{(-1)} \to \ku_{p,n+1}^{tC_p}$, and the induced map $\ku_{p,n}^{(-1)}\to \ku_{p,n+1}^{(-1)}$ is the canonical map. The claim follows.
\end{proof}
\begin{notation}
    Let $\tilde \varphi_n : \ku_{p,n}^{(-1)}\to \ku_{p,n-1}^{(-1),tC_p}$ be the $S^1$-equivariant map induced by \cref{lem::factorisation_of_frob_on_kupn}, so that $\varphi:\ku_{p,n-1}^{(-1)} \to \ku_{p,n-1}^{(-1),tC_p}$ factors as the composition of $\ku_{p,n-1}^{(-1)} \to \ku_{p,n}^{(-1)}$ and $\tilde \varphi_n$. 
\end{notation}
\begin{remark}
    On $\pi_0$, $\tilde \varphi_n$ induces the isomorphism \[\Z_p\llb q_{n+1}-1\rrb \xrightarrow[q_{n+1}\mapsto q_n]{\simeq} \Z_p\llb q_n-1\rrb.\]
\end{remark}

Varying $n$, the cyclotomic spectra $\ku_{p,n}^{(-1)}$ are related to each other by the following.
\begin{lemma}\label{lem::labelling_the_q_n_inside_ku_pn(-1)}
    There are $\Z_p^\times$-equivariant cyclotomic $E_\infty$-ring maps $\S_p\llb p^{-n-1}\Z_p\rrb^\Tate \to \ku_{p,n}^{(-1)}$ for $n\ge 0$ inducing equivalences of cyclotomic spectra \[\ku_{p,n}^{(-1)} \simeq \ku_{p,n-1}^{(-1)} \otimes_{\S_p\llb p^{-n}\Z_p\rrb^\Tate} \S_p\llb p^{-n-1}\Z_p\rrb^\Tate. \]
\end{lemma}
\begin{proof}
    This is obvious from the equivalence $$\ku_{p,n-1}^{(-1)} \simeq \THH(\Z_p[\zeta_{p^n}] /\S_p\llb q_n-1\rrb) \simeq \THH(\Z_p[\zeta_{p^n}]) \otimes_{\THH(\S_p\llb p^{-n}\Z_p\rrb)} \S_p\llb p^{-n}\Z_p\rrb^\Tate$$
    of \cref{cor::rel_thh_of_cyc_extns}.
\end{proof}
\begin{remark}\label{rmk::where_p^-n_goes_into_kupn-1}
    On underlying spectra with $S^1$-action, the map $\S_p\llb p^{-n-1}\Z_p\rrb^\triv \to \ku_{p,n}^{(-1)}$ is adjoint to a map of ring spectra $\S_p\llb p^{-n-1}\Z_p\rrb \to \ku_{p,n}^{(-1), hS^1}$, and the data of this latter map is equivalent to the data of a map \[p^{-n-1}\Z_p \to \pi_0(\ku_{p,n}^{(-1),hS^1}) = \Z_p \llb q^{1/p^{n+1}}-1\rrb .\]
    This sends the generator $p^{-n-1}$ of $p^{-n-1}\Z_p$ to $q_{n+1}$.
\end{remark}

\begin{construction}\label{cons::projection_map_on_kupn_via_semiadd}
    We construct an $\Z_p^\times\times S^1$-equivariant $\ku_{p,n-1}^{(-1)}$-linear map \[\rho_n : \ku_{p,n}^{(-1)} \to \ku_{p,n-1}^{(-1)}\]
    as follows. Consider the following composition of $S^1$-equivariant $K(1)$-local spectra
    \[\KU_p^{hC_p} \xleftarrow[\Nm_{C_p}]{\simeq} \KU_{p,hC_p} \to \KU_p\]
    where $\KU_p$ has the trivial circle action, the first equivalence is due to $K(1)$-local ambidexterity of $BC_p$ (cf. \cite{CSY_ambidexterity_and_chromatic}), and the second map exists due to the trivial circle action on $\KU_p$. This composition is $\KU_p$-linear. Taking $\tau_{\ge 0}(-)^{hC_{p^{n-1}}}$ then yields an $S^1$-equivariant $\ku_{p,n-1}$-linear map \[\ku_{p,n} \to \ku_{p,n-1}.\]
    Applying $\tau_{\ge 0} (-)^{tC_p}$ then yields a $\ku_{p,n-1}^{(-1)}$-linear map \[\ku_{p,n}^{(-1)}\to \ku_{p,n-1}^{(-1)}.\]
\end{construction}

\begin{lemma}\label{lem::properties_of_projection_map_on_kupn}
    The map $\rho_n$ satisfies the following properties.
    \begin{enumerate}
        \item There is a canonical commuting diagram 
        \[\begin{tikzcd}
            \ku_{p,n}^{(-1)} \ar{r}{\tilde \varphi_n} \ar{d}{\rho_n} & \ku_{p,n-1}^{(-1),tC_p} \ar{d}{\rho_{n-1}^{tC_p}} \\
            \ku_{p,n-1}^{(-1)} \ar{r}{\tilde \varphi_n} & \ku_{p,n-2}^{(-1),tC_p} .
        \end{tikzcd}\]
        \item The map $\rho_n$ canonically identifies with the following composition 
        \[\ku_{p,n}^{(-1)} \simeq \ku_{p,n-1}^{(-1)} \otimes_{\S_p\llb q_n-1\rrb} \S_p\llb q_{n+1}-1\rrb \xrightarrow{\id \otimes \proj} \ku_{p,n-1}^{(-1)},\]
        where the first equivalence is \cref{lem::labelling_the_q_n_inside_ku_pn(-1)} and $\proj : \S_p\llb q_{n+1}-1\rrb \to \S_p\llb q_n-1\rrb$ is the $\S_p\llb q_n-1\rrb$-linear map given by $q_{n+1}^i \mapsto 0$ if $p\nmid i$ and $q_{n+1}^{pi} \mapsto q_n^i$. 
    \end{enumerate}
\end{lemma}
\begin{proof}
    Recall that $\tilde \varphi_n$ is obtained by applying $(-)^{(-1)}$ to $$\tau_{\ge 0}(\ku_p^{hC_{p^n}}) \simeq \tau_{\ge 0}\left(\tau_{\ge 0}(\ku_p^{hC_{p^{n-1}}})\right)^{hC_p} \xrightarrow\can \tau_{\ge 0}\left(\tau_{\ge 0}(\ku_p^{hC_{p^{n-1}}})\right)^{tC_p}$$
    followed by the connective cover map $(\ku_{p,n-1}^{(-1)})^{(-1)}\to \ku_{p,n-1}^{(-1), tC_p}$. Since $\tau_{\ge 0}\ku_p^{hC_N} \simeq \tau_{\ge 0} \KU_p^{hC_N}$ for any $N\ge 1$, statement (1) follows by noticing that the following diagram (obtained by functoriality of $\Nm_{C_p}$ and $\can:(-)^{hC_p}\to (-)^{tC_p}$) commutes:
   
\[\begin{tikzcd}
	{\tau_{\ge 0}(\KU_p^{hC_p})^{hC_{p^{n-1}}}} & {\tau_{\ge 0}\left(\tau_{\ge 0}\KU_p^{hC_{p^{n-1}}}\right)^{hC_p}} & {\tau_{\ge 0}\left(\tau_{\ge 0}\KU_p^{hC_{p^{n-1}}}\right)^{tC_p}} \\
	{\tau_{\ge 0}(\KU_{p,hC_p})^{hC_{p^{n-1}}}} & {\tau_{\ge 0}\left(\tau_{\ge 0}(\KU_{p,hC_p})^{hC_{p^{n-2}}}\right)^{hC_p}} & {\tau_{\ge 0}\left(\tau_{\ge 0}(\KU_{p,hC_p})^{hC_{p^{n-2}}}\right)^{tC_p}} \\
	{\tau_{\ge 0}\KU_p^{hC_{p^{n-1}}}} & {\tau_{\ge 0}\left(\tau_{\ge 0}\KU_p^{hC_{p^{n-2}}}\right)^{hC_p}} & {\tau_{\ge 0}\left(\tau_{\ge 0}\KU_p^{hC_{p^{n-2}}}\right)^{tC_p}}
	\arrow["\simeq"{description}, draw=none, from=1-1, to=1-2]
	\arrow["\can", from=1-2, to=1-3]
	\arrow["\simeq"{marking, allow upside down}, shift left=4, draw=none, from=2-1, to=1-1]
	\arrow["{\Nm_{C_p}^{hC_{p^{n-1}}}}"', shift left, from=2-1, to=1-1]
	\arrow["\simeq"{description}, draw=none, from=2-1, to=2-2]
	\arrow[from=2-1, to=3-1]
	\arrow["{\Nm_{C_p}^{hC_{p^{n-2}}}}"', from=2-2, to=1-2]
	\arrow["\simeq"{marking, allow upside down}, shift left=3, draw=none, from=2-2, to=1-2]
	\arrow["\can", from=2-2, to=2-3]
	\arrow[from=2-2, to=3-2]
	\arrow["{\Nm_{C_p}^{hC_{p^{n-2}}}}"', from=2-3, to=1-3]
	\arrow["\simeq"{marking, allow upside down}, shift left=3, draw=none, from=2-3, to=1-3]
	\arrow[from=2-3, to=3-3]
	\arrow["\simeq"{description}, draw=none, from=3-1, to=3-2]
	\arrow["\can", from=3-2, to=3-3]
\end{tikzcd}\]

    For the second statement, notice first that the composition $$\KU_p \to \KU_p^{hC_p} \xleftarrow[\Nm_{C_p}]{\simeq} \KU_{p,hC_p} \to \KU_p$$ is the identity map since the $K(1)$-local cardinality of $BC_p$ is 1 (see \cite[Proposition 2.2.5]{CSY_ambidexterity_and_height}). Applying $\left(\tau_{\ge 0}(-)^{hC_{p^{n-1}}}\right)^{(-1)}$ then shows that the map $\ku_{p,n-1}^{(-1)} \to \ku_{p,n}^{(-1)} \xrightarrow{\rho_n} \ku_{p,n-1}^{(-1)}$ is the identity map, where the first map is the canonical map $\ku_{p,n-1}^{(-1)}\to \ku_{p,n}^{(-1)}$. The same is of course true for $\id \otimes \proj$ too.
    By $\ku_{p,n-1}^{(-1)}$-linearity of $\rho_n$ and of $\id\otimes \proj$, it thus suffices to check that the classes $q_{n+1}^i \in \pi_0(\ku_{p,n}^{(-1),hS^1})$ are killed by $\rho_n$ for $1\le i\le p-1$. For this, it actually suffices to show that the composition $$\KU_p^{hS^1} \simeq (\KU_p^{hC_p})^{hS^1} \xleftarrow[\Nm_{C_p}]{\simeq} (\KU_{p,hC_p})^{hS^1} \to \KU_p^{hS^1}$$ kills $q^i$ for $1\le i\le p-1$. Using \cite[Example 2.16]{yuan2024sphere} (see also \cite[Remark 2.2.13]{sanath_arpon_image_of_j}) and the fact that $q^i\in \pi_0(\KU_p^{hS^1})$ represents the $i$'th power of the tautological line bundle on $BS^1\simeq \C\P^\infty$, the claim that $q^i$ dies under the above composition follows from the simple observation that the $i$'th power of the tautological line bundle is not fixed under the $p$-power map $\C\P^\infty \to \C\P^\infty$ whenever $p\nmid i$.
\end{proof}

\subsection{The Case of a Point}\label{sctn::case_of_point}
The main result of this subsection is \cref{prop::computing_trace_maps_explicitly_non_equivariant}. The key idea behind the proof is the following lemma, which allows us to use Schlichtkrull's description of the transfer maps in $\THH$ of group algebras (given in \cref{sctn::thh_of_grp_algebras}) to then describe the transfer maps for the cyclotomic extensions of $\Z_p$. 

\begin{lemma}\label{lem::THH_transfer_bootstrap}
    Suppose $A$ is an $E_\infty$-ring spectrum, and suppose $B$ and $C$ are $E_\infty$-$A$-algebras such that $B$ is dualizable over $A$ (so that there is a $\THH(A)$-linear transfer map $\THH(B)\to \THH(A)$). Then, the transfer map $\THH(C\otimes_A B)\to \THH(C)$ associated to the dualisable $E_\infty$-$C$-algebra $C\otimes_A B$ coincides with the map \[\THH(C\otimes_A B) \simeq \THH(C)\otimes_{\THH(A)}\THH(B) \to \THH(C)\]
    induced by the transfer $\THH(B)\to \THH(A)$.
\end{lemma}
\begin{proof}
    Notice that the following diagram canonically commutes 
    \[\begin{tikzcd}
        \Perf(B) \ar{d}\ar{r} & \Perf(A)\ar{d} \\
        \Perf(C\otimes_A B) \ar{r} & \Perf(C),
    \end{tikzcd}\]
    where both horizontal maps are given by restriction of scalars, and both vertical maps are given by base-change; this is the Beck--Chevalley condition for pull-backs of perfect modules. In more elementary terms, this is simply the statement that for any $B$-module $M$, we have a canonical equivalence of $C$-modules $$M \otimes_B (B\otimes_A C) \simeq M\otimes_A C.$$ Taking $\THH$, we see that  
    \[\begin{tikzcd}
        \THH(B) \ar{d}\ar{r} & \THH(A)\ar{d} \\
        \THH(C\otimes_A B) \ar{r} & \THH(C),
    \end{tikzcd}\]
    where the horizontal maps are transfers and the vertical maps are the canonical ones. The lemma follows since $\THH$ is symmetric monoidal, and since the top horizontal map is $\THH(A)$-linear while the bottom horizontal map is $\THH(C)$-linear. 
\end{proof}

\begin{lemma}\label{lem::stable_transfer_compatible_with_adams_op_transfer}
    For $n\ge 1$, there is an equivalence of cyclotomic spectra \[ j_{n-1}^{(-1)} \otimes_{\S_p[Bp^{-n}\Z]^\triv} \S_p[Bp^{-n-1}\Z]^\triv \simeq \fib(\psi_{\circ\circ}^{g^{p^n}} : \ku_{p,n-1}^{(-1)} \to \ku_{p,n-1}^{(-1)}(1)[2]),\]
    where $g$ is a topological generator of $1+p\Z_p$. 
    
    Moreover, the $ j_{n-1}^{(-1)}$-linear map \[ j_{n-1}^{(-1)} \otimes_{\S_p[Bp^{-n}\Z]^\triv} \S_p[Bp^{-n-1}\Z]^\triv \to  j_{n-1}^{(-1)}\]
    induced by the stable transfer map $\S_p[Bp^{-n-1}\Z]\to \S_p[Bp^{-n}\Z]$ sits in the following commuting diagram of fibre sequences of cyclotomic spectra.
    \[\begin{tikzcd}
         j_{n-1}^{(-1)} \otimes_{\S_p[Bp^{-n}\Z]^\triv} \S_p[Bp^{-n-1}\Z]^\triv \ar{d} \ar{r} & \ku_{p,n-1}^{(-1)} \ar{r}{\psi^{g^{p^n}}_{\circ\circ}}\ar{d}{\sum_{i=0}^{p-1} \psi^{g^{ip^{n-1}}}} & \ku_{p,n-1}^{(-1)}(1)[2] \ar[equals]{d} \\
         j_{n-1}^{(-1)} \ar{r} & \ku_{p,n-1}^{(-1)} \ar[swap]{r}{\psi^{g^{p^{n-1}}}_{\circ\circ}}  & \ku_{p,n-1}^{(-1)}(1)[2]
    \end{tikzcd}\]
\end{lemma}
\begin{proof}
    We need to show that \[ j_{n-1}^{(-1)} \otimes_{\S_p[Bp^{-n}\Z]^\triv} \S_p[Bp^{-n-1}\Z]^\triv \simeq \tau_{\ge 0}\left(\ku_{p,n-1}^{(-1)}\right)^{h(1+p^{n+1}\Z_p)}.\]
    The inclusion $C_{p^{n+1}}\inj S^1$ induces a map \[ j_{n-1}^{(-1)} \otimes_{\S_p[BC_{p^n}]^\triv} \S_p[BC_{p^{n+1}}]^\triv \to  j_{n-1}^{(-1)} \otimes_{\S_p[BC_{p^n}]^\triv} \S_p[BS^1]^\triv \simeq \ku_{p,n-1}^{(-1)}\]
    where the equivalence is \cref{lem::some_pushouts_of_cyc_spectra}. Since $(1+p^{n+1}\Z_p)$ acts trivially on the left hand side, and since the left hand side is connective, the above composition factors as a map \[  j_{n-1}^{(-1)} \otimes_{\S_p[Bp^{-n}\Z]^\triv} \S_p[Bp^{-n-1}\Z]^\triv \simeq  j_{n-1}^{(-1)} \otimes_{\S_p[BC_{p^n}]^\triv} \S_p[BC_{p^{n+1}}]^\triv \to \tau_{\ge 0}\left(\ku_{p,n-1}^{(-1)}\right)^{h(1+p^{n+1}\Z_p)}\]
    of cyclotomic rings. It thus suffices to check that this map is an equivalence on underlying $p$-complete spectra. By $p$-completeness, it in turn suffices to check that the map induces an equivalence on mod $p$ homotopy groups. One checks using \cref{cor::-1-twist_of_fibre_seq_computing_jpn} that there are isomorphisms of graded rings \[\pi_* j_{n-1}^{(-1)}/p \cong \Z_p[\zeta_{p^n}]/p[u] \otimes_{\F_p}{ \bigwedge}_{\F_p}(\epsilon) \quad \text{ and } \quad \pi_*\left(\tau_{\ge 0}\left(\ku_{p,n-1}^{(-1)}\right)^{h(1+p^{n+1}\Z_p)}/p\right) \cong \Z_p[\zeta_{p^n}]/p[u] \otimes_{\F_p} {\bigwedge}_{\F_p}(\epsilon')\]
    where $u$ is a class in degree 2, $\epsilon$ is a class in degree 1 coming from the fibre sequence of \cref{cor::-1-twist_of_fibre_seq_computing_jpn}, and $\epsilon'$ is a class in degree 1 similarly defined. One moreover checks that the map $(\S/p)[Bp^{-n}\Z] \to  j_{n-1}^{(-1)}/p$ picks out the class $\epsilon$ while $(\S/p)[Bp^{-n-1}\Z] \to \tau_{\ge 0}\left(\ku_{p,n-1}^{(-1)}\right)^{h(1+p^{n+1}\Z_p)}/p$ picks out the class $\epsilon'$. The claim follows.

    It remains to prove the existence of the commutative diagram. The right hand square clearly canonically commutes, and so we just need to show that there exists an $S^1$-equivariant commutative square     \begin{equation}\label{eqn::fibre_square_computing_trasnfer_map_on_jn^-1}
        \begin{tikzcd}
         j_{n-1}^{(-1)}\otimes_{\S_p[Bp^{-n}\Z]^\triv} \S_p[Bp^{-n-1}\Z]^\triv \ar{d} \ar{r} & \ku_{p,n-1}^{(-1)} \ar{d}{\sum_{i=0}^{p-1} \psi^{g^{ip^{n-1}}}} \\
         j_{n-1}^{(-1)} \ar{r} & \ku_{p,n-1}^{(-1)}
    \end{tikzcd}
    \end{equation}
    of cyclotomic spectra. However, notice that we have a commuting diagram of spectra \[\begin{tikzcd}
        \S_p[BC_{p^{n+1}}] \ar[swap]{d}{\text{gp transfer}} \ar{r} & \S_p[B^2\Z_p] \ar{d}{\sum_{i=0}^{p-1} {g^{ip^{n-1}}}} \\
        \S_p[BC_{p^n}] \ar{r} & \S_p[B^2\Z_p]
    \end{tikzcd}\]
    where $\Z_p^\times$ acts on $B^2\Z_p$ via its multiplicative action on $\Z_p$. Applying $-\otimes_{\S_p[BC_{p^n}]}  j_{n-1}^{(-1)}$ and using \cref{lem::some_pushouts_of_cyc_spectra} yields the result; here, we also need to use that the map $\S_p[B^2\Z] \to \ku_p$ classifying the Bott element is $\Z_p^\times$-equivariant. 
\end{proof}
\begin{proposition}\label{prop::computing_trace_maps_explicitly_non_equivariant}
    Let $g$ be a topological generator of $(1+p\Z_p)$. The transfer map $$\tr: \THH(\Z_p[\zeta_{p^{n+1}}])\to \THH(\Z_p[\zeta_{p^n}])$$ sits in the following canonical $S^1$-equivariant commutative diagram of $\THH(\Z_p[\zeta_{p^n}]) \simeq  j_{n-1}^{(-1)}$-module spectra
\[\begin{tikzcd}
	{ j_n^{(-1)}} && {\ku_{p,n}^{(-1)}} && {\ku_{p,n}^{(-1)}(1)[2]} \\
	{\THH(\Z_p[\zeta_{p^{n+1}}])} \\
	&& {\ku_{p,n-1}^{(-1)}} && {\ku_{p,n-1}^{(-1)}(1)[2]} \\
	{\THH(\Z_p[\zeta_{p^n}])} \\
	{ j_{n-1}^{(-1)}} && {\ku_{p,n-1}^{(-1)}} && {\ku_{p,n-1}^{(-1)}(1)[2]}
	\arrow[from=1-1, to=1-3]
	\arrow["{\psi_{\circ\circ}^{g^{p^n}}}", from=1-3, to=1-5]
	\arrow["{\rho_n}", from=1-3, to=3-3]
	\arrow["{\rho_{n}(1)[2]}", from=1-5, to=3-5]
	\arrow["\simeq"{marking, allow upside down}, draw=none, from=2-1, to=1-1]
	\arrow["\tr"', from=2-1, to=4-1]
	\arrow["{\psi_{\circ\circ}^{g^{p^n}}}", from=3-3, to=3-5]
	\arrow["{\sum_{i=0}^{p-1}\psi^{g^{ip^{n-1}}}}", from=3-3, to=5-3]
	\arrow["\id", equals, from=3-5, to=5-5]
	\arrow["\simeq"{marking, allow upside down}, draw=none, from=4-1, to=5-1]
	\arrow[from=5-1, to=5-3]
	\arrow["{\psi_{\circ\circ}^{g^{p^{n-1}}}}", from=5-3, to=5-5]
\end{tikzcd}\]
    where the top and bottom rows are fibre sequences and the maps $\rho_n$ are from Construction \ref{cons::projection_map_on_kupn_via_semiadd}. 
\end{proposition}
\begin{proof}
    By \cref{lem::THH_transfer_bootstrap}, the transfer map $\tr :  j_n^{(-1)} \to  j_{n-1}^{(-1)}$ of cyclotomic spectra is the base-change along $\THH(\S_p[p^{-n}\Z]) \to  j_{n-1}^{(-1)}$ of the free loop transfer map \[\S_p[ p^{-n-1}\Z \times Bp^{-n-1}\Z] \simeq  \THH(\S_p[p^{-n-1}\Z]) \to \THH(\S_p[p^{-n}\Z])\simeq \S_p[ p^{-n}\Z \times Bp^{-n}\Z].\]
    where the $S^1$-action on the suspension spectra in the source and target is the twisted one. 
    By \cref{prop::explicit_thh_transfer_for_cyclic_groups}, the free loop transfer is the $S^1$-equivariant composition
    \begin{multline*}
        \THH(\S_p[p^{-n-1}\Z]) \simeq \bigoplus_{\gamma\in p^{-n-1}\Z/p^{-n}\Z} \THH(\S_p[p^{-n}\Z]) \otimes_{\S_p[ Bp^{-n}\Z]^\triv} \S_p[ Bp^{-n-1}\Z]_{(\gamma)} \\
        \xrightarrow{\proj_{\gamma=0}} \THH(\S_p[p^{-n}\Z]) \otimes_{\S_p[ Bp^{-n}\Z]^\triv} \S_p[ Bp^{-n-1}\Z]^\triv \xrightarrow{\id \otimes \tr_{st}} \THH(\S_p[p^{-n}\Z])
    \end{multline*}
    where $\tr_{st}$ is the stable transfer. Tensoring up yields the composition 
    \begin{multline*}
          j_n^{(-1)} \simeq \bigoplus_{\gamma\in p^{-n-1}\Z/p^{-n}\Z}  j_{n-1}^{(-1)} \otimes_{\S_p[ Bp^{-n}\Z]^\triv} \S_p[Bp^{-n-1}\Z]_{(\gamma)} \\
        \xrightarrow{\proj_{\gamma=0}}   j_{n-1}^{(-1)} \otimes_{\S_p[ Bp^{-n}\Z]^\triv} \S_p[ Bp^{-n-1}\Z]^\triv \xrightarrow{\id \otimes \tr_{st}}   j_{n-1}^{(-1)} .
    \end{multline*}
    By \cref{lem::stable_transfer_compatible_with_adams_op_transfer}, it remains to check the existence of a 2-simplex witnessing commutativity of the following diagram 
    \[\begin{tikzcd}
         j_n^{(-1)} \ar{rr} && \ku_{p,n}^{(-1)} \\
          j_{n-1}^{(-1)} \otimes_{\S_p[ Bp^{-n}\Z]^\triv} \bigoplus_{\gamma} \S_p[Bp^{-n-1}\Z]_{(\gamma)} \ar{d}{\proj_{\gamma=0}} && \ku_{p,n-1}^{(-1)} \otimes_{\S_p[p^{-n}\Z]} \S_p[p^{-n-1}\Z]  \ar{d}{\mr{proj}_{p^{-n-1}\Z/p^{-n}\Z}} \\
         j_{n-1}^{(-1)} \otimes_{\S_p[Bp^{-n}\Z]^\triv} \S_p[Bp^{-n-1}\Z]^\triv \ar{rr} && \ku_{p,n-1}^{(-1)} 
        \arrow["\simeq"{marking, allow upside down}, draw=none, from=1-1, to=2-1]
        \arrow["\simeq"{marking, allow upside down}, draw=none, from=1-3, to=2-3]
    \end{tikzcd}\]
    of $S^1$-equivariant $ j_{n-1}^{(-1)}$-module spectra compatible in $n$, where the horizontal maps are the usual ring maps. Here, we have used \cref{lem::properties_of_projection_map_on_kupn} (and the fact that $-\otimes_{\S_p\llb p^{-n}\Z_p\rrb} \S_p\llb p^{-n-1}\Z_p\rrb \simeq -\otimes_{\S_p[p^{-n}\Z]} \S_p[p^{-n-1}\Z]$) to identify $\rho_n$ with $\id \otimes \proj$. That this commutative diagram exists is in fact clear by Remark \ref{rmk::where_p^-n_goes_into_kupn-1} and by the fact that the map \[\bigoplus_{\gamma \in p^{-n}\Z} \S_p[ Bp^{-n}\Z]_{(\gamma)} \simeq \S_p[p^{-n}\Z \times Bp^{-n}\Z]\simeq \THH(\S_p[p^{-n}\Z]) \to \S_p[p^{-n}\Z]\]
    sends the $\gamma$-summand to $\S_p \cdot q^\gamma$. 
\end{proof}

\begin{remark}
    Compare this theorem with the following description of the usual ring map $\THH(\Z_p[\zeta_{p^n}])\to \THH(\Z_p[\zeta_{p^{n+1}}])$:
\[\begin{tikzcd}
	{ j_{n-1}^{(-1)}} && {\ku_{p,n-1}^{(-1)}} && {\ku_{p,n-1}^{(-1)}(1)[2]} \\
	{\THH(\Z_p[\zeta_{p^{n}}])} && {\ku_{p,n}^{(-1)}} && {\ku_{p,n}^{(-1)}(1)[2]} \\
	{\THH(\Z_p[\zeta_{p^{n+1}}])} \\
	{ j_{n}^{(-1)}} && {\ku_{p,n}^{(-1)}} && {\ku_{p,n}^{(-1)}(1)[2]}
	\arrow[from=1-1, to=1-3]
	\arrow["{\psi_{\circ\circ}^{g^{p^{n-1}}}}", from=1-3, to=1-5]
	\arrow["\can", from=1-3, to=2-3]
	\arrow["\can", from=1-5, to=2-5]
	\arrow["\simeq"{marking, allow upside down}, draw=none, from=2-1, to=1-1]
	\arrow["\tr"', from=2-1, to=3-1]
	\arrow["{\psi_{\circ\circ}^{g^{p^{n-1}}}}", from=2-3, to=2-5]
	\arrow[equals, from=2-3, to=4-3]
	\arrow["{\sum_{i=0}^{p-1}\psi^{g^{ip^{n-1}}}}", from=2-5, to=4-5]
	\arrow["\simeq"{marking, allow upside down}, draw=none, from=3-1, to=4-1]
	\arrow[from=4-1, to=4-3]
	\arrow["{\psi_{\circ\circ}^{g^{p^{n}}}}", from=4-3, to=4-5]
\end{tikzcd}\]
    where $\can : \ku_{p,n-1}^{(-1)}\to \ku_{p,n}^{(-1)}$ is the usual ring map. 
\end{remark}
\begin{remark}
    At the level of homotopy groups, one computes rather easily that the transfer map on $\pi_0$ is given by the usual trace map \[\Tr_{\Q_p(\zeta_{p^{n+1}})/\Q_p(\zeta_{p^n})} := \sum_{g\in \Gal(\Q_p(\zeta_{p^{n+1}})/\Q_p(\zeta_{p^n}))} g : \Z_p[\zeta_{p^{n+1}}]\to \Z_p[\zeta_{p^n}] ,\]
    while the transfer map on $\pi_{2i-1}$ ($i\ge 1$) (which are the only other non-zero homotopy groups of $\THH(\Z_p[\zeta_{p^n}])$) is given by \[\frac{\Z_p[\zeta_{p^{n+1}}]}{ip^n \frac{p}{\zeta_p-1}} \xrightarrow{\Tr_n^\Tate} \frac{\Z_p[\zeta_{p^{n}}]}{ip^{n-1} \frac{p}{\zeta_p-1}},\]
    where $$\Tr_n^\Tate := \frac{1}{p} \Tr_{\Q_p(\zeta_{p^{n+1}})/\Q_p(\zeta_{p^n})} : \Z_p[\zeta_{p^{n+1}}]\to \Z_p[\zeta_{p^n}]$$ is Tate's normalized trace for the perfectoid tower $(\Q_p(\zeta_{p^n}))_n$.
\end{remark}

\subsection{General Case}\label{sctn::computing_transfers_general_case}
We now upgrade \cref{prop::computing_trace_maps_explicitly_non_equivariant} to arbitrary bounded quasi-syntomic $p$-adic formal schemes $X$ over $\Z_p[\zeta_p]$, while also keeping track of motivic filtrations. In the following, we endow $F_\ev^{\ge *} \ku_{p,n}^{(-1)}$ with the cyclotomic synthetic spectrum structure coming from Example \ref{ex::even_cyc_sp_becomes_synth}. Recall the following definition from \cite{HahnRaksitWilson_motivic_fil_on_TC}.

\begin{definition}
    A ring map of $E_\infty$-rings $A\to B$ is said to be $p$-completely evenly faithfully flat (abbreviated eff) if for every $p$-complete $E_\infty$-$A$-algebra $E$ that is even and with bounded $p$-power-torsion, the $p$-completed tensor product $E\otimes_A B$ is also even, and the map \[\pi_{2*}E\to \pi_{2*}(E\otimes_A B)\]
    is $p$-completely faithfully flat. 
\end{definition}

\begin{lemma}\label{lem::shj_n_to_shku_n_is_eff}
    For all $m\ge n\ge 0$, the maps $ j_{n}^{(-1)} \to \ku_{p,m}^{(-1)}$ and $ j_n^{(-1),tC_p} \to \ku_{p,m}^{(-1),tC_p}$ are $p$-completely eff.
\end{lemma}
\begin{proof}
    Notice that $\ku_{p,m}^{(-1)}$ is a finite free $\ku_{p,n}^{(-1)}$-algebra (for $m>n$) by \cref{lem::labelling_the_q_n_inside_ku_pn(-1)}. In particular, we reduce to $m=n$. Since $$ j_n^{(-1)} \otimes_{\S_p[Bp^{-n-1}\Z]} \S_p \simeq \ku_{p,n}^{(-1)} \quad \text{ and } \quad  j_n^{(-1),tC_p} \otimes_{\S_p[Bp^{-n-1}\Z]} \S_p \simeq \left( j_n^{(-1)} \otimes_{\S_p[Bp^{-n-1}\Z]^\triv} \S_p^\triv\right)^{tC_p} \simeq \ku_{p,n}^{(-1), tC_p}$$ by \cref{lem::some_pushouts_of_cyc_spectra} (we also use the projection formula for $(-)^{tC_p}$ to pull $-\otimes_{\S_p[Bp^{-n-1}\Z]} \S_p$ out), it suffices to check whether $\S_p[Bp^{-n-1}\Z] \simeq \S_p[S^1] \to \S_p$ is $p$-completely eff. This is proved in the same way as \cite[Corollary 2.36]{antieau2024cyclotomic}. 
\end{proof}
\begin{corollary}\label{cor::rel_thh_always_even_for_qrsps}
    Suppose $m\ge n\ge 1$. If $R\in \QSyn_{\Z_p[\zeta_{p^n}]}$ is a QRSP, then $\THH(R \otimes_{\Z_p[\zeta_{p^n}]} \Z_p[\zeta_{p^m}]/\S_p\llb q_m-1\rrb)$ is even and the motivic filtration is the double-speed Postnikov filtration.
\end{corollary}
\begin{proof}
    This follows from the lemma, the fact that \[\THH(R \otimes_{\Z_p[\zeta_{p^n}]} \Z_p[\zeta_{p^m}]/\S_p\llb q_m-1\rrb) \simeq \THH(R) \otimes_{ j_{n-1}^{(-1)}} \ku_{p,m-1}^{(-1)},\]
    and since $\THH(R)$ is even by \cite[Theorem 7.1]{bms2}.
    The assertion about the motivic filtration follows from its identification with the even filtration.
\end{proof}

The following is a generalisation of \cref{cor::-1-twist_of_fibre_seq_computing_jpn}. 
\begin{proposition}\label{prop::prismatic_coh_from_q_de_rham_coh}
    For any $n\ge 1$ and any choice of topological generator $\gamma_n\in 1+p^n\Z_p$, there is a fibre sequence of sheaves (in fact, a fibre sequence at the level of pre-sheaves) of cyclotomic synthetic spectra on $\QSyn_{\Z_p[\zeta_{p^n}]}$ \[\Fil_\mot^{\ge *} \THH(-) \to \Fil_\mot^{\ge *} \THH(-/\S_p\llb q_n-1\rrb) \xrightarrow{\psi_{\circ\circ}^{\gamma_n}} \Fil_\mot^{\ge *-1}\THH(-/\S_p\llb q_n-1\rrb)(1)[2].\] 
\end{proposition}
\begin{remark}
    Upon taking $(-)^{t\T_\ev}$ and then taking associated graded, this gives a computation of absolute (Nygaard completed) prismatic cohomology in terms of (Nygaard completed) prismatic cohomology relative to $\Z_p\llb q_n-1\rrb$. A Nygaard decompleted generalisation of this latter result to coefficients is \cite[Theorem 1.0.2]{zeyuliu2025stacky}. On the other hand, \textit{loc. cit.} does not consider the Nygaard filtration. 
\end{remark}
\begin{proof}
    Note that both $\THH(-)$ and $\THH(-/\S_p\llb q_n-1\rrb)$ are sheaves of cyclotomic spectra on $\QSyn_{\Z_p[\zeta_{p^n}]}$, and that \[\THH(-/\S_p\llb q_n-1\rrb) \simeq \THH(-) \otimes_{ j_{n-1}^{(-1)}} \ku_{p,n-1}^{(-1)}\]
    as pre-sheaves, and thus as sheaves where the right hand side denotes the point-wise tensor product. Since $\ku_{p,n-1}^{(-1)}(1)[2]$ is a perfect $\ku_{p,n-1}^{(-1)}$-module, it also follows that \[\THH(-/\S_p\llb q_n-1\rrb)(1)[2] \simeq \THH(-/\S_p\llb q_n-1\rrb) \otimes_{\ku_{p,n-1}^{(-1)}} \ku_{p,n-1}^{(-1)}(1)[2]\]
    is a sheaf of cyclotomic spectra on $\QSyn_{\Z_p[\zeta_{p^n}]}$. Applying $\THH(-)\otimes_{ j_{n-1}^{(-1)}} -$ to the fibre sequence in \cref{cor::-1-twist_of_fibre_seq_computing_jpn} thus yields the fibre sequence of sheaves of cyclotomic spectra 
    \[\THH(-) \to \THH(-/\S_p\llb q_n-1\rrb) \xrightarrow{\psi_{\circ\circ}^{\gamma_n}} \THH(-/\S_p\llb q_n-1\rrb)(1)[2].\]
    We need to upgrade this to motivic filtrations. The motivic filtrations are defined via quasi-syntomic descent from QRSPs, and so it suffices to consider a QRSP $R$. In this case, $\THH(R)$ and $\THH(R/\S_p\llb q_n-1\rrb)$ are even. We thus get a fibre sequence \[\tau_{\ge 2*}\THH(R) \to \tau_{\ge 2*}\THH(R/\S_p\llb q_n-1\rrb) \xrightarrow{\psi_{\circ\circ}^{\gamma_n}} \tau_{\ge 2*}\left(\THH(R/\S_p\llb q_n-1\rrb)(1)[2]\right)\]
    of filtered (in fact synthetic) cyclotomic spectra. Since \[\tau_{\ge 2*}\left(\THH(R/\S_p\llb q_n-1\rrb)(1)[2]\right) \simeq \left(\tau_{\ge 2*-2}\THH(R/\S_p\llb q_n-1\rrb)\right)(1)[2],\]
    unfolding the preceding fibre sequence yields the result. 
\end{proof}
\begin{corollary}\label{cor::identifying_mot_filt_on_thh(R_n)_for_R_qrsp}
    For $R$ a QRSP $\Z_p[\zeta_{p}]$-algebra, one has an equivalence \[\Fil_\mot^{\ge *} \THH(R\otimes_{\Z_p[\zeta_p]}\Z_p[\zeta_{p^n}]) \simeq \tau_{\ge 2*-1} \THH(R\otimes_{\Z_p[\zeta_p]}\Z_p[\zeta_{p^n}]) .\]
\end{corollary}
\begin{proof}
    Set $R_n := R\otimes_{\Z_p[\zeta_p]}\Z_p[\zeta_{p^n}]$. By \cref{cor::rel_thh_always_even_for_qrsps}, $\THH(R_n /\S_p\llb q_n-1\rrb)$ is even and its motivic filtration is the double-speed Postnikov filtration. The corollary follows from taking connective covers in the fibre sequence \[\THH(R_n) \to \THH(R_n/\S_p\llb q_n-1\rrb) \to \THH(R_n/\S_p\llb q_n-1\rrb)(1)[2] \]
    from the proposition. 
\end{proof}

The following generalises \cref{lem::labelling_the_q_n_inside_ku_pn(-1)}.
\begin{lemma}\label{lem::base_change_of_prisms_in_rel_thh}
    For each $r\in \Z$, there is a natural $\Z_p^\times$-equivariant $\S_p\llb q_n-1\rrb^\triv$-module structure on the sheaf $\Fil^{\ge r}_\mot\THH(-/\S_p\llb q_n-1\rrb)$.
    
    Moreover, there is a $\Z_p^\times$-equivariant equivalence of sheaves of synthetic spectra with synthetic circle action \[\Fil_\mot^{\ge *} \THH(-\otimes_{\Z_p[\zeta_p]} \Z_p[\zeta_{p^n}] /\S_p\llb q_n-1\rrb) \simeq \left(\Fil_\mot^{\ge *} \THH(-/\S_p\llb q_1-1\rrb)\right) \otimes_{\S_p\llb q_1-1\rrb^\triv} \S_p\llb q_n-1\rrb^\triv.\]
\end{lemma}
\begin{proof}
    The first follows from the fact that $$\tau_{\ge 2r} \ku_{p,n-1}^{(-1)} \simeq \Fil_\mot^{\ge r} \THH(\Z_p[\zeta_{p^n}]/\S_p\llb q_n-1\rrb)$$
    has an obvious $\S_p\llb q_n-1\rrb^\triv$ module structure. For the second statement, there is an obvious map from the right to the left, which is an equivalence for QRSPs as a consequence of \cref{lem::labelling_the_q_n_inside_ku_pn(-1)} (in which case the motivic filtration is the double-speed Postnikov filtration) and thus for everything by quasi-syntomic descent. 
\end{proof}
\begin{lemma}\label{lem::rel_tensor_prod_of_cyc_spectra}
    Suppose $R\in \QSyn_{\Z_p[\zeta_p]}$. Then, \[\Fil^{\ge *}_\mot \THH(R \otimes_{\Z_p[\zeta_p]} \Z_p[\zeta_{p^n}] /\S_p\llb q_n-1\rrb) \simeq \Fil^{\ge *}_\mot \THH(R/\S_p\llb q_1-1\rrb) \otimes_{F_\ev^{\ge *}\ku_p^{(-1)}}F_\ev^{\ge *}\ku_{p,n-1}^{(-1)}\]
    as cyclotomic synthetic spectra, functorially in $R$. 
\end{lemma}
\begin{proof}
    By symmetric monoidality of $\THH$ and \cref{cor::rel_thh_of_cyc_extns}, this equivalence holds on underlying objects. Since the motivic filtration is the even filtration, and the even filtration is lax-symmetric monoidal, we get a canonical map from the right to the left. By quasi-syntomic descent, we may assume $R$ is QRSP, in which case $\THH(R/\S_p\llb q_n-1\rrb)$ is even and all filtrations at play are the double-speed Postnikov filtration, by \cref{cor::rel_thh_always_even_for_qrsps}. To prove the filtered equivalence, it suffices to consider this equivalence on associated graded, i.e. we need to check that the canonical map \[\pi_{2*}\ku_{p,n-1}^{(-1)} \otimes_{\pi_{2*} \ku_p^{(-1)}} \pi_{2*} \THH(R/\S_p\llb q_1-1\rrb) \to \pi_{2*}\THH(R \otimes_{\Z_p[\zeta_p]} \Z_p[\zeta_{p^n}] /\S_p\llb q_n-1\rrb)\]
    is an isomorphism. 
    Since $\ku_p^{(-1)} \to \ku_{p,n-1}^{(-1)}$ is $p$-completely eff (in fact $\ku_{p,n-1}^{(-1)}$ is a finite free even $\ku_p^{(-1)}$-module), and since $\THH(R/\S_p\llb q_1-1\rrb)$ is even, this isomorphism is \cite[Proposition 2.2.6]{HahnRaksitWilson_motivic_fil_on_TC}.
\end{proof}
\begin{construction}\label{cons::construction_of_rho_n}
    Let $\rho_n$ denote the map of sheaves of $\Z_p^\times$-equivariant synthetic spectra with synthetic circle action obtained by the composition
    \begin{align*}
        \Fil^{\ge *}_\mot \THH(- \otimes_{\Z_p[\zeta_p]} \Z_p[\zeta_{p^{n+1}}] /\S_p\llb q_{n+1}-1\rrb &) \simeq \Fil^{\ge *}_\mot \THH(-/\S_p\llb q_1-1\rrb) \otimes_{F_\ev^{\ge *}\ku_p^{(-1)}}F_\ev^{\ge *}\ku_{p,n}^{(-1)} \\
        &\xrightarrow{\id \otimes F_\ev^{\ge *}\rho_{n}} \Fil^{\ge *}_\mot \THH(-/\S_p\llb q_1-1\rrb) \otimes_{F_\ev^{\ge *}\ku_p^{(-1)}}F_\ev^{\ge *}\ku_{p,n-1}^{(-1)} \\
        &\simeq \Fil^{\ge *}_\mot \THH(- \otimes_{\Z_p[\zeta_p]} \Z_p[\zeta_{p^{n}}] /\S_p\llb q_{n}-1\rrb)
    \end{align*}
    where the equivalences are from \cref{lem::rel_tensor_prod_of_cyc_spectra} and the map $\rho_n$ in the middle is from Construction \ref{cons::projection_map_on_kupn_via_semiadd}. 
\end{construction}
The following is immediate from \cref{lem::properties_of_projection_map_on_kupn}.
\begin{lemma}\label{lem::projection_for_arbitrary_X}
    The map of sheaves $\rho_n$ defined above identifies with the composition
    \begin{align*}
        \Fil^{\ge *}_\mot \THH(- \otimes_{\Z_p[\zeta_p]} \Z_p[\zeta_{p^{n+1}}] /\S_p\llb q_{n+1}-1\rrb) &\simeq \Fil^{\ge *}_\mot \THH(-/\S_p\llb q_1-1\rrb) \otimes_{\S_p\llb q_1-1\rrb^\triv} \S_p\llb q_{n+1}-1\rrb^\triv \\
        &\xrightarrow{\id \otimes \proj_n} \Fil^{\ge *}_\mot \THH(-/\S_p\llb q_1-1\rrb) \otimes_{\S_p\llb q_1-1\rrb^\triv} \S_p\llb q_{n}-1\rrb^\triv \\
        &\simeq \Fil^{\ge *}_\mot \THH(- \otimes_{\Z_p[\zeta_p]} \Z_p[\zeta_{p^{n}}] /\S_p\llb q_{n}-1\rrb)
    \end{align*}
    where the equivalences are from \cref{lem::base_change_of_prisms_in_rel_thh} and $\proj_n$ is the map $\S_p\llb q_{n+1}-1\rrb \to \S_p\llb q_n-1\rrb$ given by \[q_{n+1}^i \mapsto \begin{cases}
        0 & p\nmid i,\\
        q_n^{i/p} & p|i.
    \end{cases}\]
\end{lemma}

We can finally prove the main result of this subsection. 
\begin{theorem}\label{thm::computing_trace_maps_explicitly_for_any_R}
    Let $g$ be a topological generator of $(1+p\Z_p)$, and $R\in \QSyn_{\Z_p[\zeta_p]}$. Set $R_n := R\otimes_{\Z_p[\zeta_p]} \Z_p[\zeta_{p^n}]$. The transfer map $$\tr: \Fil_\mot^{\ge *} \THH(R_{n+1})\to \Fil_\mot^{\ge *} \THH(R_n)$$ sits in the following functorial-in-$R$ commutative diagram of $\Fil_\mot^{\ge *} \THH(R_n)$-modules in synthetic spectra with synthetic $S^1$-action:
    
    \begin{equation}\label{eqn::diagram_for_transfers_for_arbitrary_R}
        \begin{tikzcd}
        	{\Fil_\mot^{\ge *}\THH(R_{n+1})} & {\Fil_\mot^{\ge *}\THH(R_{n+1}/\S_p\llb q_{n+1}-1\rrb)} && {\Fil_\mot^{\ge *-1}\THH(R_{n+1}/\S_p\llb q_{n+1}-1\rrb)(1)[2]} \\
        	& {\Fil_\mot^{\ge *}\THH(R_{n}/\S_p\llb q_{n}-1\rrb)} && {\Fil_\mot^{\ge *-1}\THH(R_{n}/\S_p\llb q_{n}-1\rrb)(1)[2]} \\
        	{\Fil_\mot^{\ge *}\THH(R_n)} & {\Fil_\mot^{\ge *}\THH(R_{n}/\S_p\llb q_{n}-1\rrb)} && {\Fil_\mot^{\ge *-1}\THH(R_{n}/\S_p\llb q_{n}-1\rrb)(1)[2]}
        	\arrow[from=1-1, to=1-2]
        	\arrow["\tr"', from=1-1, to=3-1]
        	\arrow["{\psi_{\circ\circ}^{g^{p^n}}}", from=1-2, to=1-4]
        	\arrow["{\rho_n}"', from=1-2, to=2-2]
        	\arrow["{\rho_n(1)[2]}"', from=1-4, to=2-4]
        	\arrow["{\psi_{\circ\circ}^{g^{p^n}}}", from=2-2, to=2-4]
        	\arrow["{\sum_{i=0}^{p-1}\psi^{g^{ip^{n-1}}}}", from=2-2, to=3-2]
        	\arrow["\id", equals, from=2-4, to=3-4]
        	\arrow[from=3-1, to=3-2]
        	\arrow["{\psi_{\circ\circ}^{g^{p^{n-1}}}}", from=3-2, to=3-4]
        \end{tikzcd}
    \end{equation}
    where the top and bottom rows are the fibre sequences of \cref{prop::prismatic_coh_from_q_de_rham_coh}. 
\end{theorem}
\begin{proof}
    Hitting \cref{prop::computing_trace_maps_explicitly_non_equivariant} with $\THH(R)\otimes_{\THH(\Z_p[\zeta_p])} -$, and using \cref{lem::THH_transfer_bootstrap}, yields the result at the unfiltered level. For the motivic filtrations, by quasi-syntomic descent it suffices to consider $R$ a QRSP, in which case $\THH(R_n/\S_p\llb q_n-1\rrb)$ is even with the motivic filtration the double-speed Postnikov filtration, and \[\Fil_\mot^{\ge *}\THH(R_n) \simeq \tau_{\ge 2*-1} \THH(R_n)\]
    by \cref{cor::identifying_mot_filt_on_thh(R_n)_for_R_qrsp}. We can thus take connective covers in the unfiltered version of (\ref{eqn::diagram_for_transfers_for_arbitrary_R}) to obtain the filtered diagram (\ref{eqn::diagram_for_transfers_for_arbitrary_R}). 
\end{proof}

\subsection{Limits Over the Cyclotomic Tower}\label{sctn::lim_of_thh_in_cyc_tower}

We now compute $$\varprojlim_n \Fil_\mot^{\ge *}\THH(-\otimes_{\Z_p[\zeta_p]} \Z_p[\zeta_{p^n}])$$
as a sheaf of $\Z_p^\times$-equivariant cyclotomic synthetic spectra.

We thank Tomer Schlank for suggesting the following lemma. Recall that for a connective $E_\infty$-ring $R$, a bounded below $R$-module spectrum $M$ is said to be \emph{almost perfect} if $\tau_{\le n} M$ is a perfect $R$-module for all $n$. If $R$ is a classical ring, this notion of almost perfect is equivalent to the notion of being pseudo-coherent in the sense of \cite[Tag064Q]{stacks-project}. 
\begin{lemma}\label{lem::almost_perfect_commute_with_limit_for_connective_sp}
    Suppose $R$ is a connective $E_\infty$-ring, $N$ an almost perfect connective $R$-module, and $\{M_i\}_{i}$ any sequential diagram of connective $R$-modules with limit $M := \lim_{i} M_i$. Then, the natural map \[M \otimes_R N \to \lim_{i} (M_i \otimes_R N)\]
    is an equivalence.
\end{lemma}
\begin{remark}\label{rmk::criterion_for_almost_perfectness}
    If the connective $E_\infty$-ring $R$ is coherent (in the sense of \cite[Definition 7.2.4.16]{HA_lurie}), then an $R$-module $M$ is almost perfect if and only if $M$ is bounded below and $\pi_i M$ is finitely presented as a $\pi_0R$-module \cite[Proposition 7.2.4.17]{HA_lurie}. In particular, if $R$ is Noetherian (i.e. $\pi_0 R$ is Noetherian and $\pi_i R$ for $i>0$ is a finitely generated $\pi_0R$-module), then $R$ is coherent and the above criterion applies. 
\end{remark}
\begin{proof}
    It suffices to show that \[\tau_{\le n}(M \otimes_R N) \to \tau_{\le n}\lim_i (M_i \otimes_R N) \] is an equivalence for all $n\ge 0$. So fix $n\ge 0$. Since $N$ is almost perfect over $R$, there exists a perfect $R$-module $N'$ and a map $N' \to N$ that is an equivalence on $\tau_{\le n+1}$. Consider the diagram 
    \begin{equation}\label{eqn::square_for_proof_of_pseudocoh_commute_with_conn_lim}
        \begin{tikzcd}
            M\otimes_R N' \ar{r}{\simeq}\ar{d} & \lim_i (M_i \otimes_R N') \ar{d} \\
            M\otimes_R N \ar{r} & \lim_i (M_i \otimes_R N).
        \end{tikzcd}
    \end{equation}
    where the top arrow is an equivalence as $N'$ is a perfect $R$-module. Since $M_i$ is connective and $\cofib(N' \to N)\in \Sp_{>n+1}$, we have $M_i\otimes \cofib(N'\to N)\in \Sp_{>n+1}$ and so \[\tau_{\le n+1} (M_i\otimes_R N') \xrightarrow\simeq \tau_{\le n+1} (M_i\otimes_R N).\] 
    By the Milnor exact sequence, $\tau_{\le n} \lim_i (M_i \otimes_R N)$ depends only on $\tau_{\le n+1} (M_i \otimes_R N)$ and similarly for $\lim_i (M_i \otimes_R N')$. In particular, the preceding connectivity estimate implies that the right vertical arrow in (\ref{eqn::square_for_proof_of_pseudocoh_commute_with_conn_lim}) is an equivalence on $\tau_{\le n}$. Again by the Milnor sequence and the connectivity of $M_i$, we have $M\in \Sp_{\ge -1}$ so that $M\otimes \cofib(N' \to N)\in \Sp_{>n}$ and thus \[\tau_{\le n} (M\otimes_R N') \xrightarrow\simeq \tau_{\le n} (M\otimes_R N).\]
    Hence, in (\ref{eqn::square_for_proof_of_pseudocoh_commute_with_conn_lim}), both vertical arrows and the top horizontal arrow are equivalences on $\tau_{\le n}$. It then follows that the bottom horizontal arrow is also an equivalence on $\tau_{\le n}$ as required. 
\end{proof}

\begin{lemma}\label{lem::trivialising_action_mod_zeta_p-1}
    For any $R\in \QSyn_{\Z_p[\zeta_{p^n}]}$, the $1+p^n\Z_p$-action on $$\gr^{*}_\mot\Phi \THH_\prism(R/\S_p\llb q_n-1\rrb) \quad \text{ and } \quad \gr^{*}_\mot \THH(R/\S_p\llb q_n-1\rrb)$$ is canonically trivial modulo $\zeta_p-1$.
\end{lemma}
\begin{proof}
    By \cite[Theorem 1.0.6]{zeyuliu2025stacky}, evaluation at the prism $(\Z_p\llb q_{n-1}-1\rrb, [p]_q)$ induces a fully faithful functor from $\mc D_\qcoh(\Z_p[\zeta_{p^n}]^\prism)$ into the $\infty$-category of $\Z_p\llb q_{n-1}-1\rrb$-modules with semi-linear $1+p^n\Z_p$ action with a trivialisation of the action modulo $q-1$. Passing to the Hodge--Tate locus (see also the proof of \cite[Theorem 3.4.6]{zeyuliu2025stacky}), this embeds $\mc D_\qcoh(\Z_p[\zeta_{p^n}]^\HT)$ fully faithfully into the $\infty$-category of $\Z_p[\zeta_{p^n}]$-modules with linear $1+p^n\Z_p$-action equipped with a trivialisation modulo \[\frac{\Z_p\llb q_{n-1}-1\rrb}{([p]_q, q-1)} \simeq \frac{\Z_p[\zeta_{p^n}]}{\zeta_p-1}.\]
    In particular, this holds for the Breuil--Kisin twists and for $f_* \O$ where $f:R^\HT \to \Z_p[\zeta_{p^n}]^\HT$ is the canonical map of stacks (for more on the stacky approach see \cite{BL_absolute_prismatic_coh, BL_prismatization}). The evaluation of the Hodge--Tate crystal $f_* \O$ at $\Z_p\llb q_{n-1}-1\rrb$ is precisely $\bar \prism_{R/ \Z_p\llb q_{n-1}-1\rrb}$. Thus, the natural $1+p^n\Z_p$-action on $\bar \prism_{R/ \Z_p\llb q_{n-1}-1\rrb}\{*\}$ is canonically trivialised modulo $\zeta_p-1$. 

    Now suppose $R$ is a large quasi-syntomic $\Z_p[\zeta_{p^n}]$-algebra (see \cite[Definition 15.1]{BS_prisms} for the definition of large); such $R$ form a basis for the quasi-syntomic topology on $\QSyn_{\Z_p[\zeta_{p^n}]}$. As a consequence of \cite[Theorem 15.2(2)]{BS_prisms}, we see that $\gr^*_\Nyg \prism_{R/\Z_p\llb q_{n-1}-1\rrb}^{(1)}\{*\}$, being functorially identified as a subset of $\bar\prism_{R/\Z_p\llb q_{n-1}-1\rrb}\{*\}$ via Frobenius, has the trivial $1+p^n\Z_p$ action modulo $\zeta_p-1$. However,  $$\gr^*_\Nyg  \prism_{R/\Z_p\llb q_{n-1}-1\rrb}^{(1)}\{*\} \simeq \gr_\mot^*\Phi \THH_\prism (R/\S_p\llb q_n-1\rrb)[-2*].$$
    The lemma now follows by quasi-syntomic descent. 
\end{proof}

Finally, we give some lemmas that will be used to study the cyclotomic Frobenius on $\varprojlim_n \THH(-\otimes_{\Z_p[\zeta_p]}\Z_p[\zeta_{p^n}])$.
\begin{construction}
    Define the map of sheaves on $\QSyn_{\Z_p[\zeta_p]}$ valued in $\Z_p^\times$-equivariant synthetic spectra with synthetic circle action \[\tilde \varphi_n: \Fil_\mot^{\ge *} \THH(- \otimes_{\Z_p[\zeta_{p}]} \Z_p[\zeta_{p^{n+1}}]/\S_p\llb q_{n+1}-1\rrb) \to \left(\Fil_\mot^{\ge *} \THH(-\otimes_{\Z_p[\zeta_{p}]} \Z_p[\zeta_{p^{n}}]/\S_p\llb q_{n}-1\rrb)\right)^{tC_{p,\ev}} \]
    by the composition 
    \begin{align*}
        \Fil_\mot^{\ge *} \THH&(- \otimes_{\Z_p[\zeta_{p}]} \Z_p[\zeta_{p^{n+1}}]/\S_p\llb q_{n+1}-1\rrb) \simeq \Fil_\mot^{\ge *} \THH(-/\S_p\llb q_1-1\rrb) \otimes_{F_\ev^{\ge *}\ku_p^{(-1)}}F_\ev^{\ge *}\ku_{p,n}^{(-1)} \\
        &\xrightarrow{\varphi_{\THH(-/\S_p\llb q_1-1\rrb)} \otimes F_\ev^{\ge *}\tilde \varphi_n} (\Fil_\mot^{\ge *} \THH(-/\S_p\llb q_1-1\rrb) )^{tC_{p,\ev}} \otimes_{(F_\ev^{\ge *}\ku_p^{(-1)})^{tC_{p,\ev}}} (F_\ev^{\ge *}\ku_{p,n}^{(-1)})^{tC_{p,\ev}} \\
        &\longrightarrow \left(\Fil_\mot^{\ge *} \THH(-/\S_p\llb q_1-1\rrb) \otimes_{F_\ev^{\ge *}\ku_p^{(-1)}}F_\ev^{\ge *}\ku_{p,n}^{(-1)}\right)^{tC_{p,\ev}}\\
        &\simeq \left(\Fil_\mot^{\ge *} \THH(-\otimes_{\Z_p[\zeta_{p}]} \Z_p[\zeta_{p^{n}}]/\S_p\llb q_{n}-1\rrb)\right)^{tC_{p,\ev}}
    \end{align*}
    where the first and last equivalences are \cref{lem::rel_tensor_prod_of_cyc_spectra}, the second arrow is the tensor product of the cyclotomic Frobenius on $\THH(-/\S_p\llb q_1-1\rrb)$ with the evenly filtered map $\tilde \varphi_n : \ku_{p,n}^{(-1)}\to \ku_{p,n-1}^{(-1),tC_p}$, and the third arrow comes from the lax symmetric monoidality of Tate fixed points. Here, we are using that $F_\ev^{\ge *} \ku_{p,n-1}^{(-1),tC_p} \simeq \left(F_\ev^{\ge *} \ku_{p,n-1}^{(-1)}\right)^{tC_{p,\ev}}$ by evenness of $\ku_{p,n-1}$ \cite[Lemma 3.18]{antieau2024cyclotomic}.
\end{construction}

These maps $\tilde \varphi_n$ are moreover compatible for varying $n$ as in the following lemma. 

\begin{lemma}\label{lem::proj_and_refined_frob_commute}
    There is a commutative diagram of sheaves on $\QSyn_{\Z_p[\zeta_p]}$ valued in synthetic spectra with synthetic circle action 
    \[\begin{tikzcd}
        \Fil_\mot^{\ge *} \THH(- \otimes_{\Z_p[\zeta_{p}]} \Z_p[\zeta_{p^{n+1}}]/\S_p\llb q_{n+1}-1\rrb) \ar{r}{\tilde \varphi_n} \ar{d}{\rho_n} & \left(\Fil_\mot^{\ge *} \THH(-\otimes_{\Z_p[\zeta_{p}]} \Z_p[\zeta_{p^{n}}]/\S_p\llb q_{n}-1\rrb)\right)^{tC_{p,\ev}} \ar{d}{\rho_{n-1}^{tC_{p,\ev}}} \\
        \Fil_\mot^{\ge *} \THH(- \otimes_{\Z_p[\zeta_{p}]} \Z_p[\zeta_{p^{n}}]/\S_p\llb q_{n}-1\rrb) \ar{r}{\tilde \varphi_{n-1}} & \left(\Fil_\mot^{\ge *} \THH(-\otimes_{\Z_p[\zeta_{p}]} \Z_p[\zeta_{p^{n-1}}]/\S_p\llb q_{n-1}-1\rrb)\right)^{tC_{p,\ev}}
    \end{tikzcd}\]
    where the vertical maps are the projection maps. 
\end{lemma}
\begin{proof}
    By \cref{lem::rel_tensor_prod_of_cyc_spectra} and the definition of $\tilde \varphi_n$ and $\rho_n$, it suffices to show that the following diagram of $S^1$-equivariant spectra commutes 
    \[\begin{tikzcd}
        F_\ev^{\ge *}\ku_{p,n}^{(-1)} \ar{r}{\tilde \varphi_n} \ar{d}{\rho_n} & F_\ev^{\ge *}\ku_{p,n-1}^{(-1),tC_p} \ar{d}{\rho_{n-1}^{tC_p}} \\
        F_\ev^{\ge *}\ku_{p,n-1}^{(-1)} \ar{r}{\tilde \varphi_{n-1}} & F_\ev^{\ge *}\ku_{p,n-2}^{(-1),tC_p} .
    \end{tikzcd}\]
    This is \cref{lem::properties_of_projection_map_on_kupn}.
\end{proof}

\begin{lemma}\label{lem::equivalence_on_TP_with_lesser_n}
    Fix a topological generator $g$ of $1+p^{n+1}\Z_p$. Define the sheaf $F^{\ge *} \mc F$ of cyclotomic synthetic spectra on $\QSyn_{\Z_p[\zeta_{p^{n}}]}$ to be the mapping fibre of \[\Fil_\mot^{\ge *}\THH(- /\S_p\llb q_{n}-1\rrb) \xrightarrow{\psi_{\circ\circ}^{g}} \Fil_\mot^{\ge *-1} \THH(-/\S_p\llb q_{n}-1\rrb)(1)[2].\]
    Then there is the following commutative diagram of cyclotomic synthetic spectra functorial in $R\in \QSyn_{\Z_p[\zeta_{p^n}]}$ 
    \[\begin{tikzcd}
        F^{\ge *}\mc F(R) \ar{r} \ar{d}{\alpha} & \Fil_\mot^{\ge *}\THH(R/\S_p\llb q_{n}-1\rrb) \ar{d}{\can} \ar{r}{\psi_{\circ\circ}^{g}} & \Fil_\mot^{\ge *-1}\THH(R/\S_p\llb q_{n}-1\rrb)(1)[2] \ar{d}{\can}\\
        \Fil_\mot^{\ge *} \THH(R_{n+1}) \ar{r} \ar{d}{\alpha'} & \Fil_\mot^{\ge *}\THH(R_{n+1}/\S_p\llb q_{n+1}-1\rrb) \ar{d}{\rho_{n}} \ar{r}{\psi_{\circ\circ}^{g}} & \Fil_\mot^{\ge *-1}\THH(R_{n+1}/\S_p\llb q_{n+1}-1\rrb)(1)[2] \ar{d}{\rho_{n}} \\
        F^{\ge *}\mc F(R) \ar{r} & \Fil_\mot^{\ge *}\THH(R/\S_p\llb q_{n}-1\rrb) \ar{r}{\psi_{\circ\circ}^{g}} & \Fil_\mot^{\ge *-1}\THH(R/\S_p\llb q_{n}-1\rrb)(1)[2]
    \end{tikzcd}\]
    where the vertical compositions are the identity, all rows are fibre sequences, and $R_{n+1} := R\otimes_{\Z_p[\zeta_{p^n}]}\Z_p[\zeta_{p^{n+1}}]$. Moreover, $\alpha^{tC_{p,\ev}}$ induces an equivalence \[(F^{\ge *}\mc F_n(-))^{tC_{p,\ev}} \simeq \left(\Fil_\mot^{\ge *} \THH(- \otimes_{\Z_p[\zeta_{p^n}]}\Z_p[\zeta_{p^{n+1}}])\right)^{tC_{p,\ev}}\]
    of sheaves of synthetic spectra with synthetic circle action on $\QSyn_{\Z_p[\zeta_{p^n}]}$, with inverse $(\alpha')^{tC_{p,\ev}}$. 
\end{lemma}
\begin{proof}
    That the diagram exists is obvious from \cref{prop::prismatic_coh_from_q_de_rham_coh}, the fact that \[\THH(R/\S_p\llb q_n-1\rrb) \simeq \THH(R) \otimes_{ j_0^{(-1)}} \ku_{p,n-1}^{(-1)}\] (and similarly for $n$ replaced with $n+1$), and the definition of $\rho_n$. Here, the maps $\alpha$ and $\alpha'$ are defined to be the fibres of the corresponding middle and right vertical maps.
    
    Thus, the real claim to check is that $\alpha^{tC_{p,\ev}}$ is an equivalence. We first assume that $R$ is a large quasi-syntomic algebra over $\Z_p[\zeta_{p^n}] := A/I$ in the sense of \cite[Definition 15.1]{BS_prisms}, where $(A,I)$ is the oriented prism $(\Z_p\llb q_{n-1}-1\rrb, [p]_q)$. Using \cref{lem::S1-nilp-of-tilde-jpn-square}, we obtain a fibre square 
    \begin{equation}\label{eqn::fib_sq_for_equivalence_on_thh_tCp_with_lesser_n}
        \begin{tikzcd}
            \THH(R/\S_p\llb q_{n}-1\rrb)^{tC_p} \ar{d}{\can} \ar{r}{\psi_{\circ\circ}^{g}} & \THH(R/\S_p\llb q_{n}-1\rrb)^{tC_p}(1)[2] \ar{d}{\can}\\
            \THH(R_{n+1}/\S_p\llb q_{n+1}-1\rrb)^{tC_p}  \ar{r}{\psi_{\circ\circ}^{g}} & \THH(R_{n+1}/\S_p\llb q_{n+1}-1\rrb)^{tC_p}(1)[2].
        \end{tikzcd}
    \end{equation}
    
    We claim that both $\THH(R/\S_p\llb q_n-1\rrb)$ and $\THH(R/\S_p\llb q_{n}-1\rrb)^{tC_p}$ are even.
    By \cite[Theorem 15.2]{BS_prisms}, $\prism_{R/A}$ is a discrete $(p,I)$-faithfully flat $A$-algebra. In particular, both $\prism_{R/A}$ and $\hat\prism_{R/A}$ are discrete and $I$-torsion free. As Frobenius on $A$ is faithfully flat, this is also true for $\prism_{R/A}^{(1)}$ and $\hat\prism_{R/A}^{(1)}$. Moreover, by Theorem 15.2 of \textit{loc. cit.}, the Nygaard filtrations are honest filtrations. In particular, it follows that all of $$\TC^-(R/\S_p\llb q_n-1\rrb), \TP(R/\S_p\llb q_n-1\rrb), \THH(R/\S_p\llb q_n-1\rrb)$$
    are even, which dispenses with the first claim. Moreover, since $\ku_p^{(-1)}$ is complex-oriented and both $\THH(R/\S_p\llb q_n-1\rrb)$ and $\Phi\THH_\prism(R/\S_p\llb q_n-1\rrb)$ are $\ku_p^{(-1)}$-algebras, we have \[\THH(R/\S_p\llb q_n-1\rrb)^{tC_p} \simeq \TP(R/\S_p\llb q_n-1\rrb) / [p]_{q}t\]
    with $t$ a unit in $\TP(R/\S_p\llb q_n-1\rrb)$. Since $\hat\prism_{R/A}^{(1)}$ is $[p]_q$-torsion free, this implies that $\THH(R/\S_p\llb q_n-1\rrb)^{tC_p}$ is even.

    Evenness of $\THH(R_n/\S_p\llb q_n-1\rrb)$ and \cite[Lemma 3.18]{antieau2024cyclotomic} then imply that \[\left(\Fil_\mot^{\ge *} \THH(R_n/\S_p\llb q_n-1\rrb\right)^{tC_{p,\ev}} \simeq \tau_{\ge 2*}\left(\THH(R_n/\S_p\llb q_n-1\rrb)^{tC_p} \right).\]
    Similarly for $n$ replaced with $n+1$ since $\ku_{p,n-1}^{(-1)}\to \ku_{p,n}^{(-1)}$ is finite free and thus $p$-completely eff. Hence, applying $\tau_{\ge 2*}$ to (\ref{eqn::fib_sq_for_equivalence_on_thh_tCp_with_lesser_n}) yields a fibre square of filtered spectra
    \[\begin{tikzcd}
        \left(\Fil_\mot^{\ge 2*}\THH(R/\S_p\llb q_{n}-1\rrb)\right)^{tC_{p,\ev}} \ar{d}{\can} \ar{r}{\psi_{\circ\circ}^{g}} & \left(\Fil_\mot^{\ge 2*}\THH(R/\S_p\llb q_{n}-1\rrb)\right)^{tC_{p,\ev}}(1)[2] \ar{d}{\can}\\
        \left(\Fil_\mot^{\ge *}\THH(R_{n+1}/\S_p\llb q_{n+1}-1\rrb)\right)^{tC_{p,\ev}}  \ar{r}{\psi_{\circ\circ}^{g}} & \left(\Fil_\mot^{\ge *}\THH(R_{n+1}/\S_p\llb q_{n+1}-1\rrb)\right)^{tC_{p,\ev}}(1)[2].
    \end{tikzcd}\]
    Taking horizontal fibres and using \cref{prop::prismatic_coh_from_q_de_rham_coh} to identify the bottom horizontal fibre, we immediately get that \[\alpha^{tC_{p,\ev}}(R) : (F^{\ge *} \mc F_n(R))^{tC_{p,\ev}}  \simeq \left(\Fil_\mot^{\ge *}\THH(R_n)\right)^{tC_{p,\ev}}\]
    is an equivalence. Quasi-syntomic descent then gives the lemma in general. 
\end{proof}

\begin{proposition}\label{prop::identifying_frob_in_the_tower}
    Fix a choice of topological generator $g$ of $1+p\Z_p$. For all $n\ge 2$, the following is a commutative diagram of synthetic spectra with synthetic circle action functorial in $R\in \QSyn_{\Z_p[\zeta_p]}$, where $R_n := R\otimes_{\Z_p[\zeta_p]}\Z_p[\zeta_{p^n}]$. To fit the diagram on the page, we abbreviate $\S_p\llb q_n-1\rrb =: \S^\qdR_n$ and drop all references to $\Fil_\mot^{\ge *}$ (with shifts in filtration denoted by $\<-\>$ from \cref{sctn::cyc_synth_sp}). 
    \[\begin{tikzcd}[cramped]
	& {\THH(R_{n+1}/\S^{\qdR}_{n+1})(1)\<1\>[1]} & \\
	{\THH(R_n/\S^{\qdR}_{n})(1)\<1\>[1]} && {\THH(R_n/\S^{\qdR}_{n})^{tC_{p,\ev}}(1)\<1\>[1]} \\
	{\THH(R_n)} && {\THH(R_n)^{tC_{p,\ev}}} \\
	\\
	{\THH(R_{n-1})} && {\THH(R_{n-1})^{tC_{p,\ev}}} \\
	{\THH(R_{n-1}/\S^{\qdR}_{n-1})(1)\<1\>[1]} && {\THH(R_{n-1}/\S^{\qdR}_{n-1})^{tC_{p,\ev}}(1)\<1\>[1]} \\
	& {\THH(R_{n}/\S^{\qdR}_{n})(1)\<1\>[1]}
	\arrow["{\rho_n(1)\<1\>[1]}"', from=1-2, to=2-1]
	\arrow["{\tilde \varphi_n}", from=1-2, to=2-3]
	\arrow["{\rho_n(1)\<1\>[1]}", from=1-2, to=7-2]
	\arrow["{\partial_n}"', from=2-1, to=3-1]
	\arrow["{\rho_{n-1}(1)\<1\>[1]}", bend right = 75, from=2-1, to=6-1]
	\arrow["{\partial_n^{tC_{p,\ev}}}", from=2-3, to=3-3]
	\arrow["{\rho_{n-1}^{tC_{p,\ev}}(1)\<1\>[1]}"', bend left = 75, from=2-3, to=6-3]
	\arrow["{\varphi_{\THH(R_n)}}"'{pos=0.3}, from=3-1, to=3-3]
	\arrow["\tr", from=3-1, to=5-1]
	\arrow["\tr^{tC_{p,\ev}}"', from=3-3, to=5-3]
	\arrow["{\varphi_{\THH(R_{n-1})}}"'{pos=0.3}, from=5-1, to=5-3]
	\arrow["{\partial_{n-1}}", from=6-1, to=5-1]
	\arrow["{\partial_{n-1}^{tC_{p,\ev}}}"', from=6-3, to=5-3]
	\arrow["{\rho_{n-1}(1)\<1\>[1]}"', from=7-2, to=6-1]
	\arrow["{\tilde \varphi_{n-1}}", from=7-2, to=6-3]
\end{tikzcd}\]
    Here $\tr:\Fil_\mot^{\ge *}\THH(R_n)\to \Fil_\mot^{\ge *}\THH(R_{n-1})$ denotes the transfer map and $$\partial_n : \Fil_\mot^{\ge *}\THH(R_n/\S_p\llb q_n-1\rrb)(1)[1] \to \Fil_\mot^{\ge *}\THH(R_n)$$ is the connecting map in the fibre sequence of \cref{prop::prismatic_coh_from_q_de_rham_coh} (where we choose the topological generator $g^{p^{n-1}}$ for $1+p^n\Z_p$ for all $n$).
\end{proposition}
\begin{proof}
    The left-most and right-most commutative squares are obtained by rotating the diagram in \cref{thm::computing_trace_maps_explicitly_for_any_R}. The middle rectangle commutes since the transfer maps on $\THH$ are maps of cyclotomic spectra. The big trapezium (with a curved short side) on the left commutes trivially, while the big right trapezium (with a curved short side) commutes by \cref{lem::proj_and_refined_frob_commute}. Thus, the only thing left to check is that the following pentagon commutes for all $n$. 
    \begin{equation}\label{eqn::pentagon_describing_frob_in_tower}
        \begin{tikzcd}[cramped]
        	& {\Fil_\mot^{\ge r-1}\THH(R_{n+1}/\S^{\qdR}_{n+1})(1)[1]} & \\
        	{\Fil_\mot^{\ge r-1}\THH(R_n/\S^{\qdR}_{n})(1)[1]} && {\Fil_\mot^{\ge r-1}\THH(R_n/\S^{\qdR}_{n})^{tC_{p,\ev}}(1)[1]} \\
        	{\Fil_\mot^{\ge r}\THH(R_n)} && {\Fil_\mot^{\ge r}\THH(R_n)^{tC_{p,\ev}}}
        	\arrow["{\rho_n(1)[1]}"', from=1-2, to=2-1]
        	\arrow["{\tilde \varphi_n}", from=1-2, to=2-3]
        	\arrow["{\partial_n}"', from=2-1, to=3-1]
        	\arrow["{\partial_n^{tC_{p,\ev}}}", from=2-3, to=3-3]
        	\arrow["{\varphi_{\THH(R_n)}}"', from=3-1, to=3-3]
        \end{tikzcd}
    \end{equation}
    Define the sheaf \[F^{\ge *}\mc F_n := \fib\left(\Fil_\mot^{\ge *}\THH(-\otimes_{\Z_p[\zeta_{p}]} \Z_p[\zeta_{p^n}] /\S_p\llb q_n-1\rrb) \xrightarrow{\psi_{\circ\circ}^{g_n}} \Fil_\mot^{\ge *-1}\THH(-\otimes_{\Z_p[\zeta_{p}]} \Z_p[\zeta_{p^n}]/\S_p\llb q_n-1\rrb)(1)[2]\right)\] where $g_n$ is a topological generator $1+p^{n+1}\Z_p$ chosen so that $g_{n+1}=g_n^p$; there is an equivalence $$\alpha^{tC_{p,\ev}}: (F^{\ge *}\mc F_n(R))^{tC_{p,\ev}} \xrightarrow{\simeq} \left(\Fil_\mot^{\ge *} \THH(R_{n+1})\right)^{tC_{p,\ev}}$$ from \cref{lem::equivalence_on_TP_with_lesser_n}. Rotating the diagram of fibre sequences in \cref{lem::equivalence_on_TP_with_lesser_n} yields the following two diagrams 
    \begin{equation}\label{eqn::connecting_map_in_TP_and_isom_alpha}
        \begin{tikzcd}
        	{\left(\Fil_\mot^{\ge *-1}\THH(R_{n-1}/\S_p\llb q_{n-1}-1\rrb)\right)^{tC_{p,\ev}}(1)[1]} && {(F^{\ge *}\mc F_{n-1}(R))^{tC_{p,\ev}}} \\
        	{\left(\Fil_\mot^{\ge *-1}\THH(R_{n}/\S_p\llb q_{n}-1\rrb)\right)^{tC_{p,\ev}}(1)[1]} && {\left(\Fil_\mot^{\ge *}\THH(R_n)\right)^{tC_{p,\ev}}}
        	\arrow["{\tilde \partial_n^{tC_{p,\ev}}}", from=1-1, to=1-3]
        	\arrow["\iota^{tC_{p,\ev}}", from=1-1, to=2-1]
        	\arrow["{\alpha^{tC_{p,\ev}}}"', from=1-3, to=2-3]
        	\arrow["\simeq"{marking, allow upside down}, shift left=3, draw=none, from=1-3, to=2-3]
        	\arrow["{\partial_n^{tC_{p,\ev}}}"', from=2-1, to=2-3]
        \end{tikzcd}
    \end{equation}
    and 
    \begin{equation}\label{eqn::connecting_map_in_TP,_proj_and_inv_isom_alpha}
        \begin{tikzcd}
        	{\left(\Fil_\mot^{\ge *-1}\THH(R_n/\S_p\llb q_{n}-1\rrb)\right)^{tC_{p,\ev}}(1)[1]} && {\left(\Fil_\mot^{\ge *}\THH(R_n)\right)^{tC_{p,\ev}}} \\
        	{\left(\Fil_\mot^{\ge *-1}\THH(R_{n-1}/\S_p\llb q_{n-1}-1\rrb)\right)^{tC_{p,\ev}}(1)[1]} && {(F^{\ge *}\mc F_{n-1}(R))^{tC_{p,\ev}}}
        	\arrow["{\partial_n^{tC_{p,\ev}}}"', from=1-1, to=1-3]
        	\arrow["{\rho_{n-1}^{tC_{p,\ev}}(1)[1]}"', from=1-1, to=2-1]
        	\arrow["{(\alpha^{tC_{p,\ev}})^{-1}}", from=1-3, to=2-3]
        	\arrow["{\tilde \partial_n^{tC_{p,\ev}}}", from=2-1, to=2-3]
        	\arrow["\simeq"{marking, allow upside down}, shift left=3, draw=none, from=2-3, to=1-3]
        \end{tikzcd}
    \end{equation}  
    where $\iota$ is induced by the canonical map $\ku_{p,n-1}^{(-1)}\to \ku_{p,n}^{(-1)}$ and $\tilde \partial_n :\Fil_\mot^{\ge *-1}\THH(R_n/\S_p\llb q_{n}-1\rrb)(1)[1] \to F^{\ge *}\mc F_n(R)$ is the connecting map coming from the definition of $\mc F_n$. On the other hand, since the fibre sequence in \cref{prop::prismatic_coh_from_q_de_rham_coh} is of cyclotomic synthetic spectra, we have a commutative diagram 
    \begin{equation}\label{eqn::frob_commute_with_connecting_map}
        \begin{tikzcd}
            \Fil_\mot^{\ge *-1}\THH(R_n/\S_p\llb q_n-1\rrb)(1)[1] \ar{d}{\varphi_{n,\qdR}} \ar{r}{\partial_n} & \Fil_\mot^{\ge *}\THH(R_n)\ar{d}{\varphi_{n}}\\
            \left(\Fil_\mot^{\ge *-1}\THH(R_n/\S_p\llb q_n-1\rrb)\right)^{tC_{p,\ev}}(1)[1]\ar{r}{\partial_n^{tC_{p,\ev}}} & \left(\Fil_\mot^{\ge *}\THH(R_n)\right)^{tC_{p,\ev}}
        \end{tikzcd}
    \end{equation}
    where for notational simplicity we have denoted the cyclotomic Frobenius for $\THH(R_n)$ and for $\THH(R_n/\S_p\llb q_n-1\rrb)(1)[1]$ by $\varphi_n$ and $\varphi_{n,\qdR}$. However, as a consequence of the definition of $\tilde \varphi_{n-1}$ the map $\varphi_{n,\qdR}$ factors through $\tilde \varphi_{n-1}$ (see \cref{lem::factorisation_of_frob_on_kupn}), and so we can glue diagrams (\ref{eqn::connecting_map_in_TP_and_isom_alpha}), (\ref{eqn::connecting_map_in_TP,_proj_and_inv_isom_alpha}), and the one from \cref{lem::proj_and_refined_frob_commute} onto diagram (\ref{eqn::frob_commute_with_connecting_map}) to get the following large diagram. As before, in order to fit the diagram onto the page, we have abbreviated $\S_p\llb q_n-1\rrb := \S_n^\qdR$ and have dropped the notation $\Fil_\mot^{\ge *}$ (using $\<-\>$ to denote shifts in filtration instead).
    
\[\begin{tikzcd}
	{\THH(R_{n+1}/\S^{\qdR}_{n+1})(1)\<1\>[1]} & {\THH(R_n/\S^{\qdR}_{n})(1)\<1\>[1]} & \\
	&& {\THH(R_n)} \\
	& {\THH(R_n/\S^{\qdR}_{n})^{tC_{p,\ev}}(1)\<1\>[1]} \\
	\\
	{\THH(R_{n-1}/\S^{\qdR}_{n-1})^{tC_{p,\ev}} (1)\<1\>[1]} & {(F^{\ge *}\mc F_n(R))^{tC_{p,\ev}}} & {\THH(R_n)^{tC_{p,\ev}}} \\
	\\
	{\THH(R_n/\S^{\qdR}_{n})^{tC_{p,\ev}}(1)\<1\>[1]}
	\arrow["{\rho_n(1)\<1\>[1]}"', draw = none, shift right = 2, from=1-1, to=1-2]
	\arrow[from=1-1, to=1-2]
	\arrow[""{name=0, anchor=center, inner sep=0}, "{\tilde \varphi_n}", bend right = 80, from=1-1, to=7-1]
	\arrow["{\partial_n}", from=1-2, to=2-3]
	\arrow["{\varphi_{n,\qdR}}", from=1-2, to=3-2]
	\arrow[""{name=1, anchor=center, inner sep=0}, "{\tilde \varphi_{n-1}}"', from=1-2, to=5-1]
	\arrow["{(\ref{eqn::frob_commute_with_connecting_map})}"{description}, draw=none, from=2-3, to=3-2]
	\arrow["{\varphi_n}", from=2-3, to=5-3]
	\arrow["{(\ref{eqn::connecting_map_in_TP_and_isom_alpha})}"', draw=none, from=3-2, to=5-2]
	\arrow["{\partial_n^{tC_{p,\ev}}}", from=3-2, to=5-3]
	\arrow["{\iota^{tC_{p,\ev}}}"', from=5-1, to=3-2]
	\arrow["{\tilde \partial_n^{tC_{p,\ev}}}", from=5-1, to=5-2]
	\arrow["{\alpha^{tC_{p,\ev}}}"', from=5-2, to=5-3]
	\arrow["\simeq"{marking, allow upside down}, shift left=3, draw=none, from=5-2, to=5-3]
	\arrow["{\rho_{n-1}^{tC_{p,\ev}}(1)\<1\>[1]}"', from=7-1, to=5-1]
	\arrow[""{name=2, anchor=center, inner sep=0}, "{\partial_n^{tC_{p,\ev}}}"', bend right = 10, from=7-1, to=5-3]
	\arrow["{\text{\cref{lem::proj_and_refined_frob_commute}}}", shift left=3, draw=none, from=0, to=1]
	\arrow["{(\ref{eqn::connecting_map_in_TP,_proj_and_inv_isom_alpha})}"{description}, draw=none, from=5-1, to=2]
\end{tikzcd}\]
    The outermost pentagon is the one we are after. 
\end{proof}

\begin{lemma}\label{lem::synthetic_tate_commute_with_limit_over_transfers}
    For any bounded qcqs quasi-syntomic $p$-adic formal scheme $X$ over $\Z_p[\zeta_p]$, write $X_n := X\times_{\Spf\Z_p[\zeta_p]} \Spf \Z_p[\zeta_{p^n}].$ 
    There are functorial-in-$X$ equivalences of synthetic spectra with synthetic circle action
    \begin{align*}
        \left(\varprojlim_n \Fil_\mot^{\ge *}\THH(X_n)\right)^{tC_{p,\ev}} &\simeq \varprojlim_n \left(\Fil_\mot^{\ge *}\THH(X_n)\right)^{tC_{p,\ev}}, \\
        \left(\varprojlim_n \Fil_\mot^{\ge *}\THH(X_n/\S_p\llb q_n-1\rrb)\right)^{tC_{p,\ev}} &\simeq \varprojlim_n \left(\Fil_\mot^{\ge *}\THH(X_n/\S_p\llb q_n-1\rrb)\right)^{tC_{p,\ev}},
    \end{align*}
    where the transition maps in the left hand limits are the transfer maps and in the right hand limits are the maps $\rho_n$. 
\end{lemma}
\begin{proof}
    We may first restrict to $X=\Spf R$ affine, since the totalisation for the \v Cech nerve of a finite Zariski cover by affines is a finite limit, by the qcqs condition on $X$. In the affine case, the motivic filtrations in question are all the even filtrations on connective $E_\infty$-ring spectra, which by \cite[Remark 2.13]{antieau2024cyclotomic} are all connective in the Postnikov $t$-structure on synthetic spectra. The lemma follows after an application of \cite[Lemma 2.80]{antieau2024cyclotomic}.
\end{proof}

We can now prove a stronger form of \cref{mainthm::desc_of_limit_of_THH_of_cyc}, the main result of this paper. 
\begin{theorem}\label{thm::limit_THH_for_cyclotomic_tower}
    For any bounded qcqs quasi-syntomic $p$-adic formal  $X$ over $\Z_p[\zeta_p]$, write $$X_n := X\times_{\Spf\Z_p[\zeta_p]} \Spf \Z_p[\zeta_{p^n}] \quad \text{ and } \quad X_\infty:= X\times_{\Spf \Z_p[\zeta_p]} \Spf \Z_p^\cyc.$$ 
    There is a functorial-in-$X$ $\Z_p^\times$-equivariant equivalence of $\Fil_\mot^{\ge *}\THH(X_\infty)$ 
    modules in cyclotomic synthetic spectra:
    \begin{align*}
        \varprojlim_n \Fil_\mot^{\ge *}\THH(X_n) &\simeq \varprojlim_n\Fil_\mot^{\ge *-1}\THH(X_n/\S_p\llb q_n-1\rrb)(1)[1],
    \end{align*}
    where the transition maps on the right hand side are the maps $\rho_n$, and where the cyclotomic synthetic Frobenius on the right hand side is given by the following composition
    \begin{align*}
        \varprojlim_{n\ge 1}\Fil_\mot^{\ge *-1}\THH(X_n/\S_p\llb q_n-1\rrb)(1)[1] &\xrightarrow[\mr{shift}]{\simeq} \varprojlim_{n\ge 1} \Fil_\mot^{\ge *-1}\THH(X_{n+1}/\S_p\llb q_{n+1}-1\rrb)(1)[1] \\
        &\xrightarrow{(\tilde \varphi_n)} \varprojlim_{n\ge 1}\left(\Fil_\mot^{\ge *-1}\THH(X_n/\S_p\llb q_n-1\rrb)(1)[1]\right)^{tC_{p,\ev}},\\
        &\simeq \left(\varprojlim_{n\ge 1}\Fil_\mot^{\ge *-1}\THH(X_n/\S_p\llb q_n-1\rrb)(1)[1]\right)^{tC_{p,\ev}},
    \end{align*}
    where we use \cref{lem::proj_and_refined_frob_commute} to see that $\tilde \varphi_n$ commutes with the transition maps in the limit and the last equivalence is due to \cref{lem::synthetic_tate_commute_with_limit_over_transfers}.
\end{theorem}
\begin{proof}
    We will implicitly use \cref{lem::synthetic_tate_commute_with_limit_over_transfers} to commute the Tate construction past limits over transfers. 
    
    First, taking the limit in \cref{thm::computing_trace_maps_explicitly_for_any_R}, we get a fibre sequence of $S^1$-equivariant spectra \[\varprojlim_n \Fil_\mot^{\ge *}\THH(R_n) \to \varprojlim_n \Fil_\mot^{\ge *}\THH(R_n/\S_p\llb q_n-1\rrb) \to \varprojlim_n \Fil_\mot^{\ge *-1}\THH(R_n/\S_p\llb q_n-1\rrb)(1)[2],\]
    where the transition maps in the middle term are given by 
    \begin{align*}
        \Fil_\mot^{\ge *}\THH(R_{n+1}/\S_p\llb q_{n+1}-1\rrb)& \simeq \Fil_\mot^{\ge *}\THH(R_n/\S_p\llb q_n-1\rrb)  \otimes_{\S_p\llb p^{-n}\Z_p\rrb} \S_p\llb p^{-n-1}\Z_p\rrb \\
        &\xrightarrow{\id \otimes \mr{proj}} \Fil_\mot^{\ge *}\THH(R_n/\S_p\llb q_n-1\rrb) \\
        &\xrightarrow{\sum_{i=0}^{p-1} \psi^{g^{ip^{n-1}}}} \Fil_\mot^{\ge *}\THH(R_n/\S_p\llb q_n-1\rrb) ,
    \end{align*}
    and the transition maps in the third limit are the maps $\rho_n$.
    Here, we have fixed a topological generator $g$ of $1+p\Z_p$. 
    
    We first claim that the middle term in the above fibre sequence is zero. To do this, it suffices to show that the underlying spectrum is zero, so we may forget the circle action. Moreover, the motivic filtration being complete and exhaustive, it suffices to show the associated graded vanishes. Working non-$S^1$-equivariantly (or, more accurately, non-$\gr^*\T_\ev$-equivariantly), the middle term is a $\Z_p[\zeta_p]$-algebra, and since we are working in the $p$-complete category (equivalently, the $(\zeta_p-1)$-complete category) we may work mod $\zeta_p - 1$. In this case, it suffices to check that the transition maps \[\gr^*_\mot\THH(R_{n+1}/\S_p\llb q_{n+1}-1\rrb) /(\zeta_p - 1) \to \gr^*_\mot\THH(R_{n}/\S_p\llb q_{n} - 1\rrb) /(\zeta_p - 1)\]
    of sheaves are zero, as the limit over zero maps is zero. By definition, this transition map factors through the map \[\sum_{i=0}^{p-1} \psi^{g^{ip^{n-1}}} : \gr^*_\mot\THH(R_{n}/\S_p\llb q_{n} - 1\rrb) /(\zeta_p - 1) \to \gr^*_\mot\THH(R_{n}/\S_p\llb q_{n} - 1\rrb) /(\zeta_p - 1).\]
    \cref{lem::trivialising_action_mod_zeta_p-1} identifies this map with the multiplication by $p$ map, which is zero mod $\zeta_p-1$ as claimed.

    It remains to study the cyclotomic synthetic Frobenius. For this, we take limits in $n$ of the diagram (\ref{eqn::pentagon_describing_frob_in_tower}) in \cref{prop::identifying_frob_in_the_tower} to obtain the following commutative pentagon. Again, in order to fit the diagram onto the page, we have abbreviated $\S_p\llb q_n-1\rrb := \S_n^\qdR$ and have dropped the notation $\Fil_\mot^{\ge *}$ (using $\<-\>$ to denote shifts in filtration instead).
    \[\begin{tikzcd}
        & {\varprojlim_n \THH(R_{n+1}/\S_{n+1}^\qdR)(1)\<1\>[1]} & \\
        {\varprojlim_n \THH(R_n/\S_n^\qdR)(1)\<1\>[1]} && {\varprojlim_n \THH(R_n/\S_n^\qdR)^{tC_{p,\ev}}(1)[1]} \\
        {\varprojlim_n \THH(R_n)} && {\varprojlim_n \THH(R_n)^{tC_{p,\ev}}}
        \arrow["{(\rho_n(1)\<1\> [1])_n}"', from=1-2, to=2-1]
        \arrow["{(\tilde \varphi_n)_n}", from=1-2, to=2-3]
        \arrow["{(\partial_n)_n}"', from=2-1, to=3-1]
        \arrow["{(\partial_n^{tC_{p,\ev}})_n}", from=2-3, to=3-3]
        \arrow["{(\varphi_{\THH(R_n)})_n}"', from=3-1, to=3-3]
    \end{tikzcd}\]
    Now, we have established that the maps $(\partial_n)_n$ and $(\partial_n^{tC_{p,\ev}})_n$ are both equivalences. Moreover, since the limit is over $\rho_n(1)\<1\>[1]$ anyway, the map $(\rho_n(1)\<1\>[1])_n$ is the shift down map with inverse the shift up map appearing in the statement of \cref{thm::limit_THH_for_cyclotomic_tower}. The claim follows. 
\end{proof}

\begin{construction}
    We endow $$\ms M_\Iw := \varprojlim_n \S_p\llb q_n-1\rrb,$$ the limit over the $\S_p\llb q_1-1\rrb$-linear projection maps $\proj_n : \S_p\llb q_{n+1}-1\rrb \to \S_p\llb q_n-1\rrb$, with a cyclotomic structure.
    
    The circle action is the trivial one. The cyclotomic Frobenius is given by the composition \[\varprojlim_{n\ge 1} \S_p\llb q_n-1\rrb \xrightarrow{\mr{shift}} \varprojlim_{n\ge 1} \S_p\llb q_{n+1}-1\rrb \to \varprojlim_{n\ge 1} \S_p\llb q_n-1\rrb \simeq \varprojlim_{n\ge 1} \S_p\llb q_n-1\rrb^{tC_p} \simeq \left(\varprojlim_{n\ge 1} \S_p\llb q_n-1\rrb\right)^{tC_p}\] 
    where the second arrow is the ring map \[\tilde \varphi: \S_p\llb q_{n+1}-1\rrb \xrightarrow\simeq \S_p\llb q_{n}-1\rrb, \quad q_{n+1} \mapsto q_n.\]
\end{construction}

One checks that the tautological $\S_p\llb q_1-1\rrb$-module structure on $\varprojlim_n \S_p\llb q_n-1\rrb$ in $\Sp$ extends to a $\S_p\llb q_1-1\rrb^\Tate$-module structure on $\varprojlim_n \S_p\llb q_n-1\rrb$ in $\CycSp$. Moreover, this cyclotomic spectrum $\ms M_\Iw$ is the same as that in the introduction, as one can check that the following diagram canonically commutes
\[\begin{tikzcd}
    \S_p\llb q_{n+1}-1\rrb \ar{r}{\tilde \varphi^n} \ar{d}{\proj} & \S_p\llb q_1-1\rrb \ar{d}{\tau}\\
    \S_p\llb q_n-1\rrb \ar{r}{\tilde \varphi^{n-1}} & \S_p\llb q_1-1\rrb
\end{tikzcd}\]
where both horizontal maps are equivalences, and $\tau:\S_p\llb q_1-1\rrb \to \S_p\llb q_1-1\rrb$ is given by \[q_1^k \mapsto \begin{cases}
    0 & p\nmid k,\\
    q_1^{k/p} & p|k.
\end{cases}\]

We need to endow this cyclotomic spectrum with a synthetic structure. Consider the functor \[\mr{ins}^0 : \Sp_p \to \Fil\Sp_p\]
given by $F^{\ge i}\mr{ins}^0 X := X$ if $i\le 0$, and $F^{\ge i} \mr{ins}^0 X := 0$ if $i>0$. Recall that $\mr{ins}^0 \S_p$ is the unit of the Day convolution symmetric monoidal structure on $\Fil \Sp_p$. 
\begin{lemma}
    There is an equivalence of synthetic spectra \[\S_\ev \otimes_{\mr{ins}^0\S_p} \mr{ins}^0 \S_p\llb q-1\rrb \simeq \S_\ev \llb q-1\rrb \quad \text{ and } \quad \S_\ev \otimes_{\mr{ins}^0 \S_p} \mr{ins}^0 \ms M_\Iw \simeq \varprojlim_n \S_\ev \llb q_n-1\rrb\]
    where the limit on the right hand side of the right hand equivalence is over the projection maps.
\end{lemma}
\begin{proof}
    It suffices to check that $\S_\ev \otimes_{\mr{ins}^0 \S_p} \mr{ins}^0(-) $ commutes with sequential limits of connective spectra. Since $F^i(\S_\ev \otimes_{\mr{ins}^0 \S_p} \mr{ins}^0(-)) \simeq F^i\S_\ev \otimes - $, the claim follows from \cref{lem::almost_perfect_commute_with_limit_for_connective_sp} by noting that $F^i\S_\ev$ is almost perfect over $\S_p$. 
\end{proof}
In particular, this lemma implies that the equivalence in the following corollary identifies with the equivalence in the statement of \cref{mainthm::desc_of_limit_of_THH_of_cyc}. 

\begin{corollary}\label{cor::when_we_bring_qdR_out_of_limit}
    Continue with notation as in \cref{thm::limit_THH_for_cyclotomic_tower}. If $X$ is such that for each $i$, $\gr_\Nyg^i\hat \prism^{(1)}_{X/\Z_p\llb q-1\rrb}$ is almost perfect over $\Z_p\llb q_1-1\rrb$, then there is an equivalence of cyclotomic synthetic spectra 
        \[\varprojlim_n \Fil_\mot^{\ge *}\THH(X_n) \simeq \left(\Fil_\mot^{\ge *-1}\THH(X/\S_p\llb q_1-1\rrb)\right)(1)[1] \otimes_{\S_p\llb q_1-1\rrb^\Tate} \varprojlim_n \S_p\llb q_n-1\rrb .\] 
\end{corollary}
\begin{proof}
    By the theorem, and using \cref{lem::base_change_of_prisms_in_rel_thh} and \cref{lem::projection_for_arbitrary_X}, it suffices to show that the canonical map of spectra
    \begin{multline*}
        \varprojlim_n\left(\Fil_\mot^{\ge *-1}\THH(X/\S_p\llb q_1-1\rrb)(1)[1] \otimes_{\S_p\llb q_1-1\rrb}  \S_p\llb q_n-1\rrb \right) \\ \to \Fil_\mot^{\ge *-1}\THH(X/\S_p\llb q_1-1\rrb)(1)[2] \otimes_{\S_p\llb q_1-1\rrb} \varprojlim_n \S_p\llb q_n-1\rrb 
    \end{multline*}
    is an equivalence under the extra condition on $X$. The motivic filtration being complete, we may pass to associated graded. Notice that $\S_p\llb p^{-1}\Z_p\rrb$ is a Noetherian $E_\infty$-ring: $\pi_0\S_p\llb p^{-1}\Z_p\rrb \cong \Z_p\llb T\rrb$ is Noetherian and for $i>0$, \[\pi_i \S_p\llb p^{-1}\Z_p\rrb \simeq (\pi_i \S_p) \otimes_{\Z_p} \Z_p\llb p^{-1}\Z_p\rrb\]      
    is finitely generated over $\Z_p\llb p^{-1}\Z_p\rrb$ since $\pi_i\S_p$ is finite. Moreover, the condition on $X$ combined with Remark \ref{rmk::criterion_for_almost_perfectness} guarantees that $\gr^*_\mot \THH(X/\S_p\llb q_1-1\rrb)$ is almost perfect over $\S_p\llb q_1-1\rrb$ for each $*$. We may then conclude by \cref{lem::almost_perfect_commute_with_limit_for_connective_sp}.
\end{proof}

The following is just the associated graded version of \cref{cor::when_we_bring_qdR_out_of_limit}.

\begin{corollary}\label{cor::limit_of_nyg_filt_prismatic_coh_over_cyc_tower}
    Fix $r\ge 0$. Let $X$ and $X_n$ be as in \cref{thm::limit_THH_for_cyclotomic_tower}. If $\gr^i_\Nyg \hat \prism_{X/\Z_p\llb q-1\rrb}^{(1)}$ is almost perfect (i.e. pseudo-coherent) as a $\Z_p\llb q-1\rrb$-module for all $i$, then 
    \begin{align*}
        \varprojlim_n (\Fil_\Nyg^{\ge r} \hat \prism_{X_n}) \{r\} &\simeq (\Fil_\Nyg^{\ge r-1} \hat \prism_{X_1/\Z_p\llb q-1\rrb}^{(1)})\{r-1\} \hat \otimes_{\Z_p\llb q_1-1\rrb} \varprojlim_{n\ge 1} \Z_p\llb q_n-1\rrb(1)[-1]\\
        \varprojlim_n \hat \prism_{X_n}\{r\} &\simeq \hat \prism_{X_1/\Z_p\llb q-1\rrb}^{(1)}\{r-1\} \hat \otimes_{\Z_p\llb q_1-1\rrb} \varprojlim_{n\ge 1} \Z_p\llb q_n-1\rrb(1)[-1]
    \end{align*}
    where the right hand side is $(p, q-1)$-adically completed. 
\end{corollary}

We record a condition on $X$ that guarantees $\gr^i_\Nyg\hat \prism_{X/\Z_p\llb q-1\rrb}$ is pseudo-coherent over $\Z_p\llb q-1\rrb$. 

\begin{lemma}\label{lem::sufficient_condition_for_almost_perf_gr_Nyg}
    If $X$ is a qcqs proper regular bounded $p$-adic formal scheme over $\Z_p[\zeta_p]$, then $\gr^i_\Nyg \hat \prism^{(1)}_{X_1/\Z_p\llb q-1\rrb}$ is perfect as a $\Z_p\llb q_1-1\rrb$-module.
\end{lemma}
\begin{proof}
    By \cite[Remark 5.1.2]{BL_absolute_prismatic_coh}, we have \[\gr^i_\Nyg \hat \prism^{(1)}_{X/\Z_p\llb q-1\rrb} \simeq \Fil_i^\conj \bar \prism_{X/\Z_p\llb q-1\rrb} \{i\}\]
    where the right hand side is the (increasing) conjugate filtration on relative Hodge--Tate cohomology as defined in \cite[Construction 7.6]{BS_prisms}. Now, $$\gr_j^\conj \bar \prism_{X/\Z_p\llb q-1\rrb} \simeq (\bigwedge^j_{X} \L_{X/\Z_p[\zeta_p]})^\wedge_p[-j]\{j\}.$$ However, as $X$ is regular over $\Z_p[\zeta_p]$, it follows that $(\L_{-/\Z_p[\zeta_p]})^\wedge_p$ is a perfect quasi-coherent sheaf on $X$. Properness of $X$ then implies that $$(\L_{X/\Z_p[\zeta_p]})^\wedge_p := \R\Gamma(X, (\L_{-/\Z_p[\zeta_p]})^\wedge_p)$$ is a perfect $\Z_p[\zeta_p]$-module. All exterior powers are then also perfect, i.e. $(\bigwedge^j_{X} \L_{X/\Z_p[\zeta_p]})^\wedge_p[-j]\{j\}$ is a perfect $\Z_p[\zeta_p]$-module. Since $\Z_p[\zeta_p]$ is a perfect $\Z_p\llb q-1\rrb$-module, it follows that $\gr_j^\conj \bar \prism_{X/\Z_p\llb q-1\rrb}$ is a perfect $\Z_p\llb q-1\rrb$-module. Since $\Fil_{-1}^\conj \bar \prism_{X/\Z_p\llb q-1\rrb} \simeq 0 $, it follows by a finite induction that $\Fil_i^\conj \bar\prism_{X/\Z_p\llb q-1\rrb}\{i\}$ is itself a perfect $\Z_p\llb q-1\rrb$-module, as required. 
\end{proof}

We end this subsection with the following remark on Tate twists. 
\begin{remark}\label{rmk::the_tate_twist_in_D_Iw_and_prismatic_log}
    The equivalence in \cref{thm::limit_THH_for_cyclotomic_tower} has an implicit factor of $\beta$, the Bott class, on the right hand side. This results in a filtration shift and Tate twist in the prismatic statement given in \cref{cor::limit_of_nyg_filt_prismatic_coh_over_cyc_tower}. Up to a multiple in $\Z_p^\times$, the class in prismatic cohomology responsible for this filtration shift and Tate twist is the element $\log_\prism(q^p) \in \Fil^1_\Nyg \prism_{\Z_p[\zeta_p]/\Z_p\llb q-1\rrb}^{(1)} \{1\}$ (see \cite[Section 2]{BL_absolute_prismatic_coh} for details on this element). 
    
    Indeed, under the isomorphism \[\Z_p\llb q_1-1\rrb \{1\} \cong \pi_2\TP(\Z_p[\zeta_p]/\S_p\llb q_1-1\rrb) \simeq \pi_2\ku_p^{(-1), tS^1} \cong \Z_p\llb q_1-1\rrb t^{-1},\]
    a generator of the Breuil--Kisin twist goes to the class $t^{-1}$. On the other hand, by \cite[Proposition 2.6.1]{BL_absolute_prismatic_coh}, a generator of the Breuil--Kisin twist is \[\frac{1}{q-1} \log_\prism(q^p)\]
    where $\log_\prism$ is the prismatic logarithm. Thus, up to a unit $\alpha\in \Z_p\llb q_1-1\rrb^\times$, we have $$t^{-1} = \alpha \cdot \frac{1}{q-1} \log_\prism(q^p)$$
    and so $\beta = (q-1)t^{-1} = \alpha \cdot \log_\prism(q^p)$ (c.f. \cite[Construction 2.7]{bhatt_mathew_syntomic_and_tate_twists}).
    Since the previous isomorphisms were $\Z_p^\times$-equivariant, the above equality of elements in $\Z_p\llb q-1\rrb\{1\}$ must also be $\Z_p^\times$-equivariant. We claim that $\Z_p^\times$ acts by multiplication on $\log_\prism(q^p)$: by definition of the prismatic logarithm (\cite[Section 2.2]{BL_absolute_prismatic_coh}) it suffices to check that $$q^{\gamma p^{i}}- 1 \equiv \gamma (q^{p^{i}}-1) \modc{[p^{i}]_q^2},$$
    which is an easy calculation. Hence, $\Z_p^\times$ acts trivially on $\alpha$, so that $\alpha \in \Z_p^\times$. 
    In fact, this unit $\alpha$ should be 1, though we will not need it here. A related discussion is given in \cite[Section 2.1]{bhatt_mathew_syntomic_and_tate_twists}. 

    Since Frobenius preserves the element $\log_\prism(q^p)$ \cite[Proposition 2.5.18]{BL_absolute_prismatic_coh}, we may safely suppress this factor.
\end{remark}

\subsection{de Rham Cohomology}\label{sctn::de_rham_coh}
In this section we prove \cref{maincor::limit_of_de_rham_coh} in \cref{cor::mainthm_for_de_rham_coh} below. Forgetting motivic filtrations, one just needs to apply $-\otimes_{\THH(\Z_p)}\Z_p$ to \cref{thm::limit_THH_for_cyclotomic_tower} since $\HH \simeq \THH \otimes_{\THH(\Z_p)}\Z_p$. However, one needs to keep track of motivic filtrations, which is made difficult by the fact that none of $\THH(\Z_p)\to \Z_p$, $\THH(-)\to \HH(-)$, or $\THH(-/\S_p\llb q_n-1\rrb) \to \HH(-/\Z_p\llb q_n-1\rrb)$ are $p$-completely eff. Thus, we retrace our steps. 

\begin{lemma}\label{lem::HH(rel_qdR)_is_THH(rel_qdR)_tensored_with_Zp}
    There are equivalences of sheaves of $S^1$-equivariant spectra on $\QSyn_{\Z_p[\zeta_{p^n}]}$: 
    \begin{align*}
        \HH(-/\Z_p\llb q_n-1\rrb) &\simeq \THH(-/\S_p\llb q_n-1\rrb)\otimes_{\THH(\Z_p)}\Z_p^\triv\\
        &\simeq \HH(-) \otimes_{\THH(\Z_p[\zeta_{p^n}])} \THH(\Z_p[\zeta_{p^n}]/\S_p\llb q_n-1\rrb)
    \end{align*}
    where the right hand side denotes the point-wise tensor product.
\end{lemma}
\begin{proof}
    It suffices to show that the equivalence holds at the level of pre-sheaves. We then compute:
    \begin{align*}
        \THH(-/\S_p\llb q_n-1\rrb) \otimes_{\THH(\Z_p)}\Z_p^\triv &\simeq \S_p\llb q_n-1\rrb^\triv \otimes_{\THH(\S_p\llb q_n-1\rrb)} \THH(-)\otimes_{\THH(\Z_p)}\Z_p^\triv \\
        &\simeq \S_p\llb q_n-1\rrb^\triv \otimes_{\THH(\S_p\llb q_n-1\rrb)} \HH(-) \\
        &\simeq \Z_p\llb q_n-1\rrb^\triv \otimes_{\THH(\S_p\llb q_n-1\rrb) \otimes \Z_p} \HH(-) \\
        &\simeq \HH(-/\Z_p\llb q_n-1\rrb).
    \end{align*}
    This is the first equivalence. The second equivalence follows since \[\THH(\Z_p[\zeta_{p^n}]/\S_p\llb q_n-1\rrb) \simeq \THH(\Z_p[\zeta_{p^n}]) \otimes_{\THH(\S_p\llb q_n-1\rrb)} \S_p\llb q_n-1\rrb^\triv. \]
\end{proof}
\begin{corollary}\label{cor::HH_and_rel_HH_even_for_QRSP_R}
    For $R$ a QRSP, $\HH(R)$ is even. If $R$ is furthermore a $\Z_p[\zeta_{p^n}]$-algebra, then $\HH(R/\Z_p\llb q_n-1\rrb)$ is also even. The motivic filtrations are moreover the double-speed Postnikov filtration.
\end{corollary}
\begin{proof}
    The first is \cite[Lemma 5.14]{bms2}. The second follows from the lemma using that $\THH(\Z_p[\zeta_{p^n}])\to  \THH(\Z_p[\zeta_{p^n}]/\S_p\llb q_n-1\rrb)$ is $p$-completely eff by \cref{lem::shj_n_to_shku_n_is_eff}. The final claim is immediate from evenness and the fact that the motivic filtration is the even filtration.  
\end{proof}
The following is a version of \cref{prop::prismatic_coh_from_q_de_rham_coh}, and is proved similarly. 
\begin{lemma}\label{lem::de_rham_coh_from_rel_dR}
    For any $n\ge 1$ and any choice of topological generator $\gamma_n\in 1+p^n\Z_p$, there is a fibre sequence of sheaves (in fact, a fibre sequence at the level of pre-sheaves) of synthetic spectra with synthetic circle action on $\QSyn_{\Z_p[\zeta_{p^n}]}$ \[\Fil_\mot^{\ge *} \HH(-) \to \Fil_\mot^{\ge *} \HH(-/\S_p\llb q_n-1\rrb) \xrightarrow{\psi_{\circ\circ}^{\gamma_n}} \Fil_\mot^{\ge *-1}\HH(-/\S_p\llb q_n-1\rrb)(1)[2].\]
\end{lemma}
\begin{proof}
    On underlying sheaves of spectra, one can just apply $-\otimes_{\THH(\Z_p)}\Z_p$ to the underlying fibre sequence of \cref{prop::prismatic_coh_from_q_de_rham_coh}. For motivic filtrations, by quasi-syntomic descent to QRSP $\Z_p[\zeta_{p^n}]$-algebras $R$, we can just apply $\tau_{\ge 2*}$ to the fibre sequence on underlying $S^1$-spectra for $R$. 
\end{proof}
The following is proved in the exact same way as \cref{cor::identifying_mot_filt_on_thh(R_n)_for_R_qrsp}.
\begin{corollary}\label{cor::identifying_mot_filt_on_HH(R_n)_for_R_qrsp}
    For $R$ a QRSP $\Z_p[\zeta_{p}]$-algebra, one has an equivalence \[\Fil_\mot^{\ge *} \HH(R\otimes_{\Z_p[\zeta_p]}\Z_p[\zeta_{p^n}]) \simeq \tau_{\ge 2*-1} \HH(R\otimes_{\Z_p[\zeta_p]}\Z_p[\zeta_{p^n}]) .\]
\end{corollary}

\begin{proposition}\label{prop::trace_for_HH}
    Let $g$ be a topological generator of $(1+p\Z_p)$, and $R\in \QSyn_{\Z_p[\zeta_p]}$. Set $R_n := R\otimes_{\Z_p[\zeta_p]} \Z_p[\zeta_{p^n}]$. The transfer map $$\tr: \Fil_\mot^{\ge *} \HH(R_{n+1})\to \Fil_\mot^{\ge *} \HH(R_n)$$ sits in the following functorial-in-$R$ commutative diagram of $\Fil_\mot^{\ge *} \HH(R_n)$-modules in synthetic spectra with synthetic $S^1$-action:
    \[\begin{tikzcd}
        	{\Fil_\mot^{\ge *}\HH(R_{n+1})} & {\Fil_\mot^{\ge *}\HH(R_{n+1}/\S_p\llb q_{n+1}-1\rrb)} && {\Fil_\mot^{\ge *-1}\HH(R_{n+1}/\S_p\llb q_{n+1}-1\rrb)(1)[2]} \\
        	& {\Fil_\mot^{\ge *}\HH(R_{n}/\S_p\llb q_{n}-1\rrb)} && {\Fil_\mot^{\ge *-1}\HH(R_{n}/\S_p\llb q_{n}-1\rrb)(1)[2]} \\
        	{\Fil_\mot^{\ge *}\HH(R_n)} & {\Fil_\mot^{\ge *}\HH(R_{n}/\S_p\llb q_{n}-1\rrb)} && {\Fil_\mot^{\ge *-1}\HH(R_{n}/\S_p\llb q_{n}-1\rrb)(1)[2]}
        	\arrow[from=1-1, to=1-2]
        	\arrow["\tr"', from=1-1, to=3-1]
        	\arrow["{\psi_{\circ\circ}^{g^{p^n}}}", from=1-2, to=1-4]
        	\arrow["{\rho_n}"', from=1-2, to=2-2]
        	\arrow["{\rho_n(1)[2]}"', from=1-4, to=2-4]
        	\arrow["{\psi_{\circ\circ}^{g^{p^n}}}", from=2-2, to=2-4]
        	\arrow["{\sum_{i=0}^{p-1}\psi^{g^{ip^{n-1}}}}", from=2-2, to=3-2]
        	\arrow["\id", equals, from=2-4, to=3-4]
        	\arrow[from=3-1, to=3-2]
        	\arrow["{\psi_{\circ\circ}^{g^{p^{n-1}}}}", from=3-2, to=3-4]
        \end{tikzcd}\]
    where the top and bottom rows are the fibre sequences of \cref{lem::de_rham_coh_from_rel_dR}. 
\end{proposition}
\begin{proof}
    Forgetting filtrations, we can just apply $-\otimes_{\THH(\Z_p)}\Z_p$ to the underlying unfiltered diagram in \cref{thm::computing_trace_maps_explicitly_for_any_R}. For motivic filtrations, just as in the proof of \cref{thm::computing_trace_maps_explicitly_for_any_R}, we can take connective covers in the underlying unfiltered diagram for $R$ a QRSP, using \cref{cor::HH_and_rel_HH_even_for_QRSP_R} to identify $$\tau_{\ge 2*} \HH(R_n/\Z_p\llb q_n-1\rrb)\simeq \Fil_\mot^{\ge *} \HH(R_n/\Z_p\llb q_n-1\rrb)$$ as well as \cref{cor::identifying_mot_filt_on_HH(R_n)_for_R_qrsp} to identify \(\tau_{\ge 2*-1} \HH(R_n)\simeq \Fil_\mot^{\ge *} \HH(R_n).\)
\end{proof}

\begin{proposition}
    For any bounded qcqs quasi-syntomic $p$-adic formal scheme $X$ over $\Z_p[\zeta_p]$, write $$X_n := X\times_{\Spf\Z_p[\zeta_p]} \Spf \Z_p[\zeta_{p^n}] .$$ 
    There is a functorial-in-$X$ $\Z_p^\times$-equivariant equivalence of synthetic spectra with synthetic circle action:
    \begin{align*}
        \varprojlim_n \Fil_\mot^{\ge *}\HH(X_n) &\simeq \varprojlim_n\Fil_\mot^{\ge *-1}\HH(X_n/\Z_p\llb q_n-1\rrb)(1)[1]
    \end{align*}
    where the transition maps on the right hand side are given by 
    \begin{align*}
        \Fil_\mot^{\ge *-1}\HH(X_n/\Z_p\llb q_n-1\rrb) &\simeq \left(\Fil_\mot^{\ge *-1}\HH(X_1/\Z_p\llb q_1-1\rrb) \right) \otimes_{\Z_p\llb q_1-1\rrb} \Z_p\llb q_n-1\rrb \\
        &\xrightarrow{\id \otimes \proj} \left(\Fil_\mot^{\ge *-1}\HH(X_1/\Z_p\llb q_1-1\rrb) \right) \otimes_{\Z_p\llb q_1-1\rrb} \Z_p\llb q_{n-1}-1\rrb  \\
        &\simeq \Fil_\mot^{\ge *-1}\HH(X_{n-1}/\Z_p\llb q_{n-1}-1\rrb).
    \end{align*}
\end{proposition}
\begin{proof}
    As in the proof of \cref{thm::limit_THH_for_cyclotomic_tower}, as a consequence of \cref{prop::trace_for_HH} it suffices to show that
    \[\varprojlim_n \Fil_\mot^{\ge *} \THH(R_n/\S_p\llb q_n-1\rrb)\]
    is zero where the transition maps in the middle term are given by 
    \begin{align*}
        \Fil_\mot^{\ge *}\HH(R_{n+1}/\S_p\llb q_{n+1}-1\rrb)& \simeq \Fil_\mot^{\ge *}\HH(R_n/\S_p\llb q_n-1\rrb)  \otimes_{\S_p\llb p^{-n}\Z_p\rrb} \S_p\llb p^{-n-1}\Z_p\rrb \\
        &\xrightarrow{\id \otimes \mr{proj}} \Fil_\mot^{\ge *}\HH(R_n/\S_p\llb q_n-1\rrb) \\
        &\xrightarrow{\sum_{i=0}^{p-1} \psi^{g^{ip^{n-1}}}} \Fil_\mot^{\ge *}\HH(R_n/\S_p\llb q_n-1\rrb).
    \end{align*}
    Since $\varprojlim_n \Fil_\mot^{\ge *} \HH(R_n/\S_p\llb q_n-1\rrb)$ is a quasi-syntomic sheaf in $R$, it suffices by quasi-syntomic descent to show the vanishing when $R$ is a QRSP. Using \cref{cor::HH_and_rel_HH_even_for_QRSP_R}, we need to show that $$\varprojlim_n \tau_{\ge 2*} \HH(R_n/\S_p\llb q_n-1\rrb)$$
    vanishes. By the Milnor sequence, it suffices to show that $\varprojlim_n \HH(R_n/\S_p\llb q_n-1\rrb) = 0$. This follows from \[\varprojlim_n \THH(R_n/\S_p\llb q_n-1\rrb) \simeq 0\]
    established in the proof of \cref{thm::limit_THH_for_cyclotomic_tower} by applying $-\otimes_{\THH(\Z_p)}\Z_p$, noting that this commutes with the limit as a consequence of \cref{lem::almost_perfect_commute_with_limit_for_connective_sp}. 
\end{proof}

The same argument as in \cref{cor::limit_of_nyg_filt_prismatic_coh_over_cyc_tower} then yields the following.
\begin{corollary}\label{cor::mainthm_for_de_rham_coh}
    Suppose $X$ is a qcqs bounded $p$-adic formal scheme over $\Z_p[\zeta_p]$ such that $\L_{X/\Z_p\llb q_1-1\rrb}$ is pseudo-coherent as a $\Z_p\llb q_1-1\rrb$-module, then there are equivalences 
    \begin{align*}
        \varprojlim_n \hat \dR^{\ge r}_{X_n} &\simeq \hat \dR_{X/\Z_p\llb q_1-1\rrb}^{\ge r-1} \hat\otimes_{\Z_p\llb q_1-1\rrb} \varprojlim_n \Z_p\llb q_n-1\rrb (1)[-1],\\
        \varprojlim_n \hat \dR_{X_n} &\simeq \hat \dR_{X/\Z_p\llb q_1-1\rrb} \hat\otimes_{\Z_p\llb q_1-1\rrb} \varprojlim_n \Z_p\llb q_n-1\rrb (1)[-1],\\
        \varprojlim_n \hat \dR^{\le r}_{X_n} &\simeq \hat \dR_{X/\Z_p\llb q_1-1\rrb}^{\le r-1} \hat\otimes_{\Z_p\llb q_1-1\rrb} \varprojlim_n \Z_p\llb q_n-1\rrb (1)[-1],
    \end{align*}
    where $X_n:= X\times_{\Spf\Z_p[\zeta_p]} \Spf \Z_p[\zeta_{p^n}]$, and the right hand side of all three equivalences are $(p,q_1-1)$-adically completed.
\end{corollary}
\begin{remark}
    The usual transitivity triangle for the cotangent complex, and the standard computation $\L_{\Z_p[\zeta_p]/\Z_p\llb q_1-1\rrb} \simeq \Z_p[\zeta_p]\{1\}[1]$ yield a fibre sequence \[\R\Gamma(X, \O_X)[1] \to \L_{X/\Z_p\llb q_1-1\rrb} \to \L_{X/\Z_p[\zeta_p]}.\]
    Here, $\R\Gamma(X,\O_X)$ is the coherent cohomology of the structure sheaf of $X$. Since $\Z_p[\zeta_p]$ is itself a pseudo-coherent $\Z_p\llb q_1-1\rrb$-algebra, \cite[Tag064Z]{stacks-project} implies that $\L_{X/\Z_p\llb q_1-1\rrb}$ is a pseudo-coherent $\Z_p\llb q_1-1\rrb$-module if $\R\Gamma(X,\O_X)$ and $\L_{X/\Z_p[\zeta_p]}$ are pseudo-coherent $\Z_p[\zeta_p]$-modules.
    
    If $X$ is proper, then $\R\Gamma(X,\O_X)$ is a perfect $\Z_p[\zeta_p]$-module. We also saw in the proof of \cref{lem::sufficient_condition_for_almost_perf_gr_Nyg} that $\L_{X/\Z_p[\zeta_p]}$ is a perfect quasi-coherent sheaf whenever $X$ is proper regular over $\Z_p[\zeta_p]$. Thus, \cref{cor::mainthm_for_de_rham_coh} holds for proper regular $X$ over $\Z_p[\zeta_p]$. 
\end{remark}

\printbibliography

\end{document}